\documentclass[10pt,a4paper,leqno]{amsart}

\usepackage{amsmath,amssymb,dsfont,bm, mathtools, mathabx, blkarray}

\numberwithin{equation}{section}

\newcommand{\ve}[1]{{\bf #1}}
\newcommand{\Rz}{\mathcal{H}}
\newcommand{\aleq}{\lesssim}
\newcommand{\brac}[1]{\left (#1 \right )}
\newcommand{\lapms}[1]{I_{#1}}
\newcommand{\laps}[1]{\Dx^{#1}}
\newcommand{\lapv}{\Dx^{1/2}}
\renewcommand{\S}{\mathbb{S}}

\newcommand{\ff}{\mathbf{f}}

\newcommand{\R}{\mathbb{R}}
\newcommand{\C}{\mathbb{C}}
\newcommand{\N}{\mathbb{N}}

\newcommand{\Z}{\mathbb{Z}}
\newcommand{\Ss}{\mathbb{S}}

\newcommand{\im}{\mathrm{i}}

\newcommand{\be}{\begin{equation}}
\newcommand{\ee}{\end{equation}}

\newcommand{\vphi}{\varphi}
\newcommand{\tvphi}{\widetilde{\varphi}}
\renewcommand{\rho}{\varrho}
\newcommand{\eps}{\varepsilon}

\newcommand{\xB}{\mathbf{X}}
\newcommand{\uu}{\mathbf{u}}
\newcommand{\vv}{\mathbf{v}}

\newcommand{\QB}{\mathbf{Q}}
\newcommand{\QBm}{\mathbf{Q}_m}
\newcommand{\hh}{\mathbf{h}}
\newcommand{\eB}{\mathbf{e}}

\newcommand{\Dx}{|\nabla|}

\renewcommand{\theta}{\vartheta}

\newcommand{\pt}{\partial}

\newcommand{\DD}{\Delta}

\newcommand{\HH}{\mathcal{H}}

\newcommand{\tL}{\widetilde{L}}
\newtheorem{thm}{Theorem}
\newtheorem{thm2}{Theorem}[section]
\newtheorem{lemma}{Lemma}[section]
\newtheorem{prop}{Proposition}[section]

\newtheorem{cor}{Corollary}[section]

\newtheorem*{remark*}{Remark}
\newtheorem*{remarks*}{Remarks}

\catcode`@=11
\def\section{\@startsection{section}{1}%
  \z@{1.5\linespacing\@plus\linespacing}{.5\linespacing}%
  {\normalfont\bfseries\large\centering}}
\catcode`@=12

\begin{document}

\title[On Energy-Critical Half-Wave Maps into $\Ss^2$]{On Energy-Critical Half-Wave Maps into $\Ss^2$}

\begin{abstract}
We consider the energy-critical half-wave maps equation 
$$
\pt_t \uu + \uu \wedge |\nabla| \uu = 0
$$
for $\uu : [0,T) \times \R \to \Ss^2$. We give a complete classification of all traveling solitary waves with finite energy. The proof is based on a geometric characterization of these solutions as minimal surfaces with (not necessarily free) boundary on $\Ss^2$. In particular, we discover an explicit Lorentz boost symmetry, which is implemented by the conformal M\"obius group on the target $\Ss^2$ applied to half-harmonic maps from $\R$ to $\Ss^2$.

Complementing our classification result, we carry out a detailed analysis of the linearized operator $L$ around half-harmonic maps $\QB$ with arbitrary degree $m \geq 1$. Here we explicitly determine the nullspace including the zero-energy resonances; in particular, we  prove the nondegeneracy of  $\QB$. Moreover, we give a full description of the spectrum of $L$ by finding all its $L^2$-eigenvalues and proving their simplicity. Furthermore, we prove a coercivity estimate for $L$ and we rule out embedded eigenvalues inside the essential spectrum. Our spectral analysis is based on a reformulation in terms of certain Jacobi operators (tridiagonal infinite matrices) obtained from a conformal transformation of the spectral problem posed on $\R$ to the unit circle $\Ss$. Finally, we construct a unitary map which can be seen as a gauge transform tailored for a future stability and blowup analysis close to half-harmonic maps. Our spectral results also have potential applications to the half-harmonic map heat flow, which is the parabolic counterpart of the half-wave maps equation.
\end{abstract}

\author{Enno Lenzmann}
\address{University of Basel, Department of Mathematics and Computer Science, Spiegelgasse 1, CH-4051 Basel, Switzerland.}
\email{enno.lenzmann@unibas.ch}

\author{Armin Schikorra}
\address{Mathematisches Institut, Abt. f\"ur Reine Mathematik, Albert-Ludwigs-Universit\"at
Freiburg, Eckerstrasse 1, 79104 Freiburg im Breisgau, Germany}
\email{armin.schikorra@math.unifreiburg.de}
\maketitle

\tableofcontents

\section{Introduction}
In the present paper,  we study the {\bf half-wave maps equation} that is given by
\be \tag{H-WM} \label{eq:hwm}
\pt_t \uu + \uu \wedge \Dx \uu=0
\ee
for the unknown function $\uu : [0,T) \times \R \to \Ss^2$. Here $\Ss^2$ is the two-dimensional unit-sphere embedded in $\R^3$ and the symbol $\wedge$ denotes the cross product in $\R^3$. The operator $|\nabla| = \sqrt{-\DD}$ stands for the square root of the Laplacian on  $\R$. Below we provide more details on the precise functional framework. For some background on the physical motivation and the occurrence of \eqref{eq:hwm} as a universal continuum limit of completely integrable spin systems of {\em Calogero--Moser type}, we refer the reader to Appendix \ref{sec:primer}.

The main purpose of this paper is twofold: First, we completely classify all traveling solitary waves with finite energy. Here we use a characterization of these solutions as minimal surfaces satisfying certain boundary conditions on $\Ss^2$, thereby generalizing the notion of free boundary minimal disks that have already appeared in the study of half-harmonic maps (see below). As a surprising fact, we will discover an explicit Lorentz boost symmetry in the problem realized by the conformal M\"obius group acting on the target $\Ss^2$. Secondly, we provide a complete spectral analysis of the linearized operator around static traveling solitary waves, which correspond to half-harmonic maps with arbitrary degree $m \geq 1$. In fact, these spectral results will lay the groundwork for any future stability and blowup analysis for solutions of \eqref{eq:hwm} as well as its parabolic counterpart given by the energy-critical {\em half-harmonic map heat flow}.

Let us first collect some general facts about the problem under consideration. The evolution equation \eqref{eq:hwm} has a Hamiltonian structure with the corresponding conserved energy given by
\be
E[\uu] = \frac{1}{2} \int_{\R} \uu \cdot |\nabla| \uu \, dx.
\ee
As a consequence, the homogenouse Sobolev space $\dot{H}^{\frac 1 2}(\R; \Ss^2)$ is the natural energy space for the Cauchy problem of \eqref{eq:hwm}. Furthermore, it is easy to see that the rescaling $\uu(t,x) \mapsto \uu(\lambda t, \lambda x)$ with  constant $\lambda > 0$ maps solutions into solutions. Since the energy $E[\uu] = E[\uu(\lambda \cdot)]$ remains unaffected by such a change of scales, the half-wave maps equation \eqref{eq:hwm} happens to be {\em energy-critical}. Therefore, delicate phenomena such as blowup (singularity formation) are conceivable scenarios for solutions of the half-wave maps equation. In addition to the feature of energy-criticality, we shall also see that the {\em conformal symmetry} of the energy functional $E[\uu]$ together with a (hidden) Lorentz boost symmetry (implemented by the conformal group of $\Ss^2$) will both play an important role in the complete classification and spectral properties of traveling solitary waves.

Let us briefly put  \eqref{eq:hwm} into the context of other geometric evolution equations. On the one hand, it is not surprising that the half-wave maps equation shows strong similarities to well-established models in the field of dispersive geometric PDEs such as the energy-critical {\em wave-maps equation (WM)} and {\em Schr\"odinger maps equation (SM)} with target $\Ss^2$. For results on singularity formation (blowup) for these equations, see \cite{KrScTa-08, MeRaRo-13, RoSt-10, GuNaTs-10, Sh-88, BeTa-14, Pe-14}. However, as opposed to (HM) and (SM), we mention that satisfactory well-posedness results for \eqref{eq:hwm} pose an interesting open problem. For small data global well-posedness in the energy-supercritical case in sufficiently high space dimensions, see the recent results in \cite{KrSi-16}.

On the other hand, our results below will show that \eqref{eq:hwm} possesses a list of noteworthy features that are in  striking contrast to both (WM) and (SM). For instance, the set of traveling solitary waves turns out to be very rich involving closed analytic expression in terms of Blaschke products from complex analysis. Furthermore, the existence of traveling solitary waves with arbitrarily small energy (as shown below) provides a decisive difference to both the energy-critical (WM) and (SM). As a consequence, the energy-critical half-wave maps equation has no minimal energy threshold for creating traveling solitary waves (i.\,e.~non-trivial bubbles of energy). This fact implies that small energy solutions in general cannot scatter to a free wave. We remark that such a phenomenon was recently exploited in \cite{GeLePoRa-16} for an explicit construction of turbulent solutions for the $L^2$-critical scalar half-wave equation in one spatial dimension. One may thus speculate whether such a turbulence mechanism also exists for \eqref{eq:hwm} in a suitable regime of small energy solutions. 

With regard to the physical motivation, we mention that the half-wave maps equation \eqref{eq:hwm} formally arises as a universal continuum limit for the dynamics of long-range lattice spin system of {\em Calogero--Moser type}, which appear in the theory of completely integrable systems; see Appendix \ref{sec:primer} for more details. Of course, it is a natural  question whether some sort of complete integrability still holds for the Hamiltonian evolution equation formulated by \eqref{eq:hwm}. The remarkable degree of explicitness in the classification results below together with the intriguing properties of the linearized problem both indicate that some ``completely integrable structure'' may be at work here. Yet, even if complete integrability in a certain sense was present,  it would nevertheless be conceivable that the flow also admits some kind of smooth blowup solutions, e.\,g., becoming singular in $\dot{H}^s$ with some $s > \frac 1 2$, as the set of conserved quantities may not be able to control regularity above the energy level. In fact, such a phenomenon would be highly interesting from the analytic point of view.

\section{Classification of Traveling Solitary Waves}
Let us consider {\bf traveling solitary waves}  for the energy-critical half-wave maps equation with finite energy. That is, we look for solutions of the form
\be
\uu(t,x) = \QB_v (x- vt),
\ee
where the parameter $v \in \R$ denotes the traveling velocity and $\QB_v \in \dot{H}^{\frac 1 2}( \R; \Ss^2)$ is some profile with finite energy. If we plug this ansatz into \eqref{eq:hwm}, we readily check that $\QB_v: \R \to \Ss^2$ has to be a (weak) solution of the equation
\be \label{eq:uv}
\QB_v \wedge |\nabla| \QB_v - v \pt_x \QB_v = 0  .
\ee 
On a formal level,  this equation arises from a variational principle; see the remarks made below for more details. However, we will not pursue  this approach but instead we will work directly with the equation itself without resorting to any variational structure. 

In fact, our first two main results in Theorem \ref{thm:vsmall} and \ref{thm:vlarge} below will give a complete and explicit classification of all solutions $\QB_v \in \dot{H}^{\frac 1 2}(\R; \Ss^2)$ for $|v| < 1$ and the triviality of $\QB_v$ when $|v| \geq 1$, respectively. In particular, our classification result reveals a surprising symmetry due to the proper Lorentz group $SO^+(3,1)$ of the problem at hand, which is expressed through the action of the isomorphic M\"obius group $PSL(2, \C)\simeq SO^{+}(3,1)$, which is the (orientation preserving) conformal group on the target sphere $\Ss^2$. We find the presence of such an exact symmetry quite remarkable in view of the fact that the time-dependent equation \eqref{eq:hwm} itself does not have a relativistic Lorentz invariance. In addition to this, we will see that the energy $E[\QB_v]$ is {\em decreasing} under Lorentz boosts, providing a mechansim for producing solitary waves (and hence non-scattering) solutions of arbitrary small energy. This feature is in striking to contrast to the energy-critical wave-maps equation with target $\Ss^2$, say, where Lorentz boosts are known to increase the energy of solutions and small data in energy norm is expected to always lead to scattering of solutions.

To start the analysis of \eqref{eq:uv}, it is obviously of use to first consider the special case with vanishing velocity $v = 0$, which corresponds to static solutions for \eqref{eq:hwm}. Then the profile equation \eqref{eq:uv} reduces to the so-called {\bf half-harmonic maps equation} for $\QB \in \dot{H}^{\frac 1 2}(\R;  \Ss^2)$  which is given by
\be \label{eq:hh}
\QB \wedge |\nabla| \QB = 0.
\ee 
This nonlinear equation has recently attracted substantial attention in the field of geometric PDEs; see the seminal works of Da Lio and Rivi\`ere \cite{DR1dMan,DaLioBubbles} on the regularity theory, see also \cite{DSpalphSphere, Millot-Sire-2015, Scheven-2006, Schikorra-epsilon}. An essential feature of the half-harmonic maps equation is that its (finite-energy) solutions can be identified via a harmonic extension as {\em free boundary minimal disks} inside the unit ball $B \subset \R^3$, i.\,e., minimal disks $\Sigma \subset B$ with boundary that intersect the boundary of $B$ orthogonally, which we write as $\pt \Sigma \perp \Ss^2 = \pt B$ (see \cite{DaMaRi-15, DSpalphSphere, Millot-Sire-2015, DaLio-16, StruweIM84}). Moreover, by a classical result of Nitsche \cite{Nitsche70} any such free boundary minimal disk $\Sigma \subset B$ is a flat disk with $\pt \Sigma$ being a great circle on $\Ss^2$ (for higher dimensions, see also the recent generalization of Nitsche's result by Fraser and Schoen \cite{FrSc-15}). Thus, after a suitable rotation, we conclude that any finite-energy half-harmonic map can be brought into the form $\QB : \R \to \Ss^1 \times \{0 \} \subset \Ss^2$, i.\,e., its image lies in the equatorial plane. Once we are in this setting, we can apply methods of complex analysis and some Hardy space theory to deduce that all such half-harmonic maps (with finite energy) from $\R$ to $\Ss^1$ are explicitly given by suitable Blaschke products, as recently proven by Millot and Sire \cite{Millot-Sire-2015}; see also \cite{BeMiRySa-14} for a related classification result in Ginzburg--Landau theory. 

Let us now come back to the general case with arbitrary velocity $v \in \R$. A moment's reflection shows that if $v \neq 0$ holds, the identification of solutions for \eqref{eq:uv} as free boundary minimal disks inside $B$ fails due to the fact that $\pt \Sigma \not \perp \Ss^2$ in general, where $\Sigma \subset B$ denotes the minimal disk given by the harmonic extension of $\QB_v$. Thus to tackle the case for general velocities $v$, it will turn out that we need a broader notion of not necessarily free boundary minimal disks. By carrying out such an analysis, we can in fact prove the following complete classification result for $|v| < 1$; the case $|v| \geq 1$ will be addressed in Theorem \ref{thm:vlarge} below.

\begin{thm}[Complete Classification of Traveling Solitary Waves for $|v| < 1$] \label{thm:vsmall}
Let $v \in \R$ with $|v| < 1$ be given. Then $\QB_v \in \dot{H}^{\frac 1 2}(\R; \Ss^2)$ solves equation \eqref{eq:uv} if and only if
$$
\QB_v(x) = R \left ( \sqrt{1- v^2} f(x), \sqrt{1-v^2} g(x), s \cdot v  \right )
$$
with some rotation $R \in SO(3)$. Here $s \in \{\pm 1\}$ is a sign factor and the functions $f,g : \R \to \R$ are given by
$$
f(x) =  \mathrm{Re} \, \mathcal{B}(x+\im 0) \quad \mbox{and}  \quad  g(x) = s \cdot \mathrm{Im} \, \mathcal{B}(x + \im 0) ,
$$
where $\mathcal{B} : \overline{\mathbb{C}}_+ \to \C$ is a {\bf finite Blaschke product} of the form
$$
\mathcal{B}(z) =  e^{\im \vartheta} \prod_{k=1}^m \frac{ \lambda_k(z-a_k) - \im }{\lambda_k (z - a_k) + \im } 
$$
with some integer $m \in \N$ and real parameters $\vartheta \in \R$, $\lambda_1, \ldots, \lambda_m  \in \R_{>0}$, $a_1, \ldots, a_m \in \R$.
\end{thm}

\begin{remarks*}
{\em
1) The integer number $m \in \N$ is the {\bf degree} of the finite Blaschke product $\mathcal{B}(z)$. Note that the case $m=0$ corresponds to the trivial case of constant profiles $\QB_v \equiv \mbox{const}$. The choice of $s=+1$ in Theorem \ref{thm:vsmall} corresponds to $\QB_v$ obtained as boundary values of {\bf holomorphic} functions, whereas $s=-1$ corresponds to the {\bf anti-holomorphic} case. The anti-holomorphic case could also be encoded by introducing a negative degree $m \in - \N$. However, we shall not use the nomenclature here. Furthermore, without loss of generality, we adapt the convention that we mean by default that $\QB_v$ is given by a holomorphic function, i.\,e., we choose $s=+1$, unless we say something else explicitly.

2) The energy of $\QB_v \in \dot{H}^{\frac 1 2}(\R; \Ss^2)$ with  $|v| < 1$ and degree $m \in \N$ is found to be
$$
E[\QB_v] = (1 - v^2) \cdot \pi m.
$$
As a direct consequence, the energy-critical half-wave maps equation \eqref{eq:hwm} possesses traveling solitary waves with arbitrarily small energy $E[\QB_v] \to 0$ as $|v| \to 1^-$. In particular, initial data with small energy may not asymptotically scatter to 0 as $|t| \to \infty$ in contrast to the energy-critical wave-maps  and Schr\"odinger maps equation with target $\Ss^2$; see \cite{StTa-10a, StTa-10b} and \cite{BeIoKeTa-13}, respectively.

3) Note that we could absorb the phase factor $e^{\im \vartheta}$ in $\mathcal{B}(z)$ into the rotation $R \in SO(3)$, but we prefer not to do so.

4) Let $\QB=(f,g,0) \in \dot{H}^{\frac 1 2}(\R; \Ss^2)$ be a given half-harmonic map in the equatorial plane. For $0 < |v| < 1$, the transformation
$$
\QB(x) = (f(x), g(x), 0) \mapsto \QB_v(x) = \left ( \sqrt{1-v^2} f(x), \sqrt{1-v^2} g(x), v \right )
$$ 
can be viewed as {\bf Lorentz boost}  with velocity $v$ in $z$-direction, implemented by action of the M\"obius group on the sphere $\Ss^2$; see Appendix \ref{sec:primer} for details. Since all finite energy solutions of \eqref{eq:uv} with $|v| <  1$ are obtained via the boost mechanism (and rotations) by Theorem \ref{thm:vsmall}, we can refer to $\QB_v$ as {\bf boosted} (or {\bf squeezed}) {\bf half-harmonic maps} from $\R$ to $\Ss^2$.  Note that $\QB_v$ traces out a circle on $\Ss^2$ with radius  $r=\sqrt{1-v^2}<1$, which can be seen as the phenomenon of {\em Lorentz contraction} known in special relativity. 

5) For a given degree $m \in \N$, the family of finite Blaschke products $\mathcal{B}(z)$ above is determined the $2m + 1$ real parameters $\{\vartheta, \lambda_1, \ldots, \lambda_m, a_1, \ldots, a_m \}$. As a consequence, the nullspace of the linearized operator $L$ around $\QB_v$ will contain zero modes generated by differentiation with respect to these parameters. In Theorem \ref{thm:null} below, we will explicitly determine the nullspace of $L$ and rule out any additional zero modes that are not due to symmetries. Thus we will show the {\bf nondegeneracy} of the linearized operator.

6) As a curious aside, we mention that equation \eqref{eq:uv} also has solutions with {\bf infinite energy} when $|v| < 1$. For instance, the function $\QB_v : \R \to \Ss^2$ with
$$
\QB_v(x) = \left ( \sqrt{1-v^2} \cos x, \sqrt{1-v^2} \sin x , v \right ) \in \dot{H}^{\frac 1 2}_{\mathrm{loc}}(\R; \Ss^2)
$$
is easily seen to solve \eqref{eq:uv} by using that $|\nabla| \cos x = \cos x$ and $|\nabla| \sin x = \sin x$. In particular, the function $\QB(x) = ( \cos x, \sin x, 0) \in \dot{H}^{\frac 1 2}_{\mathrm{loc}}(\R; \Ss^2)$ is an infinite energy solution of the half-harmonic maps equation \eqref{eq:hh}. It seems that the existence of these infinite-energy solutions has gone unnoticed so far.

7) As mentioned above, the boosted half-harmonic maps $\QB_v \in \dot{H}^{\frac 1 2}(\R; \Ss^2)$ can be {\em formally} seen as {\bf critical points} of the (conformally invariant) functional
\be
S[\uu] = E [\uu] + v P[\uu],
\ee
where $v \in \R$ enters as a parameter and $P[\uu]$ is a functional corresponding to the linear momentum; see Appendix \ref{sec:primer} for more details. In geometric terms, the energy $E[\uu]$ corresponds to the area of the surface $\Sigma$ spanned by the harmonic extension $\uu^e : \R_+^2 \to \R$ with boundary $\partial \Sigma$ given by the closed curve $\uu : \R \to \Ss^2$; whereas the momentum functional $P[\uu]$ yields the {\em solid angle} traced out by $\Sigma$ subtended to a given point $\eB \in \Ss^2$, which can be chosen in the definition of $P[\uu]$. Of course, there is an inherent ambiguity of $4 \pi$ in the definition of $P[\uu]$ due to possible different choices for the point $\eB$. 
}

\end{remarks*}

We conclude our classification analysis of solutions for  \eqref{eq:uv} by showing that traveling speeds $|v| \geq 1$ imply the triviality of the profile $\QB_v$.

\begin{thm}[Triviality of Traveling Solitary Waves for $|v| \geq 1$] \label{thm:vlarge}
Let $v \in \R$ with $|v| \geq 1$ be given. Then $\QB_v \in \dot{H}^{\frac 1 2}(\R; \Ss^2)$ solves \eqref{eq:uv} if and only if $$\QB_v(x) \equiv \mathbf{p}$$ with some constant $\mathbf{p} \in \Ss^2$.
\end{thm}

\subsection{Comments on the Proofs}
We close this section by briefly sketching the main ideas behind the proofs of Theorems \ref{thm:vsmall} and \ref{thm:vlarge}. 

We start with some remarks on how to establish Theorem \ref{thm:vsmall}. Let us assume that $\QB_v \in \dot{H}^{\frac 1 2}(\R; \Ss^2)$ is a non-constant solution of the profile equation \eqref{eq:uv} with some given $|v| < 1$, where we shall write $\QB = \QB_v$ for simplicity. We define $\QB^e : \R_+^2 \to \R^3$ to be the harmonic extension of $\QB$ to the upper halfplane $\R_+^2$ by Poisson's formula. This leads us to study the equivalent boundary-value problem that reads
$$
\left \{ \begin{array}{ll} \DD \QB^e = 0 & \quad \mbox{in $\R_+^2$}, \\
v \pt_x \QB^e  + \QB^e \wedge \pt_y \QB^e = 0 & \quad \mbox{on $\pt \R_+^2 = \R \times \{y=0 \}$},
\end{array} \right . 
$$ 
As in the special half-harmonic case when $v=0$ (see \cite{DaLioBubbles, DaLioRi-16, Millot-Sire-2015}), the adaption of a well-established argument involving a Hopf-type differential shows that $\QB^e : \R^2_+ \to \R^3$ is a (weakly) {\em conformal map}, i.\,e., it holds that  $|\pt_x \QB^e| = |\pt_y \QB^e|$ and  $\pt_x \QB^e \cdot \pt_y \QB^e = 0$. Using complex notation, we recall that the conformality conditions can be compactly written as
$$
\pt_z \QB^e \cdot \pt_z \QB^e = 0
$$
with $\pt_z = \frac{1}{2} ( \pt_x - \im \pt_y)$ and $\mathbf{v} \cdot \mathbf{w} =\sum_{k=1}^3 v_k w_k$ for $\mathbf{v}, \mathbf{w} \in \C^3$. It is a well-known fact that the harmonicity of $\QB^e$ together with its conformality implies that $(x,y) \mapsto \QB^e(x,y)$ traces out a (possibly branched) minimal surface $\Sigma \subset \R^3$. In fact, since $|\QB^e| < 1$ on $\R^2_+$ by the strong maximum principle, we have that $\Sigma \subset B$ (the unit ball in $\R^3$).
 
However, a mandatory -- and quite technical -- prerequisite for carrying out the aforementioned steps consists in first proving {\em higher regularity} for the map $\QB$. In Appendix \ref{sec:regularity} below, we provide a careful extension of the regularity proof by Da Lio and Rivi\`ere \cite{DR1dMan} (developed for half-harmonic maps) to the general case $|v| < 1$. We emphasize that  the condition that $|v|$ be strictly less than one is crucial to close the arguments, enabling us to effectively treat the boost term $v \pt_x \QB$ as a perturbation with respect to the main term involving $|\nabla| \QB$ in the profile equation \eqref{eq:uv}.  

Now, with the help of the conformality of the harmonic extension $\QB^e$, we can reformulate equation \eqref{eq:uv} in the following alternative form:
$$
|\nabla | \QB = \sqrt{1-v^2} |\pt_x \QB | \QB - v \pt_x \QB \wedge \QB .
$$
As an interesting aside, we notice that this version already contains some nontrivial geometric information. Indeed, by taking the scalar product with $\QB(x) \in \Ss^2$ of this equation, followed by integrating and using the properties of the harmonic extension, we deduce
$$
2 A[\Sigma] = \sqrt{1-v^2} L[\pt \Sigma] .
$$
Here the two functionals
$$
\left \{ \begin{array}{l}
\displaystyle A[\Sigma] = \frac 1 2 \int \! \! \int_{\R^2_+} |\nabla \QB^e|^2 \, dx \, dy = \frac{1}{2} \int_{\R} \QB \cdot |\nabla| \QB \,d x = E[\QB], \\[1ex]
\displaystyle L[\pt \Sigma] = \int_{\R} |\pt_x \QB| \, dx, \end{array} \right .
$$
denote the area of the (possibly multi-covered) surface $\Sigma$ parametrized by the map $\QB^e :Ê\R_+^2 \to \R^3$ and  the length of its boundary $\pt \Sigma$ parametrized by the curve $\QB : \R \to \Ss^2$, respectively. Combined with the general {\em isoperimetric inequality} $L[\pt \Sigma]^2 \geq 4 \pi A[\Sigma]$, we thus obtain the lower bound for the energy of $\QB$ such that 
$$
E[\QB] = A[\Sigma] \geq (1-v^2) \cdot  \pi  .
$$ 
At this point, however, we are not able to conclude that equality holds and hence $\QB$ must trace out a circle on $\Ss^2$, which would yield a great deal of information enabling us to basically conclude the proof of Theorem \ref{thm:vsmall}. Instead, we need to proceed by introducing a second Hopf-type differential argument showing that
$$
\pt_z^2 \QB^e \cdot \pt_z^2 \QB^e = 0 .
$$
By combining this equation with the classical Weierstrass--Enneper representation formula for minimal surfaces in $\R^3$, we eventually find that $\Sigma$ must be a flat disk and hence $\pt \Sigma$ is a circle on $\Ss^2$; more precisely, we find that $\pt \Sigma$ is great circle if $v=0$ and a small circle on $\Ss^2$ if $0 < |v| < 1$ . With this geometrical information, combined with complex analysis involving factorization theorems in Hardy spaces, we can then conclude the rest of the proof of Theorem \ref{thm:vsmall}.

The proof of Theorem \ref{thm:vlarge} runs along completely different lines by building up a contradiction argument. Note that a proof of higher regularity for $\QB_v \in \dot{H}^{\frac 1 2}(\R; \Ss^2)$ is {\em not} our disposal when $|v| \geq 1$. In particular, we cannot simply deduce the conformality of the harmonic extension $\QB^e$ due to the lack of boundary regularity.

To circumvent this difficulty, we derive a  {\em ``Pohozaev identity''} for the harmonic extension $\QB^e$ by testing the equation \eqref{eq:uv} against the one-dimensional Hilbert transform $\mathcal{H}( \QB)$ followed by some integration by parts. In summary, we obtain the following identity:
$$
\int \! \! \int_{\R_+^2} |\nabla \QB^e|^2 \, dx \, dy = \frac{2}{v} \int \! \! \int_{\R_+^2} (\QB^e \wedge \pt_x \QB^e) \cdot \pt_y \QB^e \, dx \, dy. 
$$
Now, by using that $| \QB^e | \leq 1$ on $\R_+^2$ (thanks to the maximum principle) it is then straightforward to conclude that $\nabla \QB^e \equiv 0$ holds when $|v| \geq 1$, which shows that $\QB \equiv \mathbf{p}$  with some constant $\mathbf{p} \in \Ss^2$ in this case.

\section{Spectral Analysis of the Linearized Operator}

We now study the linearization for the energy-critical half-wave maps eqaution \eqref{eq:hwm} around a given traveling solitary wave. In what follows, we will deal with the important special case of static solutions, i.\,e., the solitary wave profile is a half-harmonic map. In fact, the spectral results derived below will provide the basis for any future stability and possible blowup analysis for the flow \eqref{eq:hwm} close to half-harmonic maps.  Furthermore, our findings will also have direct applications to the study of the {\em half-harmonic map heat flow} close to half-harmonic maps; see below for more details on this.

Let us suppose  that $\uu(t,x)$ is a (sufficiently regular) solution of the half-wave maps equation \eqref{eq:hwm} defined on some time interval $[0,T)$. We make the following splitting
\be
\uu(t,x) = \QB(x) + \hh(t,x),
\ee
where $\QB \in \dot{H}^{\frac 1 2}(\R; \Ss^2)$ is some fixed half-harmonic map, i.\,e., a solution of \eqref{eq:uv} with $v=0$. 

Thanks to Theorem \ref{thm:vsmall} we can assume without loss of generality that $\QB = (f,g,0)$ is an equatorial half-harmonic map. In order to express the perturbation field $\hh(t,x)$ in a convenient way, we choose the orthonormal frame $\{ \eB, J \eB , \QB \}$ on $\R^3$ with $\eB  := \eB_z =(0,0,1)$ and $J \eB := \QB \wedge \eB$. Hence we can write
\be
\hh = h_1 \eB + h_2 J\eB + h_3 \QB
\ee
with some real-valued functions $h_i(t,x)$ where  $i=1,2,3$. Note that the orthonormal frame $\{ \eB, J\eB \}$ spans the tangent space $T_{\QB} \Ss^2$ and observe that $h_3 = O(\hh^2)$ due to the fact that $ | \QB + \hh |^2 = 1$ holds almost everywhere. A straightforward calculation (see Section \ref{sec:linearize} below) shows that the components of the tangential part $h_1 \eB  + h_2  J \eB \in T_{\QB} \Ss^2$ solve the equation  
\be \label{eq:linearized_hwm}
\pt_t \left [ \begin{array}{c} h_1 \\ h_2 \end{array} \right ]  = J L   \left [ \begin{array}{c} h_1 \\ h_2 \end{array} \right ] + O(\hh^2) ,
\ee
where $O(\hh^2)$ stands for quadratic terms in $\hh$ and $|\nabla| \hh$. The linearized operator is found to be
\be
J L = \left [ \begin{array}{cc} 0 & -1 \\ 1 & 0 \end{array} \right ]  \left [ \begin{array}{cc} L_+ & 0 \\ 0 & L_- \end{array} \right ] = \left [ \begin{array}{cc} 0 & - L_- \\ L_+ & 0  \end{array} \right ].
\ee
Here $L_+$ and $L_-$ are the scalar operators given by
\be
L_+ = |\nabla| - | \pt_x \QB | \quad \mbox{and} \quad L_- = |\nabla| - |\pt_x \QB|  + R ,
\ee
and $R$ denotes the integral operator of the form
\be
(R f)(x) = \frac{1}{2 \pi} \int_{\R} \frac{|\QB(x)- \QB(y)|^2}{|x-y|^2} f(y) \, dy. 
\ee
Our choice to name the scalar operators above as $L_+$ and $L_-$ stems from the commonly used notation in the analysis of solitary waves for nonlinear Schr\"odinger equations (NLS), which has some structural resemblance. In a further analogy to (NLS), the operators $L_+$ and  $L_-$ also arise in the linearization of the energy functional. More specifically, it is easy to see that
\be
E[\QB + \hh] = E[\QB] + \frac{1}{2} \left [ (h_1, L_+ h_1) + (h_2, L_- h_2)  \right ] + o(\hh^2)
\ee
with the notation introduced above, expressing the fact that $L= \mathrm{diag}(L_+, L_-)$ is the Hessian  $D^2 E[\QB]$ of the energy functional  at the critical point $\QB$. We refer to Section \ref{sec:linearize} for details on how to obtain the linearization sketched above.
 
For the rest of this section,  we shall focus on the natural case where $\QB$ is a {\bf pure-power half-harmonic map} of degree $m \geq 1$. By this we mean that the Blaschke product $\mathcal{B}(z)$ appearing in Theorem \ref{thm:vsmall} consists of identical factors. In such a case, we can assume without loss of generality that we have
\be \label{eq:QBm}
\QB(x) = \QB_m(x) := (f_m(x), g_m(x), 0 ) ,
\ee 
where the functions $f_k, g_k \in \dot{H}^{\frac 1 2}(\R)$ with $k \in \N$ are given by
\be \label{def:fkgk}
f_k(x) := \mathrm{Re} \, \left ( \im \cdot \frac{x-\im}{x+\im} \right )^k \quad \mbox{and} \quad g_k(x) := \mathrm{Im} \, \left ( \im \cdot \frac{x-\im}{x+\im} \right)^k.
\ee
For an expression of the functions $f_k$ and $g_k$ in terms of {\em Chebyshev polynomials}, see formula \eqref{eq:TmUm} below. We remark that the appearance of the global phase factor $e^{\im \vartheta} = \im^m$ in $\QB_m=(f_m,g_m,0)$ turns out to be convenient with respect to our choice of the stereographic projection $\Pi : \Ss^1 \to \R \cup \{ \infty \}$ used below. In particular,  the complex-valued function $Q_m = f_m + \im g_m$ becomes the monomial $z^m$ under the conformal lifting to the unit circle $\Ss^1 \subset \C$, i.\,e., we have
\be
(Q_m \circ \Pi)(e^{\im \theta}) = e^{\im m \theta} = z^m \quad \mbox{for $|z| = 1$}.
\ee

We remark that in the {\em ground state case}, i.\,e., for half-harmonic maps with degree $m=1$, we can assume that the form \eqref{eq:QBm} holds with {\em no loss of generality}. That is, any half-harmonic map $\QB \in \dot{H}^{\frac 1 2}(\R; \Ss^2)$ with degree $m=1$ can be brought into this form after a suitable rotation on the target $\Ss^2$ combined with suitable rescalings and translations on the domain $\R$. Of course, in the higher degree case $m \geq 2$, the assumed form in \eqref{eq:QBm} imposes an extra condition on $\QB$.
 
For $\QB = \QB_m$ as above, it is elementary to check that $|\pt_x \QB|= \frac{2m}{1+x^2}$ holds. In this case, we thus find that
\be
L_+ = |\nabla| - \frac{2m}{1+x^2} \quad \mbox{and} \quad L_- = L_+ + R .
\ee
In the following discussion, we will in fact give a complete spectral analysis of these operators $L_+$ and $L_-$. As an interesting aside, we notice that operators of the form $L_+$ also play a central in the study of solitons for the completely integrable {\em Benjamin--Ono equation (BO)}; see e.\,g.~\cite{BO-83}. However, the known methods to study the spectral properties of linearized operators arising in the solion study for (BO) do not seem to be applicable for general $m \geq 1$ and, moreover, the nonlocal integral operator $R$ appearing in $L_-$ does not seem to fit these methods.

\subsection{Nullspace, Nondegeneracy, and $L^2$-Eigenvalues}

As our first main result in this section, we explicitly determine the nullspaces of $L_+$ and $L_-$ in $\dot{H}^{\frac 1 2}$, which we denote as
\be
\mathcal{N}(L_+) = \{ f \in \dot{H}^{\frac 1 2}(\R) : L_+ f = 0 \} ,
\ee 
and we define $\mathcal{N}(L_-)$ accordingly. Note that elements  of $\mathcal{N}(L_+)$ or $\mathcal{N}(L_-)$ do not necessarily belong to $L^2(\R)$, in which case we say that we have a {\bf resonance}. In fact, we will prove the existence of a (common) resonance of $L_+$ and $L_-$ below. 

We remark that the rotational symmetry on the target $\Ss^2$ and the explicit classification result stated in Theorem \ref{thm:vsmall} imply the general lower bounds:
\be
\dim \mathcal{N}(L_+) \geq 2 \quad \mbox{and} \quad \dim \mathcal{N}(L_-) \geq 2m + 1 .
\ee
In fact, this follows from the rotational symmetry around the $x$- and $y$-axis and the presence of $2m+1$ real parameters in the Blaschke product $\mathcal{B}(z)$ that appears in Theorem \ref{thm:vsmall}, respectively. If both inequalities above turn out to be equalities, we say that the linearized operator $L= \mathrm{diag}(L_+, L_-)$ is {\bf nondegenerate}, since in this case all zero modes are due to the degrees of freedom of the solution set. Indeed, the following result establishes this key fact among other things, where we recall that the functions $f_k, g_k \in \dot{H}^{\frac 1 2}(\R)$ have been introduced in \eqref{def:fkgk} above.
 
\begin{thm}[Nullspace, Nondegeneracy, and Resonance] \label{thm:null}
For any degree $m \geq 1$, the nullspaces of the operators $L_+$ and $L_-$ as above are given by
$$
\mathcal{N}(L_+) = \mathrm{span} \, \{ f_m , g_m \} \quad \mbox{and} \quad \mathcal{N}(L_-) = \mathrm{span} \, \{ 1, f_1, \ldots, f_m, g_1, \ldots, g_m \} .
$$ 
As a consequence, we have that 
$$
\dim \mathcal{N}(L_+) = 2 \quad \mbox{and} \quad \dim \mathcal{N}(L_-) = 2m+1
$$
and hence {\bf nondegeneracy} holds for the linearized operator around half-harmonic maps $\QB= \QB_m \in \dot{H}^{\frac 1 2}(\R;\Ss^2)$ for any degree $m \geq 1$.

In particular, the operators $L_+$ and $L_-$ have a common {\bf resonance} $\vphi \in \dot{H}^{\frac 1 2}(\R) \setminus L^2(\R)$, which is given by 
$$
\vphi = \begin{dcases*} f_m & if $m$ is odd, \\ g_m & if $m$ is even. \end{dcases*} 
$$
\end{thm}
 
\begin{remarks*}
{\em 1) We mention that, for the energy-critical Schr\"odinger maps equation with target $\Ss^2$, the zero energy resonance plays an important role in the construction of finite-time blowup solutions close to the ground state harmonic-map with degree $m=1$; see \cite{MeRaRo-13}.

2) After finalizing our paper, it was brought to our attention that the very recent work by Sire--Wei--Zheng in \cite{SiWeZh-17} establishes a nondegeneracy result for half-harmonic maps $\QB : \R \to \S^1$ in the special case of degree $m=1$. }
\end{remarks*}

Next, we turn to study the eigenvalues of the operators $L_+$ and $L_-$ acting on $L^2(\R)$. From standard operator theory we infer that $L_+$ and $L_-$ are both self-adjoint with operator domain $H^1(\R)$ and essential spectrum $\sigma_{\mathrm{ess}}(L_+) = \sigma_{\mathrm{ess}}(L_-) = [0, \infty)$. In order to study the set of eigenvalues, we  introduce the point spectrum denoted by
$$
\sigma_{\mathrm{p}}(L_+) = \{ E \in \R : \mbox{$E$ is an $L^2$-eigenvalue of $L_+$} \}
$$
where the definition of $\sigma_{\mathrm{p}}(L_-)$ is analogous. We have the following  result.

\begin{thm}[Point Spectrum in $L^2$] \label{thm:L2}
Let $m \geq 1$ and consider the operators $L_+$ and $L_-$ from above as acting on $L^2(\R)$ with domain $H^1(\R)$. Then the following properties hold.
 
\begin{itemize}
\item[(i)] {\bf $L^2$-Eigenvalues of $L_+$:} The point spectrum of $L_+$  consists of exactly $2m$ eigenvalues, i.\,e., 
$$
\sigma_\mathrm{p}(L_+) = \left \{ E_{0}, E_1, \ldots, E_{2m-1} \right  \}.
$$
Moreover, each eigenvalue $E_k$ is {\bf simple} and we have the inequalities
$$
E_0 < E_1 < \ldots < E_{2m-2} < E_{2m-1} = 0.
$$

\item[(ii)] {\bf $L^2$-Eigenvalues of $L_-$:} The point spectrum of $L_-$ only contains zero, i.\,e.,
$$
\sigma_{\mathrm{p}}(L_-) = \{ 0 \}.
$$
Moreover, the eigenvalue $E=0$ is exactly $2m$-fold degenerate.
\end{itemize}

In particular, both operators $L_+$ and $L_-$ have {\bf no embedded} eigenvalue $E > 0$ inside  $\sigma_{\mathrm{ess}}(L_+) = \sigma_{\mathrm{ess}}(L_-) = [0, \infty)$.
\end{thm}

As an application of our analysis of the nullspaces and the $L^2$-eigenvalues of $L_+$ and $L_-$ from above, we can derive the following (sharp) coercivity estimate in the energy space $\dot{H}^{\frac 1 2}$, which can be seen as necessary starting point for the modulational analysis. In order to formulate suitable orthogonality conditions, we introduce the following set of $2m+1$ functions $\{ \psi_k \}_{k=0}^{2m} \subset H^1(\R)$ defined as
\be \label{def:psi_k}
\psi_k = \begin{dcases*} \vphi_k & for $0 \leq k \leq 2m-2$, \\
\frac{\vphi_{2m-1}}{1+x^2}  & for $k=2m-1$,\\ 
 \frac{\vphi}{1+x^2} & for $k=2m$. \end{dcases*}
\ee
Here $\{ \vphi_0, \ldots, \vphi_{2m-2} \}$ denote the $L^2$-eigenfunctions of $L_+$ with strictly negative energy $E < 0$, whereas $\vphi_{2m-1}$ denotes the unique $L^2$-eigenfunction of $L_+$ for the eigenvalue $E=0$, and $\vphi \in \dot{H}^{\frac 1 2}(\R) \setminus L^2(\R)$ is the zero energy resonance from Theorem \ref{thm:null}. The weight factor $(1+x^2)^{-1}$ turns out to be a convenient choice for localizing the weakly and non-decaying functions $|\vphi_{2m-1}(x)| \sim \langle x \rangle^{-1}$ (for $m=1$) and $|\vphi(x)| \sim 1$ (for any $m \geq 1$) as $|x| \to \infty$, respectively. Indeed, we  have the following coercivity estimate with the optimal coercivity constant.

\begin{cor}[Coercivity Estimate in $\dot{H}^{\frac 1 2}$] \label{cor:coercive}
Let $L_+$, $L_-$ and $\{ \psi_k \}_{k=0}^{2m}$ be as above. Suppose that $f, g \in \dot{H}^{\frac 1 2}$ satisfy the $2m+1$ orthogonality conditions given by
$$
\int_{\R} \psi_k f \, dx  = \int_{\R} \psi_k g \,dx =0 \quad \mbox{for $0 \leq k \leq 2m$}.
$$
Then it holds
$$
(f, L_+ f) + (g, L_- g) \geq \frac{1}{m+1} \left ( \| f \|_{\dot{H}^{\frac 1 2}}^2 + \| g \|_{\dot{H}^{\frac 1 2}}^2 \right ).
$$
\end{cor}

\begin{remarks*}
{\em 1) We have the decay estimate $|\psi_k(x)| \lesssim \langle x \rangle^{-2}$, which guarantees that the pairing $\int_{\R} \psi_k f \,d x$ is well-defined for $f \in \dot{H}^{\frac 1 2}(\R) \subset L^1(\R, (1+x^2)^{-1} dx)$.

2) Of course, different choices for orthogonality conditions are possible. But the choice above turns out to be natural with regard to the stereographic projection used in the spectral analysis below.}
\end{remarks*}

\subsection{Continuous Spectrum and Unitary Gauge Transform}

Let us now turn to the ``scattering states'' of $L_+$ and $L_-$, which are given by the orthogonal complement of the bound states, i.\,e., the eigenfunctions given by Theorem \ref{thm:L2} above. Indeed, let $\{ \vphi_k \}_{k=0}^{2m-1} \subset H^1(\R)$ denote the set of $L^2$-orthonormalized eigenfunctions for $L_+$. The corresponding orthogonal projection $P : L^2(\R) \to L^2(\R)$ onto the subspace spanned by these eigenfunctions is given by 
\be \label{def:Pproj}
P = \sum_{k=0}^{2m-1}  \vphi_k  ( \vphi_k , \cdot) .
\ee
A remarkable and essential fact shown below is that the operators $L_+$ and $L_-$ {\em coincide} on the orthogonal complement spanned by the eigenfunctions $\{ \vphi_k \}_{k=0}^{2m-1}$. Thus if we define the orthogonal projection $P^\perp = \mathds{1} - P$, we obtain
\be
P^\perp L_+ P^\perp = P^\perp L_- P^\perp .
\ee
In fact, we will prove much more below by showing a common unitary equivalence of $L_+$ and $L_-$ on $P^\perp L^2(\R)$ to the ``free'' operator $|\nabla|$. 

To do so, we first need to introduce the  projection $\Pi_+ f := \mathcal{F}(\mathds{1}_{\xi \geq 0} \widehat{f})$ for $f \in L^2(\R)$ onto the subspace of positive Fourier frequencies, which we denote by $L_+^2(\R)$. Accordingly, we define $\Pi_- :=  \mathds{1} - \Pi_+$ as the projection onto the $L^2$-subspace of negative frequencies denoted as $L_-^2(\R)$. With the help $\Pi_{\pm}$, we define the map
$$
U : L^2(\R) \to  P^\perp L^2(\R), \quad f \mapsto Q_m f_+ + \overline{Q}_m f_-.
$$
with $f_+ = \Pi_+ f$, $f_- = \Pi_- f$, and where we recall that
$$
Q_m(x) = \left ( \im \cdot \frac{x -\im}{x+\im} \right )^m.
$$ 
The adjoint operator of $U$ is given by
$$
U^* : P^\perp L^2(\R) \to L^2(\R), \quad g \mapsto \overline{Q}_m g_+ + Q_m g_-,
$$
where $g_+ =\Pi_+ g$ and $g_- = \Pi_- g$. Indeed, we now establish that $U : L^2(\R) \to P^\perp L^2(\R)$ is  well-defined and that is furnishes {\bf unitary map} as follows. 

\begin{thm}[Unitary Equivalence of $L_\pm$ to $|\nabla|$ on Continuous Spectrum] \label{thm:unitaryL2} The map $U : L^2(\R) \to P^\perp L^2(\R)$ is unitary and, for the operators $L_+$ and $L_-$ as above, we have the unitary equivalence
$$
U^* P^\perp L_+ P^\perp U = U^* P^\perp  L_- P^\perp U =  |\nabla| .
$$
As a consequence, the operators $L_+$ and $L_-$ have purely absolutely continuous spectrum equal to $[0, \infty)$.
\end{thm}

The mechanism behind the unitary map $U$ introduced above can be seen as a sort of a {\bf gauge transform} that allows us to treat perturbations around $\QB$ in a convenient way. We refer to the next subsection below for a potential future application of $U$ for the linearized flow \eqref{eq:linearized_hwm} of the half-wave maps equation  to obtain a {\em complex-scalar half-wave equation} in the linearized setting. Also, another potential application of this gauge transform $U$ can arise for the linearization around half-harmonic maps for the {\em (energy-critical) half-harmonic map heat flow} which can be written as
\be
\pt_t \uu = P_\uu |\nabla| \uu,
\ee
where $P_\uu \vv = \vv - (\vv \cdot \uu) \uu$ denotes the projection of $\vv \in \R^3$ onto the tangent space $T_\uu \Ss^2$ at $\uu \in \Ss^2$.


\subsection{Spectral Properties of the Matrix Operator $JL$} To conclude the spectral analysis of the linearized operator, we finally consider the matrix operator
\be
\mathcal{L} = JL = \left [ \begin{array}{cc} 0 & - L_- \\ L_+ & 0  \end{array} \right ],
\ee
which appears in the linearization of the half-wave maps equation close the a half-harmonic map $\QB$. For a potential stability or blowup analysis close to $\QB=\QB_m$, it is of importance to obtain spectral information about $\mathcal{L}$. Although our knowledge about the spectral properties of $L=\mathrm{diag}(L_+, L_-)$ is complete, it is far from obvious that this extends to $\mathcal{L} = JL$. In fact, it is in general an intricate spectral problem to derive detailed information about the spectrum $\sigma(\mathcal{L})$ from the spectral knowledge of the diagonal operator $L$. But luckily in our case, we are able to obtain a complete description of the spectral properties of the matrix operator $\mathcal{L}$ that arises from linearizing around $\QB = \QB_m$. 

\begin{thm}[Spectrum of $\mathcal{L}=JL$] \label{thm:JL}
Consider the operator $\mathcal{L}=JL$ from above as acting on $L^2(\R) \times L^2(\R)$ with domain $H^1(\R) \times H^1(\R)$. Then the following properties hold true.
\begin{itemize}
\item[(i)] The spectrum is given by $\sigma( \mathcal{L}) = \im \R$.
\item[(ii)] $E=0$ is the only eigenvalue of $\mathcal{L}$ and we have $\dim \ker \, ( \mathcal{L} ) = 2m + 1$. In particular, the operator $\mathcal{L}$ has {\bf no embedded eigenvalue} $E \in \im \R \setminus \{ 0 \}$.
\item[(iii)] The generalized nullspace of $\mathcal{L}$ acting on $L^2(\R) \times L^2(\R)$ is given by
$$
\bigcup_{n \geq 1} \mathrm{ker} (\mathcal{L}^n) = \mathrm{ker} ( \mathcal{L}^2 ) = \mathrm{span} \left  \{ \left [ \begin{array}{c} \vphi_0 \\  0 \end{array} \right ] , \ldots,   \left [ \begin{array}{c} \vphi_{2m-1} \\  0 \end{array} \right ],  \left [ \begin{array}{c} 0 \\  \vphi_0 \end{array} \right ] , \ldots,  \left [ \begin{array}{c} 0 \\ \vphi_{2m-1} \\   \end{array} \right ]\right \}
$$
\item[(iv)] $\mathcal{L}$ has two linearly independent {\bf resonances} at $E=0$ given by $[ \vphi, 0]^T$ and $[0, \vphi]^T$, where $\vphi \in \dot{H}^{\frac 1 2}(\R) \setminus L^2(\R)$ is the same function as in Theorem \ref{thm:null} above.
\end{itemize}
 \end{thm}
In fact, this theorem is essentially a corollary of the preceding discussion and we sketch the main ideas. Given the projection $P: L^2(\R) \to L^2(\R)$  from \eqref{def:Pproj} above, we define the orthogonal projection 
$$
\mathcal{P} : L^2(\R) \times L^2(\R) \to L^2(\R) \times L^2(\R) \quad \mbox{with} \quad \mathcal{P} = \left [ \begin{array}{cc} P & 0 \\ 0 & P \end{array} \right ] ,
$$
where we set $\mathcal{P}^\perp = \mathds{1}  - \mathcal{P}$. 
Recalling that $P L_- P = 0$ and $P^\perp L_+ P^\perp = P^\perp L_- P^\perp$ and using the unitary equivalence provided by Theorem \ref{thm:unitaryL2}, we obtain
\be \label{eq:niceL}
\mathcal{L} = \mathcal{P} \mathcal{L} \mathcal{P} + \mathcal{P}^\perp \mathcal{L} \mathcal{P}^\perp =  \left [ \begin{array}{cc} 0 & 0 \\ P L_+ P & 0 \end{array} \right ]  + \mathcal{U}  \left [ \begin{array}{cc} 0 & -|\nabla| \\ |\nabla| & 0 \end{array} \right ] \mathcal{U}^*
\ee
with the unitary map defined as
\be
\mathcal{U} : L^2(\R) \times L^2(\R) \to \mathcal{P}^\perp (L^2(\R) \times L^2(\R)) \quad \mbox{with} \quad \mathcal{U} = \left [ \begin{array}{cc} U & 0 \\ 0 & U \end{array} \right ] ,
\ee 
where $U : L^2(\R) \to P^\perp L^2(\R)$ is the unitary map from Theorem \ref{thm:unitaryL2} above. By using the key identity \eqref{eq:niceL}, it is not hard to deduce the properties in Theorem \ref{thm:JL}; see Section \ref{sec:spectrum} for more details.

 A natural application of \eqref{eq:niceL} with respect to the linearized equation \eqref{eq:linearized_hwm} is as follows. Suppose that, by carrying out some modulation theory close to $\QB$, we can arrange that the perturbation satisfies the orthogonality conditions $P h_1 = Ph_2 = 0$. Then if we apply the ``gauge transform'' given by
 $$
\left [ \begin{array}{c} h_1 \\ h_2 \end{array} \right ] \mapsto    \left [ \begin{array}{c} h_1^g \\ h_2^g \end{array} \right ] = \mathcal{U}  \left [ \begin{array}{c} h_1 \\ h_2 \end{array} \right ]
 $$
and if we define the complex-valued function $\psi(t,x) = h_1^g(t,x) + \im h_2^g(t,x) \in \C$, we see 
 \be
 \im \pt_t \psi = |\nabla| \psi +  \mbox{nonlinear terms}.
 \ee
 Thus the linearized equation \eqref{eq:linearized_hwm} can be naturally cast into the form of a half-wave equation (with nonlinear terms) for a complex field $\psi : [0,T) \times \R \to \C$. We believe that this observation will be of great use for the future well-posedness as well as stability/blowup analysis close to half-harmonic maps for the half-wave maps equation.

\subsection{Comments on the Proofs} 
The proofs of Theorems \ref{thm:null} -- \ref{thm:JL} all have in common that we make essential use of the stereographic projection $\Pi$ from the unit circle $\Ss^1$ to the projective real line $\hat{\R} = \R \cup \{ \infty \}$. Let us remark that the use of stereographic projections to gain insight into nonlinear problems with conformal symmetry is, of course, not new by itself; in particular, we mention here the pioneering works of E.~Lieb \cite{Li-83} and W.~Beckner \cite{Be-93} on conformally invariant functional inequalities. However, the application of such a conformal change of coordinates from $\R$ to $\Ss^1$ turns out to be very successful in the spectral analysis of the operators $L_+$ and $L_-$, leading to a complete understanding of these nonlocal operators that seems to be unmatched to the best of our knowledge.   

More precisely, the starting point is that, by means of this conformal transformation $\Pi$, we can recast the spectral analysis for the operators
$$
L_+= |\nabla| - \frac{2m}{1+x^2} \quad \mbox{and} \quad L_-= |\nabla| - \frac{2m}{1+x^2} + R
$$
in terms of the unbounded operators given by
$$
J = (1-\sin \theta) ( |\nabla|_{\Ss^1} - m \mathds{1}) \quad \mbox{and} \quad H = (1- \sin \theta) ( |\nabla|_{\Ss^1} - m \mathds{1} + \widetilde{R} )
$$
acting on $L^2(\Ss^1)$, where $|\nabla|_{\Ss^1}$ is the square root of the Laplacian on $\Ss^1$ and $\widetilde{R}$ is some integral operator. Clearly, the operators $J \neq J^*$ and $H\neq H^*$ are {\em not self-adjoint}, and therefore their detailed spectral analysis seems to be a hopeless enterprise at first sight. However, the benefit of this approach is that $J$ and $H$ are both seen to be {\bf Jacobi operators}, meaning that the corresponding (infinite) matrices have a tridiagonal structure with respect to the standard Fourier basis $\{ e^{\im k \theta} \}_{k \in \Z}$ of $L^2(\Ss^1)$. For instance, the matrix for $J$ is of the form
$$
[J_{kl}] = \left [ \begin{array}{cccccc} \ddots & \ddots & \ddots &  & &  0 \\ & a_{n} & b_{n} & c_{n} \\ & & a_{n+1} & b_{n+1} & c_{n+1}   \\ 0 & &  & \ddots & \ddots & \ddots  \end{array} \right ] 
$$
with certain sequences $(a_n)$, $(b_n)$, $(c_n)$  and $n \in \Z$. Roughly speaking, the degree $m \in \N$ of the half-harmonic maps $\QB = \QB_m$ plays a central role by splitting the frequencies $k \in \Z$ on the unit circle with $|k| \leq m$ and $|k| \geq m$. More precisely, the analysis of the Jacobi operators given by $J$ and $H$ then shows the following. 
\begin{itemize}
\item The {\em bound states} of $L_+$ and $L_-$ are determined via the actions of $J$ and $H$ on the  $2m+1$-dimensional subspace $\mathrm{span} \, \{ e^{\im k \theta} : |k| \leq m \}$ in $L^2(\Ss^1)$. 
\item The {\em scattering states} of $L_+$ and $L_-$  can be analyzed in detail via the action of $J$ and $H$ on the infinite-dimensional subspace  $\mathrm{span} \, \{ e^{\im k \theta} : |k| \geq m \}$ in $L^2(\Ss^1)$.
\end{itemize}
Here, a particular role will be played the two-dimensional subspace spanned by $\{ e^{\im k \theta} : k = \pm m \}$, which lies at the interface between the bound states and scattering states (i.\,e.~the continuous spectrum), and which yields the the two linearly independent solutions of $L_+ \vphi = L_-\vphi = 0$ given by the $L^2$-zero mode and the zero-energy resonance. 

Let us mention two facts in the spectral analysis that we find particularly remarkable. First, we are able to find a {\bf Darboux-type factorization formula}\footnote{Note the reminiscence to the classical Darboux factorization trick for one-dimensional Schr\"odinger operators $H=-\pt_{xx} +V$, i.\,e., we can write $H= S^* S + e_0$ with $S=\phi \pt_x \phi^{-1}$ where $\phi>0$ is the ground state of $H$ with $H \phi = e_0 \phi$. Note that we have $Q_m^{-1} = \overline{Q}_m$ in our case, since $|Q_m|=1$ holds.} of the form
$$
L_+ f = Q_m \left ( -\im \frac{d}{dx} \right ) \overline{Q}_m f_+ + \overline{Q}_m \left ( + \im \frac{d}{dx} \right ) Q_m f_-
$$ 
where $f_+ = \Pi_+ f$ and $f_- = \Pi_- f$ denote the projections of $f$ onto positive and negative Fourier frequencies, respectively. The use of the frequency projectors $\Pi_{\pm}$ will introduce a fair amount of results from Hardy space theory at various places in the analysis. Furthermore, the somewhat intriguing factorization formula above will serve as the starting point for finding the gauge transform $U$ that provides the joint unitary equivalence in Theorem \ref{thm:unitaryL2} of the free operator $|\nabla|$ to both $L_+$ and $L_-$ restricted to the continuous spectrum. As an interesting aside, we note that the above formula shows that the generalized scattering solutions $\vphi \in L_{1/2}(\R)$ of $L_+ \vphi = E \vphi$ with $E > 0$ are given by linear combinations of
$$
\vphi_+ (x) = Q_m(x) e^{\im E x} \quad \mbox{and} \quad \vphi_-(x) = \overline{Q}_m(x) e^{-\im E x} ,
$$
which correspond to the {\bf Jost solutions} in classical scattering theory for one-dimensional Schr\"odinger operators $H = -\pt_{xx} + V$. Another interesting spectral feature ( ue to the tridiagonal structure of the matrix for $J$) is that the $L^2$-eigenvalues $E_k$ of $L_+$ are all simple and that $E_k$ together with the corresponding $L^2$-eigenfunctions $\vphi_k$ can be explicitly calculated by using orthogonal polynomials; see Section \ref{sec:spectrum} below for explicit examples for degree $m=1$ and $m=2$.

 
\subsection*{Definitions and Notations} For $n \in \N$, we define the weighted space $L_{1/2}(\R; \R^n) = \{ f \in L^{1}_{\mathrm{loc}}(\R; \R^n) : \int_{\R} \frac{|f(x)|}{1+x^2} \,dx < +\infty \}$. By duality (see, e.\,g., \cite{DaMaRi-15}), we can define the action $|\nabla|$ on $L_{1/2}(\R; \R^n)$ and introduce the space
\be
\dot{H}^{\frac 1 2}(\R; \R^n) = \{ f \in L_{1/2}(\R; \R^n) : \| f \|_{\dot{H}^{\frac 1 2}} < +\infty \},
\ee 
with the (semi)-norm $\| f \|_{\dot{H}^{\frac 1 2}}$ given by
\be
\| f \|_{\dot{H}^{\frac 1 2}}^2 = \int_{\R} f \cdot |\nabla| f \, dx = \frac{1}{2 \pi} \int \! \! \int_{\R \times \R} \frac{|f(x)-f(y)|^2}{|x-y|^2} \, dx \, dy .
\ee
Note that any non-trivial constant function $f = \mbox{const}.$ belongs to $\dot{H}^{\frac 1 2}$. Furthermore, we define the space 
$$
\dot{H}^{\frac 1 2}(\R; \Ss^2) = \left \{ \uu \in \dot{H}^{\frac 1 2}(\R; \R^3) : \mbox{$|\uu(x)| = 1$ for a.\,e.~$x \in \R$}  \right \} .
$$
For $s \geq 0$, we define homogeneous and inhomogeneous Sobolev spaces $\dot{H}^s$ and $H^s$  in a similar standard manner. On $L^2(\R)$ and $L^2(\Ss^1)$, we define the complex scalar products
$$
(f,g )= \int_{\R} \overline{f}(x) g(x) \, dx \quad \mbox{and} \quad \langle u, v \rangle = \int_{\Ss^1} \overline{u}(e^{\im \theta}) v(e^{\im \theta}) \, d \theta,
$$
respectively. In fact, we ultimately deal with real-valued functions, but nevertheless it is sometimes convenient to make use of complex-valued functions in our analysis. 

\subsection*{Acknowledgments}
It is a pleasure to thank G.\,M.~Graf and R.\,L.~Frank for helpful conversations regarding the action of the Lorentz group on $\Ss^2$ and the use of the stereographic projection in the spectral analysis, respectively. E.~Lenzmann acknowledges financial support by the Swiss National Science Foundation (SNF) and he is also grateful for the kind hospitality of the I.H.\'E.S., where part of this work was done. A.~Schikorra is supported by the German Research Foundation (DFG) through Grant No.~SCHI-1257-3-1 and a Heisenberg fellowship.

\section{Proving the Classification Results}

In this section we prove Theorems \ref{thm:vsmall} and \ref{thm:vlarge}, which provide the explicit classification of all boosted half-harmonic maps, i.\,e., solutions $\QB_v \in \dot{H}^{\frac 1 2}(\R; \Ss^2)$ of \eqref{eq:uv} with $v \in \R$ given.

\subsection{Free and Non-Free Boundary Minimal Disks}

The main step in the proof of Theorem \ref{thm:vsmall} will involve the use of {\em two Hopf-type differentials} to show that the image $\QB_v(\R)$ belongs to a fixed plane in $\R^3$ and hence $x \mapsto \QB_v(x)$ traces out a circle on $\Ss^2$ (which will be a small circle if $v \neq 0$). After having established this geometric fact, we can apply a M\"obius transformation on the target $\Ss^2$ and transform $\QB_v$ to a half-harmonic map $\QB : \R \to \Ss^2$ lying in the equatorial plane, i.\,e., we have $\QB(x) = (f(x), g(x), 0)$ for some functions such that $f(x)^2 + g(x)^2 =1$ for a.\,e.~$x$. The latter situation can then be completely understood by means of complex analysis, as already used in \cite{Millot-Sire-2015}; see also \cite{BeMiRySa-14}.

Throughout this subsection, we assume that $\QB_v \in \dot{H}^{\frac 1 2}(\R; \Ss^2)$ is a  solution to the profile equation \eqref{eq:uv} with some given $v \in \R$ such that $|v| < 1$. For notational simplicity, we drop the dependence on $v$ and we shall write 
$$
\QB \equiv \QB_v .
$$ 
 By the regularity result in Theorem \ref{th:regularity} worked out in Section \ref{sec:regularity}, we have that $\QB \in  \dot{H}^2 \cap C^\infty(\R; \Ss^2)$. We will make use of this regularity result frequently without further reference. 
 
 We begin with the following useful pointwise identity, which is simple to prove.

\begin{lemma} \label{lem:rigid}
For any $x \in \R$, it holds that
$$
(\QB \cdot |\nabla| \QB)(x) = \frac{1}{2 \pi} \int_{\R} \frac{ | \QB(x) - \QB(y)|^2}{|x-y|^2} \, dy .
$$
In particular, we have the following rigidity property: If $(\QB \cdot |\nabla| \QB)(x) = 0$ for some $x \in \R$, then $\QB \equiv \mbox{const}.$ holds.
\end{lemma}

\begin{proof}
By combining the singular integral formula $|\nabla| \uu(x) = \frac{1}{\pi} \int_{\R} \frac{ \uu(x) - \uu(y)}{|x-y|^2} \, dy$ with the identity $2 \uu \cdot (\uu - \vv) = | \uu - \vv |^2$ valid for $\uu, \vv \in \R^3$ with $|\uu|=|\vv|=1$, we obtain the desired identity. The rigidity property is obvious from the identity.
\end{proof}
 
As a next step, we reformulate the problem of solving \eqref{eq:uv} in terms of its harmonic extension to the upper half-plane $\R_+^2 = \R \times \{ y > 0 \}$. To this end, we let $\QB^e : \R_+^2 \to \R^3$ denote the harmonic extension of $\QB : \R \to \Ss^2$ to the upper half-plane $\R_+^2$ given by the classical Poisson extension formula
\be \label{eq:Udef}
\QB^e(x,y) = (P_y \ast \QB)(x) = \int_{\R} P_y(x-x') \QB(x') \, dx' \quad \mbox{for $(x,y) \in \R_+^2$},
\ee
where 
\be
P_y(x) = \frac{1}{ \pi} \frac{y}{y^2+x^2}
\ee
denotes the Poisson kernel defined for $x \in \R$ and $y > 0$. Since $\QB$ lies in the space  $\dot{H}^{\frac 1 2}$, it is a classical fact that its harmonic extension $\QB^e$ belongs to $\dot{H}^1(\R_+^2)$ with
\be
\| \QB^e \|_{\dot{H}^1(\R_+^2)}^2 = \int \! \! \int_{\R_+^2} \left ( |\pt_x \QB^e|^2 + |\pt_y \QB^e|^2  \right ) \, dx \, dy = \int_{\R} \QB \cdot |\nabla| \QB \, dx .
\ee
Moreover, we recall that $\QB^e : \R_+^2 \to \R^3$ satisfies the boundary-value problem
\be
\left \{ \begin{array}{ll} \DD \QB^e = 0 & \quad \mbox{in $\R_+^2$}, \\
v \pt_x \QB^e  + \QB^e \wedge \pt_y \QB^e = 0& \quad \mbox{on $\pt \R_+^2 = \R \times \{y=0 \}$},
\end{array} \right . 
\ee
thanks to well-known facts that $- \pt_y \QB^e = |\nabla| \QB$ and $\QB^e = \QB$ hold on $\pt \R_+^2$ in the trace sense. We also notice that it is elementary to check from \eqref{eq:Udef} that
$$
\sup_{(x,y) \in \R_+^2} | \QB^e(x,y) | \leq \sup_{x \in \R} |\QB(x)| = 1,
$$
which can also be seen from the maximum principle applied to the subharmonic function $|\QB^e|^2$ on $\R_+^2$. In fact, the strict maximum principle tells us that $|\QB^e(x,y)| < 1$ holds in $\R_+^2$, unless $\QB^e$ is constant. 

In the following, it will often be convenient to identify $\R_+^2$ with the complex upper halfplane $\C_+ = \{ z \in \C : \mathrm{Im} \, z > 0 \}$ via the canonical relation $z= x+ \im y$ for $x,y \in \R$. Furthermore, we recall the definition of the Wirtinger operators given by
\be
\pt_z := \frac 1 2 ( \pt_x - \im \pt_y) \quad \mbox{and} \quad \pt_{\bar{z}} := \frac 1 2 ( \pt_x + \im \pt_y ).
\ee

Next, we make an observation  about the conformality of the harmonic extension $\QB^e$, which is well-known for the half-harmonic maps case when $v=0$ (see \cite{DaMaRi-15, DSpalphSphere, Millot-Sire-2015, DaLio-16}) .

\begin{lemma} \label{lem:Uz1}
For $\QB^e : \R_+^2 \simeq \C_+ \to \R^3$ as above, it holds that
$$
\pt_z \QB^e \cdot \pt_z \QB^e = 0 ,
$$
where $\mathbf{v} \cdot \mathbf{w} = \sum_{i=1}^3 v_i w_i$ for $\mathbf{v}, \mathbf{w} \in \C^3$. Equivalently, we have the $\QB^e$ is a (weakly) conformal map, i.\,e.,
$$
| \pt_x \QB^e|^2 = |\pt_y \QB^e|^2 \quad \mbox{and} \quad \pt_x \QB^e \cdot \pt_y \QB^e = 0  .
$$
In particular, by taking boundary values, we obtain the identities
$$
| \pt_x \QB |^2 = | |\nabla| \QB |^2 \quad \mbox{and} \quad \pt_x \QB \cdot |\nabla| \QB = 0 .
$$ 
\end{lemma}

\begin{remark*}
{\em The conformality together with harmonicity imply that $\QB^e$ parametrizes a (possibly branched) minimal surface $\Sigma$ of disk type that lies inside the unit ball $B \subset \R^3$. In the half-harmonic map case when $v=0$, we readily see that normal vector $\pt_y \QB^e |_{y=0}$ at the boundary is orthogonal to the tangent space $T_{\QB} \Ss^2$, i.\,e., the boundary $\pt \Sigma$ intersects $\pt B = \Ss^2$ orthogonally. In this case, we can refer to $\Sigma$ as a {\bf free minimal disk} inside $B$. By a classical result of Nitsche \cite{Nitsche70}, any such free minimal disk $\Sigma$ is a flat disk in the equatorial plane (up to rotations). 

However, for the case $v \neq 0$, we readily see that $\pt_y \QB^e |_{y=0}$ is {\em not orthogonal} to $T_\QB \Ss^2$. Consequently, we refer to the corresponding surface $\Sigma$ as a {\bf non-free minimal disk} inside $B$. Below, we will prove that $\Sigma$ is a flat disk with radius $r=\sqrt{1-v^2}$.}
\end{remark*}

\begin{proof}
We use a classical technique involving Hopf differentials; we follow closely the arguments in \cite{Millot-Sire-2015} for half-harmonic maps corresponding to the special case when $v=0$. To this end, we define the function $\Phi : \C_+ \to \C$ given by
$$
\Phi(z) := 4 \left ( \pt_z \QB^e \cdot \pt_z \QB^e \right ) = \left ( |\pt_x \QB^e|^2 - |\pt_y \QB^e|^2 \right ) - 2 \im ( \pt_x \QB^e \cdot \pt_y \QB^e  ).
$$
Note that $\Phi$ is holomorphic, since we have that $\pt_{\bar{z}} \Phi = 0$ and $4 \pt_{\bar{z}} \pt_z \QB^e =  \Delta \QB^e = 0$.

We claim that $\Phi(z) \equiv 0$ holds, which will complete the proof of Lemma \ref{lem:Uz1}. To prove that $\Phi$ is identically zero, we first note that
\be \label{eq:Hopf1}
- \pt_x \QB^e \cdot \pt_y \QB^e = \pt_x \QB \cdot |\nabla| \QB = 0 \quad \mbox{on $\partial \C_+ $} \, .
\ee
Indeed, the first equality holds because of $\pt_y \QB^e = - |\nabla| \QB $ and $\pt_x \QB^e = \pt_x \QB$ for every $x \in \R \simeq \partial \C_+$ and using the regularity $\QB \in C^1$. To see that $\pt_x \QB \cdot |\nabla| \QB  = 0$ holds on $\R$, we use \eqref{eq:uv} directly as follows. If $v = 0$, then by \eqref{eq:uv} we have that $|\nabla| \QB = \lambda(\QB) \QB$ for some function $\lambda : \Ss^2 \to \R$. Since $\pt_x \QB \perp \QB$ by the fact that $|\QB|^2 = 1$, we conclude that  \eqref{eq:Hopf1} holds when $v=0$. Now let us assume that $v \neq 0$. But in this case we immediately obtain that $\pt_x \QB \perp |\nabla| \QB$ by the boundary condition in \eqref{eq:uv}. Again, we conclude that \eqref{eq:Hopf1} holds.

Thus we have shown that $g(z) := \mathrm{Im} \, \Phi(z)$ vanishes identically on $\C_+$. By odd reflection across $\pt \C_+$, we can extend the harmonic function $g$ to all of $\C$. However, since $g$ is harmonic and $g \in L^1(\R_+^2)$ because of $\QB^e \in \dot{H}^1(\R_+^2)$, we conclude that $g \equiv 0$ on $\C$. Thus $\Phi$ is real-valued and holomorphic, which implies that $\Phi$ is constant. Since $\Phi \in L^1(\R_+^2)$, we deduce that $\Phi(z) \equiv 0$ holds.
\end{proof}

As a consequence of Lemma \ref{lem:Uz1}, we can recast equation \eqref{eq:uv} into the following form.

\begin{lemma} \label{lem:upoint}
For $\QB : \R \to \Ss^2$ as above, we have the identity
$$
|\nabla| \QB = \sqrt{1-v^2} |\pt_x \QB| \QB - v   \pt_x \QB \wedge \QB .
$$
\end{lemma}

\begin{proof}
We assume that $\QB$ is not constant, since otherwise the assertion is trivial. By Lemma \ref{lem:rigid} and \ref{lem:Uz1}, we know  that $|\pt_x \QB (x) | = | \Dx \QB (x)| \neq 0$ for all $x \in \R$. Furthermore, we recall $\pt_x \QB \perp \Dx \QB$ from Lemma \ref{lem:Uz1} above. Hence, we find
\be \label{eq:uspan}
\Dx \QB = \lambda \QB +  \mu \pt_x \QB \wedge \QB 
\ee
with some real-valued functions $\lambda, \mu : \R \to \R$. By plugging this ansatz into \eqref{eq:uv}, we readily deduce that $\mu(x) \equiv -v$ holds. In order to determine the function $\lambda$, we square \eqref{eq:uspan} on both sides and use that $|\QB|^2 = 1$ to find that $|\Dx \QB|^2 = \lambda^2 + v^2 |\pt_x \QB|^2$. Since $|\pt_x \QB| = | |\nabla| \QB|$ by Lemma \ref{lem:Uz1}, we see that $\lambda = \pm \sqrt{1-v^2} | \pt_x \QB|$. Since $(\QB \cdot \Dx \QB)(x)> 0 $ for all $x \in \R$ by  Lemma \ref{lem:rigid}, we finally see that $\lambda(x) > 0$ must be a positive function. 
\end{proof}

\begin{lemma} \label{lem:Uz2}
For $\QB^e : \R_+^2 \simeq \C_+ \to \R^3$ as above, we have the identity
$$
\pt_{zz} \QB^e \cdot \pt_{zz} \QB^e = 0.
$$
\end{lemma}

\begin{proof}
 Again, we use a Hopf differential type argument similar to the proof of Lemma \ref{lem:Uz1}. But now we consider the function $\Psi : \C_+ \to \C$ defined as
\be
\Psi(z) := 16 \left ( \pt_{zz} \QB^e \cdot \pt_{zz} \QB^e \right ) .
\ee
Note that $\Psi$ is holomorphic, since $\pt_{\bar{z}} \Psi = 0$. Furthermore, we find 
\begin{align*}
\Psi(z) & = \left ( | \pt_{xx} \QB^e|^2  + |\pt_{yy} \QB^e|^2 - 2 | \pt_{xy} \QB^e |^2  \right )  + 4 \im \left ( \pt_{xx} \QB^e  \cdot \pt_{xy} \QB^e - \pt_{yy} \QB^e  \cdot \pt_{xy} \QB^e \right ) \\
& = 2 \left ( |\pt_{xx} \QB^e|^2 - | \pt_{xy} \QB^e|^2  \right ) +  8 \im \left ( \pt_{xx} \QB^e \cdot \pt_{xy} \QB^e \right ) \, ,
\end{align*}
where the last step follows from the harmonicity $-\pt_{xx} \QB^e = \pt_{yy} \QB^e$.

We claim that $\Psi (z) \equiv 0$ holds. To this end, we show that
\be \label{eq:Psi}
- \pt_{xx} \QB^e \cdot \pt_{xy} \QB^e = \pt_{xx} \QB \cdot  \pt_x \Dx \QB = 0 \quad \mbox{on $\pt \C_+$}.
\ee 
Indeed, the first identity follows again by the properties of the harmonic extension together with the regularity properties shown for the boundary function, i.\,e., here we use that $\QB$ is of class $C^2$. To see that the second equation holds true in \eqref{eq:Psi}, we differentiate the identity given in Lemma \ref{lem:upoint}, which gives us
\begin{align*}
\pt_x \Dx  \QB & = \sqrt{1-v^2} \left ( \frac{\pt_x \QB \cdot \pt_{xx} \QB }{| \pt_x \QB |} \QB + |\pt_x \QB| \pt_x \QB \right ) -  v  \pt_{xx} \QB \wedge \QB .
\end{align*}
Recall that $|\pt_x \QB | \neq 0$ for all $x \in \R$ as shown above (since otherwise $\QB$ must be constant by Lemma \ref{lem:rigid}). Now, by taking the scalar product with $\pt_{xx} \QB$ and using that $\pt_{xx} \QB \cdot \QB + \pt_x \QB \cdot \pt_x \QB = 0$ thanks to $|\QB|^2 \equiv 1$, we obtain 
\begin{align*}
\pt_{xx} \QB \cdot \pt_x \Dx \QB & = \sqrt{1-v^2} \left (  - ( \pt_x \QB \cdot \pt_{xx} \QB ) |\pt_x \QB| + |\pt_x \QB | (\pt_{xx} \QB \cdot \pt_x \QB)   \right ) = 0.
\end{align*}
Thus we see that \eqref{eq:Psi} holds, i.\,e., we have $\mathrm{Im} \, \Psi(z) \equiv 0$ on $\pt \C_+$. Next, by the regularity results shown in Appendix \ref{sec:regularity}, we have $\QB \in \dot{H}^{3/2}(\R; \Ss^2)$. Thus its harmonic extension satisfies $\QB^e \in \dot{H}^2(\R_+^2; \R^3)$. By using \eqref{eq:Psi} and adapting the arguments in the proof of Lemma \ref{lem:Uz1}, we finally deduce that $\Psi (z) \equiv 0$ vanishes identically, which completes the proof of Lemma \ref{lem:Uz2}.
\end{proof}

We conclude this subsection by showing that the image of the harmonic extension $\QB^e$ lies in a fixed plane in $\R^3$.

\begin{lemma} \label{lem:plane}
Let $\QB : \R_+^2 \simeq \C_+ \to \R^3$ be as above. Then the image $\QB(\R_+^2)$ belongs to a fixed plane in $\R^3$. As a consequence, the image of the boundary $\QB^e(\partial \R_+^2) = \QB(\R)$ is a circle on $\Ss^2$.
\end{lemma}

\begin{proof} 
We define the function $\xB := \pt_z \QB^e : \C_+ \to \C^3$,  where we write $\xB = (X_1, X_2, X_3)$. We recall from Lemma \ref{lem:Uz1} and \ref{lem:Uz2} that the identities
\be \label{eq:X1}
\xB \cdot \xB = X_1^2 + X_2^2 + X_3^3 = 0 ,
\ee
\be \label{eq:X2}
\pt_z \xB \cdot \pt_z \xB = (\pt_z X_1)^2 + (\pt_z X_2)^2 + (\pt_z X_3)^2 = 0,
\ee
hold on $\C_+$. Note that every component $X_i : \C_+ \to \C$ is a holomorphic function. 

Suppose now that $X_k \not \equiv 0$ for each $k=1,2,3$. (Otherwise if $X_3 \equiv 0$, say, then we readily check that the image $\QB^e(\R_+^2)$ belongs to a plane parallel to the $x_1x_2$-plane.) By the classical {\em Enneper--Weierstrass representation formula} for minimal surfaces (see, e.\,g., \cite{CoMi-11,DiHiSa-10}), we deduce from identity \eqref{eq:X1} that $\xB : \C_+ \to \C^3$ is of the form 
$$
\xB(z) = \left ( \frac{F(z)}{2} \left (1- G(z)^2 \right), \frac{\im F(z)}{2} \left (1+ G(z)^2 \right ), F(z)G(z) \right ),
$$
where $F : \C_+ \to \C$ is some holomorphic function and $G: \C_+ \setminus A \to \C$ is some meromorphic function, where the discrete set $A \subset \C_+$ denotes the poles of $G$. Furthermore, the function $FG^2$ extends to a holomorphic function on all of $\C_+$. 

Let $N = \{ z \in \C_+ : F(z) = 0\}$ denote the set of zeros of the holomorphic function $F$. Note that $N$ only contains isolated points (since $F \not \equiv 0$ by assumption) and observe that $A \subset N$ holds, since every pole of $G$ must be a zero of $F$. 

Next, we introduce the the holomorphic function $H := F G^2$ defined on $\C_+$. Using that $\pt_z H = \pt_z F G^2 + 2 F G \pt_z G$ holds for $z \in \C_+ \setminus N$, we obtain from \eqref{eq:X2} that
\begin{align*}
0 & = \frac 1 4 ( \pt_z F - \pt_z H)^2 - \frac{1}{4} ( \pt_z F + \pt_z H)^2 + ( \pt_z F G + F \pt_z G)^2 
 = F^2 (\pt_z G)^2
\end{align*}
 on the set $\C_+ \setminus N$. Since $F(z) \neq 0$ for $z \in \C_+ \setminus N$, we find that $\pt_z G \equiv 0$ on $\C_+ \setminus N$. Since $N$ only contains isolated points, we deduce that the meromorphic functions $\pt_z G$ extends to all of $\C_+$ with $\pt_z G \equiv 0$ on $\C_+$. But this implies $G(z)$ is a constant function on the simply connected domain $\C_+$.

In summary, we have shown that
\be
\xB = F \mathbf{V} 
\ee
for some constant vector $\mathbf{V} \in \C^3$ and some holomorphic function $F : \C_+ \to \C$. Recalling that $\pt_x \QB^e = 2\, \mathrm{Re} \, \xB$ and $\pt_y \QB^e = -2 \, \mathrm{Im} \, \xB$, we conclude 
\be
\pt_x \QB^e \wedge \pt_y \QB^e = -4 \, \mathrm{Re}  \, ( F \mathbf{V}) \wedge \mathrm{Im}  \, ( F \mathbf{V}) = - 4 |F|^2 \left ( \mathrm{Re} \, \mathbf{V} \wedge \mathrm{Im} \, \mathbf{V} \right ),
\ee
where the last identity follows from an elementary calculation. Thus the functions $\pt_x \QB^e$ and $\pt_y \QB^e$ belong to  a fixed plane in $\R^3$ orthogonal to the constant normal vector $\mathbf{N} = \mathrm{Re} \, \mathbf{V} \wedge \mathrm{Im} \, \mathbf{V}$. This completes the proof of Lemma \ref{lem:plane}.
\end{proof}

We are now able to achieve the goal of this subsection by proving the following symmetry result.

\begin{prop}[Reduction to Planar Case] \label{prop:planar}
Suppose $\QB_v \in \dot{H}^{1/2}(\R; \Ss^2)$ solves \eqref{eq:uv} with $|v| < 1$. Then there exists some rotation $R \in SO(3)$ such that
$$
\QB_v(x) = R \left ( \sqrt{1-v^2} f(x), \sqrt{1-v^2} g(x), \pm v \right )
$$
with some functions $f, g \in \dot{H}^{\frac 1 2}(\R; \R)$ satisfying $f^2 + g^2 = 1$ a.\,e.~in $\R$. In particular, the image $\QB_{v}(\R)$ belongs to a fixed plane in $\R^3$. Consequently, the map $x \mapsto \QB_v(x)$ traces out a circle on $\Ss^2$ with radius $r=\sqrt{1-v^2}$.
\end{prop}

\begin{remark*}
{\em The sign in $\pm v$ will be determined below depending on whether $f,g$ arise as boundary values of a holomorphic or an anti-holomorphic function, respectively.}
\end{remark*}

\begin{proof}
Let $v \in \R$ with $|v| < 1$ be given and assume that $\QB_v \in \dot{H}^{\frac 1 2}(\R; \Ss^2)$ solves \eqref{eq:uv}. Recall that we write $\QB = \QB_v$ for notational simplicity. Moreover, we can suppose that $\QB$ is not constant, because otherwise the assertion of Proposition \ref{prop:planar} readily follows. 

Let $\QB^e : \R_+^2 \to \R^3$ be the harmonic extension of $\QB$  given by \eqref{eq:Udef}. By Lemma \ref{lem:plane}, we find that the image $\QB^e(\R_+^2)$ belongs to a fixed plane $E \subset \R^3$. By rotational symmetry, we can assume that $E$ is parallel to the $x_1x_2$-plane, say. Thus after applying some rotation $R \in SO(3)$ we have 
\be
\QB^e(x,y) = ( \alpha(x,y), \beta(x,y), c) \quad \mbox{for $(x,y) \in \R_+^2$},
\ee 
with some constant $c \in \R$ and some functions $\alpha, \beta : \R_+^2 \to \R$. Since $\QB(x) = \QB^e(x,0)$ on $\pt \R_+^2$, we deduce that
$\QB(x) = ( \tilde{f}(x), \tilde{g}(x), c)$ with some functions $\tilde{f}, \tilde{g} \in \dot{H}^{\frac 1 2}(\R; \R)$ such that $\tilde{f}^2 + \tilde{g}^2 + c^2 = 1$  holds a.\,e.~in $\R$. In particular, we must have $|c| \leq 1$. If $c=\pm 1$, then $\tilde{f}=\tilde{g}=0$ and hence $\QB$ is a constant solution. Thus we can assume that $|c| < 1$ holds. By introducing the functions $f(x) := (1-c^2)^{-1/2} \tilde{f}(x)$ and $g(x) := (1-c^2)^{-1/2} \tilde{g}(x)$, we get
\be 
\QB(x) = \left ( \sqrt{1-c^2} f(x), \sqrt{1-c^2} g(x), c \right ) 
\ee
with some functions $f, g \in \dot{H}^{\frac 1 2}(\R;\R)$ such that $f^2 + g^2 =1$ a.\,e.~in $\R$. Recalling that $\pt_x \QB \perp |\nabla| \QB$ and $|\pt_x \QB| = | |\nabla| \QB| \neq 0$ for all $x \in \R$ by Lemma \ref{lem:rigid} and \ref{lem:Uz1}, we see that there exists a constant $\lambda = \pm 1$ such that
\be \label{eq:CR_hidden}
\pt_x f = \lambda |\nabla| g, \quad \pt_x g = - \lambda |\nabla| f .
\ee
On the other hand, from equation \eqref{eq:uv} we readily see that $c = -\lambda v = \mp v$ holds. This completes the proof of Proposition \ref{prop:planar}.\end{proof}

\subsection{Proof of Theorem \ref{thm:vsmall}}

Let $v \in \R$ with $|v| < 1$ be given. First, by direct calculation, we note any $\QB_v \in \dot{H}^{\frac 1 2}(\R; \Ss^2)$ as given in Theorem \ref{thm:vsmall} furnishes a solution of \eqref{eq:uv}. Thus it remains to prove the ``only if'' part in Theorem \ref{thm:vsmall}. Suppose that $\QB_v \in \dot{H}^{\frac 1 2}(\R; \Ss^2)$ solves \eqref{eq:uv}. From the symmetry result given in Proposition \ref{prop:planar} we deduce
\be \label{eq:QB_good}
\QB_v(x) = R \left ( \sqrt{1-v^2} f(x), \sqrt{1-v^2} g(x), s \cdot v \right )
\ee
with some rotation $R \in SO(3)$, some sign factor $s \in \{ \pm 1 \}$, and two real-valued functions $f,g \in \dot{H}^{\frac 1 2}(\R)$ such that $f^2(x) +g^2(x) = 1$ for a.\,e.~$x \in \R$. 

We are now ready to use the following classification result for half-harmonic maps from $\R$ to $\Ss^1$, whose proof goes back to Millot and Sire \cite{Millot-Sire-2015} and Berlyand et al.~\cite{BeMiRySa-14}; see also \cite{MiPi-04}. We provide a slightly alternative and self-contained proof, which does not involve the technical use of some results due to Brezis and Nirenberg \cite{BrNi-95} about VMO-spaces.

\begin{lemma} \label{lem:half_harmonic_S1}
Let $f, g \in \dot{H}^{\frac 1 2}(\R; \R)$ satisfy 
\be 
f |\nabla| g - g |\nabla| f = 0 \quad \mbox{and} \quad f^2 + g^2 = 1 \ \ \mbox{a.\,e.~on $\R$}.
\ee
Then $f(x) = \mathrm{Re} \, \mathcal{B}(x+\im 0 )$ and  $g(x) = \mathrm{Im} \, \mathcal{B}(x+ \im 0)$, where $\mathcal{B} : \overline{\C}_+ \to \C$ (or its complex conjugate $\overline{\mathcal{B}}$) is a finite Blaschke product given by
$$
\mathcal{B}(z) = e^{\im \theta} \prod_{k=1}^m  \frac{\lambda_k (z- a_k) - \im }{\lambda_k (z-a_k) + \im}, \quad \mbox{for $z \in \overline{\C}_+$},
$$ 
with some $\theta \in \R$, $m \in \N$, $\lambda_1, \ldots, \lambda_m \in \R_{>0}$, and $a_1, \ldots, a_m \in \R$. 

Furthermore, the energy of the half-harmonic map $\QB=(f,g,0) \in \dot{H}^{\frac 1 2}(\R; \Ss^2)$ satisfies
$$
E[\QB] = \frac{1}{2} \int_{\R} \QB \cdot |\nabla| \QB \, dx = \pi m .
$$
\end{lemma}

From Lemma \ref{lem:half_harmonic_S1} we conclude that $f(x) = \mathrm{Re} \, \mathcal{B}(x+\im 0)$ and $g(x) = \sigma \cdot \mathrm{Im} \, \mathcal{B}(x+\im 0)$ with some finite Blaschke product $\mathcal{B}(z)$ as above and some sign factor $\sigma \in \{\pm 1 \}$. Suppose now that $\sigma=+1$. Then the Cauchy--Riemann equations satisfied by $\mathcal{B}(z)$ together with $|\nabla| f= -\pt_y \mathrm{Re} \, \mathcal{B}(x+\im y) |_{y=0}$ and $|\nabla| g = -\pt_y \mathrm{Im} \, \mathcal{B}(x+ \im y) |_{y=0}$ imply that $\pt_x f = -|\nabla| g$ and $|\nabla| f = \pt_x g$. Thus we see that $\lambda=+1$ holds in \eqref{eq:CR_hidden} above and hence we must have $s \cdot v = +v$ in \eqref{eq:QB_good} if $g(x) = \mathrm{Im} \, \mathcal{B}(x+\im 0)$. Likewise, we see that $s \cdot v = -v$ in the case $g(x) = -\mathrm{Im} \, \mathcal{B}(x+\im 0)$. In summary, we have shown that the sign factor $s \in \{ \pm 1\}$ in \eqref{eq:QB_good} is consistent with $g(x) = s \cdot \mathrm{Im} \, \mathcal{B}(x+\im 0)$.

It remains to give the proof of Lemma \ref{lem:half_harmonic_S1}.

\begin{proof}[Proof of Lemma \ref{lem:half_harmonic_S1}] We divide the proof into the following steps.

\subsubsection*{Step 1} Using that $z = x + \im y $, we define $F : \C_+ \to \C$ to be the Poisson harmonic extension of the complex-valued function $f + \im g : \R \to \C$. That is, we set
$$
F(z) = (P_y \ast f)(x) + \im (P_y \ast g)(x) \quad \mbox{for $z = x + \im y \in \C_+$}.
$$
The harmonicity of $F$ implies that function $\phi : \C_+ \to \C$ defined as
$$
\phi(z) := (\pt_z F )(\pt_z \overline{F}) = ( | \pt_x F|^2 - |\pt_y F|^2 ) - 2 \im \, \mathrm{Re} \, ( \pt_x \overline{F} \pt_y F)  
$$ 
is holomorphic on $\C_+$. A straightforward variation of the proof of Lemma \ref{lem:Uz1} yields that $\phi(z) \equiv 0$. Thus $\pt_z F \equiv 0$ or $\pt_z \overline{F} \equiv 0$ and hence $F$ is holomorphic or anti-holomorphic.  For the rest of the proof, let us suppose that $F$ is holomorphic (by possibly replacing $F$ with $\overline{F}$). 

Because of $|F(z)| \leq 1$ for all $z \in \C_+$, the function $F$ belongs to the Hardy space $\HH^\infty(\C_+)$ of bounded holomorphic functions on $\C_+$. Since moreover $F \in \mathcal{H}^\infty(\C_+)$ with $|F(z)|^2 = 1$ for $z \in \pt \C_+$, an application of the canonical factorization theorem in the Hardy space $\mathcal{H}^\infty(\C_+)$ (see, e.\,g.~\cite[Chapter 13]{Ma-09}) yields that
\be \label{eq:Blaschke_factor}
F(z) = \lambda B(z) e^{\im \alpha z}  \exp \left \{ - \frac{\im}{\pi} \int_{\R} \left ( \frac{1}{z-t} + \frac{t}{1+t^2} \right ) d \sigma(t) \right \}   .
\ee
Here $\lambda \in \C$ with $|\lambda|=1$, $\alpha \geq 0$, $B(z)$ is a Blaschke product on having the same zeros as $F(z)$ on $\C_+$, and $\sigma$ is positive singular Borel measure on $\R$ satisfying $\int_{\R} \frac{d \sigma(t)}{1+t^2} dt < +\infty$. Furthermore, let $\mu_k \in \C_+$ denotes the zeros of $F(z)$. Then $B(z)$ is given by
\be
B(z) = \prod_{k} e^{\im \alpha_k} \frac{z- \mu_k}{z - \overline{\mu}_k} ,
\ee 
where $\alpha_k \in \R$ are real numbers chosen such that $e^{\im \alpha_k} =   \left | \frac{\im-\mu_k}{\im - \overline{\mu}_k} \right | / \left ( \frac{\im-\mu_k}{\im - \overline{\mu}_k} \right )$. Furthermore, it is a well-known fact that $F \in \mathcal{H}^\infty(\C_+)$ implies the summability condition
\be \label{eq:Blaschke_sum}
\sum_k \frac{\mathrm{Im} \, \mu_k}{ | \im + \mu_k |^2} < +\infty,
\ee
which in turn implies that the product in $B(z)$ converges uniformly on compact sets in $\C_+$. Of course, if $F(z)$ has only finitely many zeros $\mu_k$ then the condition \eqref{eq:Blaschke_sum} is trivially satisfied. In the next step, we will show that the singular measure $\sigma \equiv 0$, $\alpha=0$, and that $B(z)$ is a finite Blaschke product in \eqref{eq:Blaschke_factor}.

\subsubsection*{Step 2} We start by showing the triviality of the singular measure, i.\,e.,
\be
\sigma \equiv 0 
\ee
holds in \eqref{eq:Blaschke_factor}. In fact, the vanishing of the singular measure $\sigma$ follows from the regularity properties of $f$ and $g$ combined with some standard arguments (see also \cite{Po-11} for a similar argument for the traveling solitary waves of the cubic Szeg\"o equation on $\R$).  Indeed, we recall the global Lipschitz estimates $\| \pt_x f \|_{L^\infty} \lesssim 1$ and $\| \pt_x g \|_{L^\infty} \lesssim 1$, whence the function $F(x+\im 0)=f(x)+ \im g(x)$ is uniformly continuous in  $x \in \R$. As consequence of Poisson's extension formula, we deduce that $F(x+\im \eps) \to F(x,0) = f(x) + \im g(x)$ as $\eps \to 0^+$ uniformly in $x \in \R$ . Since $|F(x+\im 0)|^2 = 1$ for $x \in \R$, this implies that $F(z)$ has no zeros inside some strip $\{ z \in \C : 0 \leq \mathrm{Im}\, z \leq \eps_0 \}$ for some  $\eps_0 > 0$. Consequently, the Blaschke product $B(z)$ has all its zeros $\mu_k$ satisfying $\mathrm{Im} \, \mu_k > \eps_0$ for any $k$. Furthermore, we conclude that $| F(x+\im \eps)/B(x+\im \eps)| \to |F(x)/B(x)| =1$ as $\eps \to 0^+$ uniformly in $x$ on compact subsets in $\R$.  Therefore, we conclude from \eqref{eq:Blaschke_factor} that
\begin{align*}
&  \mathrm{Re} \left \{ - \frac{\im}{\pi} \int_{\R} \left ( \frac{1}{(x+ \im \eps) -t} + \frac{t}{1+t^2} \right ) d \sigma(t) \right \} = \frac{1}{\pi} \int_{\R} \frac{\eps}{(x-t)^2 + \eps^2} d \sigma(t) \\
& = \int_{\R} P_\eps(x-t) \, d \sigma(t) \to 0 \ \ \mbox{as} \ \ \eps \to 0^+,
\end{align*}
 uniformly in $x$ on compact sets in $\R$. Now, let $\phi \in C_c(\R)$ be a continuous function on $\R$ with compact support. It is elementary to check that its Poisson extension satisfies $|(P_y \ast \phi)(x)| \leq C/(1+x^2)$ for all $x \in \R$ and $0 < y < 1$, where $C=C(\phi)$ is some constant. Thus, by the dominated convergence theorem and the fact that $\int_{\R} \frac{d \sigma(t)}{1+t^2} < +\infty$, we deduce that
$$
\int_{\R^2} \phi(x)  P_\eps(x-t) \, d \sigma(t) \, dt \, dx = \int_{\R} (P_\eps \ast \phi)(t) \, d \sigma(t) \to \int_{\R} \phi(t) d \sigma(t) \ \ \mbox{as} \ \ \eps \to 0^+.
$$
On the other hand, we deduce from above that $\int_{\R^2} \phi(x) P_\eps(x-t) \, d \sigma(t) \to 0$ as $\eps \to 0^+$. Thus we conclude that $\int_{\R} \phi(t) d \sigma(t) = 0$ for all $\phi \in C_c(\R)$. Hence $\sigma \equiv 0$ holds.

Next, we show that $\alpha=0$ holds and that $B(z)$ is a finite Blaschke product in \eqref{eq:Blaschke_factor}. To prove this claim, we use the finite energy condition
\be \label{ineq:finite_energy}
- \int_{\R} \overline{F} \pt_{y} F \big |_{y=0} \, dx = \int_{\R} ( f |\nabla| f + g |\nabla| g) \, dx < +\infty.
\ee
Since $F(z) = \lambda e^{\im \alpha z} B(z)$ with $|\lambda|= 1$, we find
\be \label{eq:Blaschke_d}
-\overline{F}(z) \pt_y F(z) \Big |_{y=0} = \alpha - \frac{B'(z)}{B(z)} \Big |_{y=0},
\ee
where we used that $|B(z)|^2 = 1$ and $\overline{B}(z) = 1/B(z)$ for $z \in \pt \C_+$, as well as $\pt_y B(z) = -\im B'(z)$ by the Cauchy--Riemann equations. Since we readily check that $B'(z)/B(z) |_{y=0}$ tends to zero as $|x| \to \infty$, we conclude from \eqref{eq:Blaschke_d} and the finite energy condition \eqref{ineq:finite_energy} that $\alpha=0$ must hold. Furthermore, by the residue theorem, we obtain
\be
\frac{1}{\im} \int_{\R} \frac{B'(z)}{B(z)} \Big |_{y=0} \, dx = 2 \sum_k \int_{\R} \frac{\mathrm{Im} \, \mu_k}{|x-\mu_k|^2} \, dx = 2 \sum_k \pi ,
\ee
which is finite if and only if $B(z)$ has only finitely many zeros $\mu_k$, i.\,e., it is a finite Blaschke product $B(z) = \prod_{k=1}^m e^{\im \alpha_k} (z-\mu_k)/(z- \overline{\mu}_k)$ with some integer $m \in \N$ and $\mu_1, \ldots, \mu_k \in \C_+$. [The case $m=0$ corresponds to the trivial case $B(z) \equiv 1$.] In addition, we conclude that $\QB=(f,g,0) \in \dot{H}^{\frac 1 2}(\R; \Ss^2)$ has the energy
\be
E[\QB] = \frac{1}{2} \int_{\R} \QB \cdot |\nabla| \QB \, dx = \pi m .
\ee
Finally, it is straightforward to check that we obtain 
\be
F(z) = \lambda \prod_{k=1}^m e^{\im \alpha_k} \frac{z-\mu_k}{z- \overline{\mu}_k} = e^{\im \theta} \prod_{k=1}^m \frac{\lambda_k(z- a_k) - \im}{\lambda_k(z - a_k) + \im }
\ee
for a suitable choice of the parameters $\theta \in \R$, $\lambda_1, \ldots, \lambda_m \in \R_{>0}$, and $a_1, \ldots, a_m \in \R$.

This completes the proof of Lemma \ref{lem:half_harmonic_S1} and the proof of Theorem \ref{thm:vsmall} as well.
\end{proof}

\subsection{Proof of Theorem \ref{thm:vlarge}}

Let $v \in \R$ with $|v| \geq 1$ and suppose that $\QB_v \in \dot{H}^{\frac 1 2}(\R; \Ss^2)$ solves \eqref{eq:uv}. For simplicity, we write $\QB \equiv \QB_v$ in the following. Furthermore, let $\QB^e : \R_+^2 \to \R^3$ denote the harmonic extension of $\QB : \R \to \R^3$ into the upper halfplane $\R_+^2$ by the Poisson formula \eqref{eq:Udef}. We emphasize the fact that the following calculations are already well-defined for $\QB \in \dot{H}^{\frac 1 2}$ and hence $\QB^e \in \dot{H}^1$. No further higher regularity (which is also not our disposal when $|v| \geq 1$) is required.

Let $\HH$ denote the one-dimensional  Hilbert transform with respect to $x \in \R$, where we recall our sign convention for $\HH$ such that $\HH \pt_x  = |\nabla|$ holds.  If we take the scalar product of \eqref{eq:uv} with $\HH (\QB)$, we obtain
\begin{align*}
 0 & = -v \pt_x \QB \cdot \HH(\QB)+(\QB \wedge |\nabla| \QB) \cdot \HH(\QB)  \\
 & = \int_{y=0}^\infty \pt_y \left \{ v \pt_x \QB^e \cdot \HH(\QB^e) + (\QB^e \wedge \pt_y \QB^e) \cdot \HH(\QB^e) \right \} dy .
\end{align*}
Note here that we also used that the Poisson extension commutes with taking the Hilbert transform, i.\,e., $\HH (P_y \ast \QB) = (P_y \ast \HH(\QB))$. Next, by integrating the identity above over $x$, we find
\begin{align*}
0  & = \int \! \! \int_{\R^2_+} \pt_y \left \{ v \pt_x \QB^e \cdot \HH(\QB^e) +  ( \QB^e  \wedge \pt_y \QB^e)  \cdot \HH(\QB^e) \right \} dx \, dy  \\
& = v \int \! \! \int_{\R^2_+} \left ( \pt_x \pt_y \QB^e \cdot \HH(\QB^e) + \pt_x \QB^e \cdot \pt_y \HH(\QB^e) \right )  dx \, dy \\
& \quad + \int \! \! \int_{\R^2_+} \left ( ( \QB^e \wedge \pt_{yy} \QB^e ) \cdot \HH(\QB^e) +( \QB^e \wedge \pt_y \QB^e ) \cdot \pt_y \HH(\QB^e) \right ) dx \, dy \\
& =: I + II  .
\end{align*} 
To analyze term $I$, we make use of the identities $\pt_y \HH (\QB^e) =  \pt_x \QB^e$ and $\pt_x \HH (\QB^e) = - \pt_y \QB^e$ and we integrate by parts in $x$ to find that
$$
I = v \int \! \! \int_{\R^2_+} \left (  |\pt_x \QB^e|^2 + |\pt_y \QB^e|^2 \right ) dx \, dy.
$$
As for $II$, we use that $\pt_{yy} \QB^e = - \pt_{xx} \QB^e$ (because $\QB^e$ is harmonic) and integrate by parts in $x$, which gives us
\begin{align*}
II & = - \int \! \! \int_{\R^2_+} ( \QB^e \wedge \pt_{xx} \QB^e ) \cdot \HH(\QB) \, dx \, dy + \int \! \! \int_{\R^2_+} ( \QB^e \wedge \pt_y \QB^e ) \cdot  \pt_x \QB^e \, dx \, dy \\
& = - \int \! \! \int_{\R^2_+} ( \QB^e \wedge \pt_x \QB^e) \cdot \pt_y \QB^e \, dx \,d y +  \int \! \! \int_{\R^2_+} ( \QB^e \wedge \pt_y \QB^e ) \cdot  \pt_x \QB^e \, dx \, dy \\
& = -2 \int \! \! \int_{\R^2_+} ( \QB^e \wedge \pt_x \QB^e ) \cdot  \pt_y \QB \, dx \, dy,
\end{align*}
using that $\pt_x \HH (\QB^e) = - \pt_y \QB^e$ once again. Next, by recalling that $I+II=0$ holds, we conclude that
\be \label{eq:IeqII}
\int \! \! \int_{\R^2_+} \left ( |\pt_x \QB^e|^2 + |\pt_y \QB^e|^2 \right ) dx \, dy  = \frac 2 v \int \! \! \int_{\R^2_+} (\QB^e \wedge \pt_x \QB^e) \cdot \pt_y \QB^e \, dx \, dy .
\ee
Since $\| \QB^e \|_{L^\infty(\R_+^2)} \leq 1$  by the maximum principle,  the Cauchy--Schwarz inequality yields
\be \label{ineq:CS}
 2 \left | ( \QB^e \wedge \pt_x \QB^e) \cdot \pt_y \QB^e \right | \leq  |\pt_x \QB^e|^2 + |\pt_y \QB^e|^2  .
\ee 
Thus if $|v| > 1$, we deduce from \eqref{eq:IeqII} that $\| \QB^e \|_{\dot{H}^1(\R_+^2)} = 0$ holds, which implies that $\QB \in \dot{H}^{\frac 1 2}(\R;\Ss^2)$ is a constant. 

Finally, we consider the borderline case $|v|=1$. Then equality must hold in \eqref{ineq:CS} almost everywhere, which is possible only if $\QB^e \wedge \pt_x \QB^e$ is collinear to $\pt_y \QB^e$ almost everywhere, i.\,e., we must have $\pt_y \QB^e \cdot \QB^e = 0$ a.\,e.~on $\R_+^2$. But by using the harmonicity of $\QB^e$ and integrating by parts in $x$, this implies that
\begin{align*}
0 & = \int \! \! \int_{\R^2_+} \pt_y (  \pt_y \QB^e \cdot \QB^e) \, dx \,dy  = \int \! \! \int_{\R^2_+} \left ( -\pt_{xx} \QB^e \cdot \QB^e + \pt_y \QB^e \cdot \pt_y  \QB^e \right )  dx \, dy \\
& = \int \! \! \int_{\R^2_+} \left (  |\pt_x \QB^e|^2 + |\pt_y \QB^e|^2 \right ) dx \, dy.
\end{align*}
Again, we conclude that $\| \QB^e \|_{\dot{H}^1(\R_+^2)}=0$ holds, which shows that $\QB \in \dot{H}^{\frac 1 2}(\R; \Ss^2)$ must be constant.  The proof of Theorem of \ref{thm:vlarge} is now complete. \hfill $\Box$

\section{Analysis of the Linearized Operator}

\label{sec:linearize}

\subsection{Preliminaries}
We now investigate the linearization around a traveling solitary wave given by Theorem \ref{thm:vsmall}, where we consider solutions $\QB_v \in \dot{H}^{\frac 1 2}(\R; \Ss^2)$ of \eqref{eq:uv} with vanishing velocity $v= 0$. For notational convenience, we shall write $\QB(x) = \QB_{v=0}(x)$ and, by Theorem \ref{thm:vsmall}, we can assume without loss of generality that $\QB : \R \to \Ss^2$ is half-harmonic map of the form
\be \label{eq:QB_equator}
\QB(x) = (f(x), g(x), 0) .
\ee  
Now let $\hh \in \dot{H}^{\frac 1 2}(\R; \Ss^2)$ be such that $|\QB + \hh|^2 = 1$ a.\,e. In view of \eqref{eq:QB_equator}, it is convenient to introduce the orthonormal frame $\{ \eB, J \eB, \QB \}$ in $\R^3$ with $\eB = \eB_z = (0,0,1)$ and $J \eB = \QB \wedge \eB$. With respect to this frame, we let $(h_1, h_2, h_3)$ denote the components of $\hh$, i.\,e., we have
\be \label{eq:hbase}
\hh = h_1 \eB + h_2 J \eB + h_3 \QB. 
\ee
The constraint $|\QB + \hh|^2 = 1$ a.\,e.~implies that $-2 h_3 = h_1^2 + h_2^2 + o(\hh^2)$ and, in particular, we have $h_3 = O(\hh^2)$. 

By expanding the energy functional, we obtain the following result.   

\begin{prop} \label{prop:Hessian}
For $\QB$ and $\hh$ as above, it holds that
$$
E[\QB+ \hh] = E[\QB] + \frac{1}{2} \left [ ( h_1, L_+ h_1 ) + ( h_2, L_- h_2 ) \right ] + o(\hh^2),
$$
where the operators $L_+$ and $L_-$ are given by
$$
L_+  = |\nabla|  - |\pt_x \QB|  \quad \mbox{and} \quad  L_- f = |\nabla|  - |\pt_x \QB|  + R ,
$$
with the integral operator
$$
 (Rf)(x) = \frac{1}{ 2 \pi} \int_{\R} \frac{|\QB(x)-\QB(y)|^2}{|x-y|^2} f(y) \, dy.
$$
\end{prop}

\begin{proof}
First, we expand the energy functional to find that
\begin{align*}
E[\QB + \hh] & = E[\QB] +  \int_{\R} \hh \cdot |\nabla| \QB \, dx + \frac 1 2 \int_{\R} \hh \cdot |\nabla| \hh \, dx \\
 & = E[\QB] - \frac 1 2 \int_{\R}  |\pt_x \QB| (h_1^2 + h_2^2) \, dx + \frac 1 2 \int_{\R} \hh \cdot |\nabla| \hh \, dx + o(\hh^2),
\end{align*}
where we used that $|\nabla| \QB = |\pt_x \QB| \QB$ (see Lemma \ref{lem:upoint} with $v=0$) and $-2h_3 = h_1^2 + h_2^2+ o(\hh^2)$. Now since $\hh = h_1 \eB + h_2 J \eB + h_3 \QB$ with $h_3 = O(\hh^2)$, we see 
$$
\int_{\R} \hh \cdot |\nabla| \hh \, dx = \int_{\R} (h_1 \eB + h_2 J \eB) \cdot |\nabla| (h_1 \eB + h_2 J \eB ) \, dx + o(\hh^2).
$$
Since $\eB = \eB_z$ is constant, we clearly have that $|\nabla| (h_1 \eB) = (|\nabla| h_1) \eB$. As for $h_2 J \eB$ with $J \eB = \QB \wedge \eB$, we note that
\be \label{eq:JeB}
|\nabla| ( h_2 J \eB )(x) =  ( |\nabla| h_2(x) ) J \eB(x) + \frac{1}{\pi} \int_{\R}  \frac{(J \eB(x)- J\eB(y))}{|x-y|^2} h_2(y) \, dy ,
\ee
which follows from a simple calculation using the singular integral expression for $|\nabla|$. Next, by using that $\eB \cdot J \eB = 0$ and 
$J \eB(x) \cdot ( J \eB(x) - J \eB(y) ) = 1 - \QB(x) \cdot \QB(y) = \frac{1}{2} |\QB(x) - \QB(y)|^2$, we conclude that
$$
\int_{\R} (h_1 \eB + h_2 J \eB) \cdot |\nabla| (h_1 \eB + h_2  J \eB) \, dx = \int_{\R} h_1 |\nabla| h_1 \, dx + \int_{\R} h_2 |\nabla| h_2 \, dx + D[h_2,h_2],
$$
with the quadratic form
$$
D[f,f] = \frac{1}{2 \pi} \int \! \! \int_{\R \times \R}  \frac{|\QB(x)-\QB(y)|^2}{|x-y|^2} f(x) f(y) \, dx \, dy.
$$
By gathering the identities shown above and defining $L_+$, $L_-$ and $R$ as stated in Proposition \ref{prop:Hessian} above, we complete the proof.
\end{proof}

Next, we show that the pair of operators $L_+$  and $L_-$ from above also arises in the linearization of half-wave maps equation \eqref{eq:hwm} around a given half-harmonic map $\QB$. Indeed, we suppose that $\uu(t,x) = \QB + \hh(t,x)$ solves \eqref{eq:hwm} with $\QB$ and $\hh = h_1 \eB + h_2 J \eB + h_3 \QB$  as above. We claim that the components of the tangential part $h_1 \eB  + h_2  J \eB \in T_{\QB} \Ss^2$ solve the equation  
\be \label{eq:linearized_hwm}
\pt_t \left [ \begin{array}{c} h_1 \\ h_2 \end{array} \right ]  = J L   \left [ \begin{array}{c} h_1 \\ h_2 \end{array} \right ] + O(\hh^2) , \quad \mbox{with} \quad  JL = \left [ \begin{array}{cc} 0 & -L_- \\ L_+ & 0 \end{array} \right ],
\ee
where $O(\hh^2)$ stands for quadratic terms in $\hh$ and $|\nabla| \hh$. To see this, we use  $\QB \wedge |\nabla| \QB = 0$ as well as $|\nabla| \QB = |\pt_x \QB| \QB$ (see Lemma \ref{lem:upoint} with $v=0$) to find
\be \label{eq:linear_hwm}
\pt_t \hh = \QB \wedge |\nabla| \hh + \hh \wedge |\nabla| \QB + O(\hh^2)  = \QB \wedge \left ( |\nabla| \hh- |\pt_x \QB| \hh \right ) + O(\hh^2) .
\ee
Next we recall that $h_3 = O(\hh^2)$ and we deduce
\be
 |\nabla| \hh  =  (|\nabla| h_1) \eB + (|\nabla| h_2) J \eB + \frac{1}{\pi} \int_{\R} \frac{(J \eB(x)- J \eB(y))}{(x-y)^2} h_2(y) \,dy + O(\hh^2).
\ee
Since $\QB(x) \wedge J\QB(x) = \QB(x) \wedge (\QB(x) \wedge \eB) = -\eB$ and $\QB(x) \wedge (J \eB(x) - J \eB(y)) =  -(1- \QB(x) \cdot \QB(y)) \eB = -\frac{1}{2} |\QB(x)-\QB(y)|^2 \eB$, we get
\be \label{eq:linear1}
\QB \wedge |\nabla| \hh = ( |\nabla| h_1)  J \eB - (|\nabla| h_2 + R h_2) \eB + O(\hh^2),
\ee
with the operator $R$ defined as in Proposition \ref{prop:Hessian} above. Furthermore, we readily find that
\be \label{eq:linear2}
\QB \wedge |\pt_x \QB| \hh = ( |\pt_x \QB| h_1) J \eB   - (|\pt_x \QB| h_2) \eB + O (\hh^2) .
\ee
From \eqref{eq:linear1} and \eqref{eq:linear2} we now see that \eqref{eq:linear_hwm} holds.

\subsection{Stereographic Lifting to $\Ss^1$}
As a next preliminary step for the spectral analysis, we introduce $\Pi$ as the stereographic projection of the unit circle $\Ss^1 = \{ z \in \C : | z | = 1\}$ without the `north pole' $z=\im$ onto $\R$, i.\,e., we set
\be
\Pi : \Ss^1 \setminus \{ \im \} \to \R, \quad e^{\im \theta} \mapsto \frac{\cos \theta}{1- \sin \theta}.
\ee
As usual, we extend $\Pi$ to all of $\Ss^1$ by setting $\Pi(\im) = \infty$, so that $\Pi$ becomes a homeomorphism from $\Ss^1$ to the real projective line $\hat{\R} = \R \cup \{ \infty \}$. 

For later use, we record some useful identities concerning the stereographic projection as follows. If we let $x = \Pi(e^{\im \theta})$, we obtain the formulas
\be \label{eq:stereo1}
\cos \theta = \frac{2x}{1+x^2}, \quad \sin \theta = \frac{x^2-1}{x^2+1},
\quad \frac{2}{1+x^2} = 1- \sin \theta, \quad \frac{2}{1+x^2} dx = d \theta .
\ee
Furthermore, a simple calculation shows that
\be \label{eq:stereo3}
2- 2 \cos ( \theta - \omega) = |e^{\im \theta} - e^{\im \omega}|^2 = \frac{2}{1+x^2} |x-y|^2 \frac{2}{1+y^2}
\ee
when $x = \Pi(e^{\im \theta})$ and $y = \Pi(e^{\im \omega})$.  As a direct consequence, we deduce the classical fact
\begin{align} 
\| \vphi \|_{\dot{H}^{\frac 1 2}(\R)}^2 & = \frac{1}{2 \pi} \int \! \! \int_{\R \times \R} \frac{ |\vphi(x) - \vphi(y)|^2}{|x-y|^2} \, dx \, dy \nonumber  \\
& =  \frac{1}{2 \pi} \int \! \! \int_{\Ss^1 \times \Ss^1} \frac{|\tvphi(e^{\im \theta}) - \tvphi(e^{\im \omega})|^2}{2- 2 \cos (\theta - \omega)} \, d\theta \, d \omega  =  \| \tvphi \|_{\dot{H}^{\frac 1 2}(\Ss^1)}^2.  \label{eq:master}
\end{align}
for $\vphi \in \dot{H}^{\frac 1 2}(\R)$ and $\tvphi = \vphi \circ \Pi : \Ss^1 \to \R$; in particular, we see that $\vphi \in \dot{H}^{\frac 1 2}(\R)$ if and only if $\tvphi \in \dot{H}^{\frac 1 2}(\Ss^1)$. Indeed, by polarization, we obtain the following ``transformation rule'' between $|\nabla|_{\Ss^1}$ and $|\nabla|$ on the level of $\dot{H}^{\frac 1 2}$-functions.

\begin{lemma} \label{lem:pol}
For $\vphi \in \dot{H}^{\frac 1 2}(\R)$ and $\tvphi = \vphi \circ \Pi \in \dot{H}^{\frac 1 2}(\Ss^1)$, it holds that 
$$
(|\nabla|_{\Ss^1} \tvphi)(e^{\im \theta}) =  \frac{(|\nabla| \vphi)(\Pi(e^{\im \theta}))}{1 - \sin \theta} \quad \mbox{in $\mathcal{D}'(\Ss^1)$}. 
$$
\end{lemma}

\begin{remark*}
{\em If we were to relax the assumption to $\vphi \in L_{1/2}(\R)$ or, equivalently, we assume $\tvphi \in L^1(\Ss^1)$ for the conformal lifting to $\Ss^1$, the right-hand side of the stated formula contains a Dirac mass located at $z=\im$ given by $-c \delta_{\im}$ with $c = \int_{\R} |\nabla| \vphi \, dx$, provided we assume that $|\nabla| \vphi \in L^1(\R)$; see \cite{DaMaRi-15}. But we do not need this generalization here, since we always work on the $\dot{H}^{\frac 1 2}$-level, where no Dirac mass term is present.} 
\end{remark*}

\begin{proof}
Suppose that $\vphi, \psi \in \dot{H}^{\frac 1 2}(\R)$ and let $\tvphi = \vphi \circ \Pi$ and $\widetilde{\psi} = \psi \circ \Pi$. By polarization of the quadratic identity \eqref{eq:master}, we deduce
\begin{align*}
(\psi, |\nabla| \vphi) & = \frac{1}{4} \left [ (\psi+\vphi, |\nabla| (\psi+\vphi)) - (\psi- \vphi, |\nabla| ( \psi- \vphi) ) \right ] \\
& = \frac{1}{4}  \left [ \langle \widetilde{\psi}+\tvphi, |\nabla|_{\Ss^1} (\widetilde{\psi}+\tvphi) \rangle - \langle \widetilde{\psi}- \tvphi, |\nabla|_{\Ss^1} ( \widetilde{\psi}- \tvphi) \rangle \right ] =  \langle \widetilde{\psi}, |\nabla|_{\Ss^1} \tvphi  \rangle ,
\end{align*}
where $\langle \widetilde{f}, \widetilde{g} \rangle = \int_{\Ss^1} \widetilde{f}(e^{\im \theta}) \widetilde{g}(e^{\im \theta}) \, d \theta$ denotes the inner product for real-valued functions $\widetilde{f}, \widetilde{g} \in L^2(\Ss^1)$. Since $\int_{\R} f(x) g(x) \, dx = \int_{\Ss^1} \widetilde{f}(e^{\im \theta}) \widetilde{g}(e^{\im \theta}) (1-\sin \theta) \, d\theta$ when $\widetilde{f} = f \circ \Pi$ and $\widetilde{g} = g \circ \Pi$, we deduce that the asserted identity holds if tested against any function in $\dot{H}^{\frac 1 2}(\Ss^1)$. Hence the desired identity holds in $\mathcal{D}'(\Ss^1)$. 
\end{proof}

Finally, we remark that the lift of the half-harmonic maps $\QBm$ to $\Ss^1$ is found to be
\be \label{eq:QBm_lift}
\widetilde{\QB}_m(e^{\im \theta}) := (\QB_m \circ \Pi)(e^{\im \theta}) = \left ( \cos (m \theta), \sin (m \theta), 0 \right ) = (e^{\im m \theta}, 0),
\ee  
which follows from $\im \left ( \frac{x-i}{x+i} \right ) = \cos \theta + \im \sin \theta$ for $x = \Pi(e^{\im \theta}) = \frac{\cos \theta}{1- \sin \theta}$. Moreover, we recall the well-known identity 
\be \label{eq:Cheby}
 \cos( m \theta) + \im \sin (m \theta) = T_m (\cos \theta) + \im (\sin \theta) U_{m-1}(\cos \theta),
\ee
where $T_m$ and $U_m$ denote the {\em Chebyshev polynomials} of the  first and second kind, respectively. From \eqref{eq:stereo1} we deduce the useful formulas
\be \label{eq:TmUm}
\left \{ \begin{array}{ll}
\displaystyle f_k(x) := \mathrm{Re} \left ( \im \cdot \frac{x-\im }{x+\im} \right)^k & \displaystyle = T_k \left ( \frac{2x}{1+x^2} \right ) \quad \mbox{for $k \in \N$}, \\[2ex]
\displaystyle g_k(x) := \mathrm{Im} \left ( \im \cdot \frac{x-\im }{x+\im} \right )^k & \displaystyle = \left ( \frac{x^2-1}{1+x^2} \right ) U_{k-1} \left ( \frac{2x}{1+x^2} \right ) \quad \mbox{for $k \in \N$, $k \geq 1$}.
 \end{array} \right .
\ee
In particular, we have that $\QB_m(x) = (f_m(x), g_m(x), 0 )$. We now record the following simple calculational observations. 

\begin{prop} The linearized operators $L_+$ and $L_-$ around $\QBm$ are given by
$$
L_+ = |\nabla| - \frac{2m}{1+x^2}, \quad L_+ = |\nabla| - \frac{2m}{1+x^2} + R_m,
$$
where $R_m$ denotes the integral operator with
$$
(R_m \vphi)(x) = \frac{1}{2 \pi} \int_{\R} \frac{|\QBm(x)- \QBm(y)|^2}{|x-y|^2} \vphi(y) \, dy. 
$$
\end{prop}

\begin{proof}
In view of Proposition \ref{prop:Hessian}, we only need to verify that
$$
|\pt_x \QBm(x)| = \sqrt{(\pt_x f_m(x))^2 + (\pt_x g_m(x))^2} = \frac{2m}{1+x^2}.
$$
Indeed, this can be checked by a direct calculation using the explicit form of $f_m$ and $g_m$. A more elegant alternative is to note that $\widetilde{\QB}_m = \QBm \circ \Pi = (e^{i m \theta}, 0)$. Since $e^{\im \theta} = \Pi^{-1}(x)$ implies $\frac{\pt \theta}{\pt x} = \frac{2}{1+x^2}$, we find that $|\pt_x \QB_m| =  |\frac{\pt}{\pt \theta} ( e^{\im m \theta})| | \frac{\pt \theta}{\pt x} |  = \frac{2m}{1+x^2}$. \end{proof}

\section{Nullspace, Resonances, and Nondegeneracy}

\label{sec:nullspace}

This section is devoted to the proof of Theorem \ref{thm:null}, which will directly follow from summarizing the results shown next. 

Recall that we always assume that $\QBm=(f_m, g_m,0) \in \dot{H}^{\frac 1 2}(\R; \Ss^2)$ is a pure-power half-harmonic map of the form \eqref{eq:QBm} with degree $m \geq 1$. We are now going to explicitly determine the nullspaces of $L_+$ and $L_-$ in $\dot{H}^{\frac{1}{2}}(\R)$, which we denote as
\be
\mathcal{N}(L_+) = \left \{ \vphi \in \dot{H}^{\frac 1 2}(\R) : L_+ \vphi =  0 \right \}, \quad \mathcal{N}(L_-) = \left \{ \vphi \in \dot{H}^{\frac 1 2}(\R) : L_- \vphi = 0 \right \}.
\ee
We start with $\mathcal{N}(L_+)$, which is the simpler to analyze than $\mathcal{N}(L_-)$.

\begin{prop}[Nondegeneracy of $L_+$] \label{prop:null_Lp}
The nullspace of $L_+$ is given by
$$
\mathcal{N}(L_+) = \mathrm{span} \left \{ f_m, g_m \right \} .
$$
In particular, it holds that $\dim \mathcal{N}(L_+) = 2$.
\end{prop}

\begin{remarks*}
{\em 1) The fact that the functions $f_m$ and $g_m$ lie in the nullspace of $L_+$ is due to the {\em rotational symmetry} of the half-harmonic maps equation. More precisely, recalling the equatorial form $\QBm=(f_m, g_m,0)$ and considering infinitesimal rotations around the the $x$- and $y$-axis yield the zero modes for $L_+$ given by $\QBm \wedge \eB_x = -g_m$ and $\QBm \wedge \eB_y = f_m$, respectively. The above lemma then shows that we have {\bf nondegeneracy} of $L_+$, as further (accidental) zero modes are ruled out for this operator.

2) From \eqref{eq:TmUm} and basic properties of Chebyshev polynomials, we deduce that
$$
\lim_{|x| \to \infty} |f_m(x)| = \begin{dcases*} 0 & \mbox{if $m$ is odd,} \\ 1 & \mbox{if $m$ is even,} \end{dcases*} \quad \lim_{|x| \to \infty} |g_m(x)| = \begin{dcases*} 1 & \mbox{if $m$ is odd,} \\ 0 & \mbox{if $m$ is even.} \end{dcases*}  
$$
We then easily conclude that:
\begin{itemize}
\item $f_m \in \dot{H}^{\frac 1 2}(\R) \setminus L^2(\R)$ is a {\bf resonance} for $L_+$ if and only if $m$ is even.
\item $g_m \in \dot{H}^{\frac 1 2}(\R) \setminus L^2(\R)$ is a {\bf resonance} for $L_+$ if and only if $m$ is odd.
\end{itemize}

3) For $m=1$, we check that $x |\pt_x \QB_{m=1}| = \frac{2x}{1+x^2} \in \mathrm{span} \, \{ f_m, g_m \}$ and hence this function (reflecting the scaling symmetry) for the half-harmonic maps equation belongs to the nullspace of $L_+$. Below, we will see that $x |\pt_x \QBm|$ always belong to the nullspace of $L_-$ for any $m \geq 1$.
}
\end{remarks*}

\begin{proof}
Let $\vphi \in \dot{H}^{\frac 1 2}(\R)$ solve $L_+ \vphi = 0$, i.\,e.,
\be
|\nabla| \vphi - \frac{2m}{1+x^2} \vphi = 0 .
\ee 
Recalling that $1-\sin \theta = \frac{2}{1+x^2}$ when $x = \Pi(e^{\im \theta})$ and by using Lemma \ref{lem:pol}, we deduce the equivalence
\be
L_+ \vphi = 0  \quad \Longleftrightarrow \quad \tL_+ \tvphi = 0 ,
\ee
with $\tvphi = \vphi \circ \Pi \in \dot{H}^{\frac  1 2}(\Ss^1)$ and
\be
\tL_+ = |\nabla|_{\Ss^1} - m \mathds{1}.
\ee
Clearly, the nullspace of $\tL_+$ is given by
\be
\mathcal{N}(\tL_+) = \left \{ \tvphi \in \dot{H}^{\frac 1 2}(\Ss^1) : \tL_+ \tvphi = 0 \right  \} =  \mathrm{span} \, \{ \cos( m \theta), \sin (m \theta) \}
\ee
 By undoing the stereographic projection, we conclude from \eqref{eq:Cheby} and \eqref{eq:TmUm} that the nullspace of $L_+$ is given by
\be
\mathcal{N}(L_+) = \mathrm{span} \left \{ f_m, g_m \right \} ,
\ee
which completes the proof of Proposition \ref{prop:null_Lp}.
\end{proof}

Next, we turn to the nullspace of $L_-$ which requires a more subtle analysis. Still, we have the following explicit and complete result.

\begin{prop}[Nondegeneracy of $L_-$] \label{prop:null_Lm}
The nullspace of $L_-$ is given by
$$
\mathcal{N}(L_-) = \mathrm{span} \left \{ 1, f_1, \ldots, f_m, g_1 , \ldots , g_m \right \} .
$$
In particular, it holds that $\dim \mathcal{N}(L_-) = 2m + 1$.
\end{prop}

\begin{remarks*}
{\em 1) From Theorem \ref{thm:vsmall} we recall that the family of half-harmonic maps of equatorial form $\QBm=(f_m(x), g_m(x), 0)$ has $2m +1$ degrees of freedom expressed by the parameters $\theta \in \R$ (phase), $\lambda_1, \ldots, \lambda_m >0$ (scalings) and $a_1, \ldots, a_m \in \R$ (translations). Hence the fact that $\dim \mathcal{N}(L_-) = 2m+1$ holds can be seen as a {\bf nondegeneracy} result for $L_-$, which rules out extra zero modes that are not due to symmetry.

2) The set of functions $\{1,  f_1, \ldots, f_m , g_1, \ldots, g_m\}$ are obtained via the inverse stereographic projection $\Pi^{-1}$ as real and imaginary parts of the set of the complex functions $\{ 1, z, \ldots, z^m, z^{-1}, \ldots, z^{-m} \}$ restricted to the unit circle $\Ss^1$.

3) Reflecting the translation and scaling symmetries and rotational symmetry around the $z$-axis, it is straightforward to check that we always have the zero modes
$$
\mathrm{span} \left \{ 1, |\pt_x \QBm|, x |\pt_x \QBm | \right \} \subset \mathcal{N}(L_+),
$$
and  this inclusion must be an equality if $m=1$.}
\end{remarks*}

\begin{proof}
 Again, we make essential use of the stereographic projection. From \eqref{eq:QBm_lift} together with the relations \eqref{eq:stereo1} and \eqref{eq:stereo3} we obtain that
\be
(R \vphi)(x) = (1- \sin \theta) (\widetilde{R}_m \tilde{\vphi})(e^{i \theta}) .
\ee 
Here $\tilde{\vphi} = \vphi \circ \Pi$ as usual and the integral operator $\widetilde{R}_m$ is found to be
\be \label{eq:Rm_tilde}
(\widetilde{R}_m \tilde{\vphi})(e^{i \theta}) = \int_{\Ss^1} K_m(\theta, \omega) \tvphi(\omega) \, d \omega ,
\ee
where the kernel $K(\theta, \omega)$ is given by
\be
K_m(\theta, \omega) = \frac{1}{2 \pi} \frac{|e^{im \theta} - e^{im \omega}|^2}{|e^{i \theta} - e^{i \omega}|^2} = \frac{1}{2\pi} \frac{1- \cos (m (\theta-\omega))}{1-\cos(\theta-\omega)}. 
\ee

Now, let $\vphi \in \dot{H}^{\frac 1 2}(\R)$ solve $L_- \vphi = 0$. As in the proof of Proposition \ref{prop:null_Lp} above, we can make use of Lemma \ref{lem:pol} to obtain the following equivalence for $\vphi \in \dot{H}^{\frac 1 2}(\R)$ and $\tvphi = \vphi \circ \Pi \in \dot{H}^{\frac 1 2}(\Ss^1)$ such that
\be
L_- \vphi = 0 \quad \Longleftrightarrow \quad \widetilde{L}_- \tvphi = 0,
\ee
with the operator
\be
 \widetilde{L}_- = |\nabla|_{\Ss^1} - m \mathds{1} + \widetilde{R}_m   .
\ee
To analyze the nullspace of $\widetilde{L}_-$, we need the following auxiliary result.

\begin{lemma} \label{lem:funk-hecke}
For $f_\ell \in \{ \cos (\ell \theta), \sin (\ell \theta) \}$ with $\ell \in \N$, we have
$$
\widetilde{R}_m f_\ell = \lambda_\ell f_\ell \quad \mbox{with} \quad \lambda_\ell = \begin{dcases*}  (m - \ell ) & for $0 \leq \ell < m$, \\ 0 & for $\ell \geq m$. \end{dcases*}  
$$
\end{lemma}

\begin{proof} Since $\cos ( m \vartheta) = T_m(\cos \vartheta)$, where $T_m$ denotes the Chebyshev polynomial of the first kind with index $m \in \N$, we see that the kernel
$$
K_m(\theta, \omega) = \frac{1}{2 \pi} \frac{1- T_m (\cos (\theta-\omega))}{1-T_1(\cos (\theta - \omega))} 
$$  
only depends on $\cos (\theta-\omega)$. With some slight abuse of notation we shall simply write $K_m(\cos(\theta-\omega))=K_m (\theta, \omega)$ in what follows. Since  the operator $\widetilde{R}_m$ commutes with rotations on $\Ss^1$, we can diagonalize it with respect to the orthogonal decomposition
$$
L^2(\Ss^1) = \bigoplus_{\ell = 0}^\infty \mathcal{Y}_\ell(\Ss^1)
$$
into subspaces $\mathcal{Y}_\ell(\Ss^1)$ indexed by the angular momentum $\ell \in \N$, where we recall that $\mathcal{Y}_0(\Ss^1) = \mathrm{span}  \{ 1 \}$ and $\mathcal{Y}_\ell(\Ss^1) = \mathrm{span} \{ \cos (\ell \theta), \sin (\ell \theta) \}$ for $\ell \geq 1$.

By invoking the classical {\bf Funk--Hecke formula} (see, e.\,g., \cite{Mu-98}), we find that the eigenvalue $\lambda_\ell$ of the operator $\widetilde{R}_m$ acting on $f_\ell \in \mathcal{Y}_\ell(\Ss^1)$ is given by
\be
\widetilde{R}_m f_\ell = \lambda_\ell f_\ell \quad \mbox{with} \quad \lambda_\ell =  2 \int_{-1}^{+1} K_m(y) T_\ell (y) \frac{dy}{\sqrt{1-y^2}}   .
\ee
In order to evaluate this integral explicitly, we shall need the following key identity
\be \label{eq:T_darboux}
2 \pi K_m(y) = \frac{1-T_m(y)}{1-T_1(y)} = m T_0(y) + 2 \sum_{k=1}^{m-1} (m-k) T_k(y) 
\ee
valid for all $y \in [-1,1)$. To show \eqref{eq:T_darboux}, we first recall the {\bf Christoffel--Darboux identity} for Chebyshev polynomials which reads 
\be \label{eq:T_darboux2}
T_0(x) T_0(y) + 2 \sum_{j=1}^n T_j(x) T_j(y) =  \frac{T_{n+1}(x) T_n(y) - T_n(x) T_{n+1}(y)}{x-y} 
\ee
for $x \neq y$ and $n \in \N$; see, e.\,g., \cite{Ri-74}. Using that $T_n(1) = 1$ for all $n \in \N$ and $T_1(y)=y$ together with  \eqref{eq:T_darboux2} for $x=1$ and $y \in [-1,1)$, we conclude that
\begin{align*}
\frac{1- T_m(y)}{1-T_1(y)} & =  \sum_{k=0}^{m-1} \frac{T_k(y) - T_{k+1}(y)}{1-y} = \sum_{k=0}^{m-1} \left ( T_0(y) + 2 \sum_{j=1}^{k} T_j(y) \right ) \\
& = m T_0(y) + 2 \sum_{k=1}^{m-1} (m-k) T_k(y) ,
\end{align*}
which proves that \eqref{eq:T_darboux} holds for $y \in [-1,1)$. With \eqref{eq:T_darboux} at hand now, we deduce
\begin{align*}
\lambda_\ell  = 2 \int_{-1}^{+1} K_m(y) T_\ell(y) \frac{dy}{\sqrt{1-y^2}} = \begin{dcases*} (m- \ell) & for $0 \leq \ell < n$, \\
0 & for $\ell \geq m$, \end{dcases*}
\end{align*}
where we used that well-known orthogonality property $\int_{-1}^{+1} T_i(y) T_j(y) (1-y^2)^{-1/2} dy = c_i \delta_{ij}$ with $c_0 = \pi$ and $c_i=\pi/2$ for $i \geq 1$.  This completes the proof of Lemma \ref{lem:funk-hecke}.
\end{proof}

We now return to the proof of Proposition \ref{prop:null_Lm}. By applying Lemma \ref{lem:funk-hecke}, we find that
\begin{align*}
\left \{ \tvphi \in \dot{H}^{\frac 1 2}(\Ss^1) :  \widetilde{L}_- \tvphi = 0 \right \} =
\mathrm{span} \left \{ 1, \cos \theta, \sin \theta, \ldots, \cos( m \theta), \sin (m \theta) \right  \},
\end{align*}
using that $\widetilde{L}_-  e_\ell = 0$ with $e_\ell \in \{ \cos (\ell \theta), \sin (\ell \theta) \}$ if and only if $0 \leq \ell \leq m$. By recalling the definition of the functions $f_k$ and $g_k$ in \eqref{eq:TmUm}, we complete proof of Proposition \ref{prop:null_Lm}.
\end{proof}

\section{Spectral Analysis in $L^2$} 

\label{sec:spectrum}

We now study the spectra of the operators $L_+$ and $L_-$ acting on $L^2(\R)$ obtained from the linearization around a half-harmonic map $\QBm$ given by \eqref{eq:QBm}. The results stated in Theorems \ref{thm:L2}--\ref{thm:JL} and Corollary \ref{cor:coercive} will follow by summarizing our findings in this section.

By standard theory, we note that $L_+ = |\nabla| - \frac{2m}{1+x^2}$ and $L_- = L_+ + R_m$ are self-adjoint operators on $L^2(\R)$ with operator domains $\mathrm{dom}(L_+)= \mathrm{dom}(L_-)=  H^{1}(\R)$. Furthermore, since $\frac{2m}{1+x^2} \to 0$ as $|x| \to \infty$ and $R_m$ is easily seen to be a compact operator on $L^2(\R)$, we conclude that the essential spectra of $L_+$ and $L_-$ are 
\be
\sigma_{\mathrm{ess}}(L_+) = \sigma_{\mathrm{ess}}(L_-) = [0, + \infty ).
\ee
Below, we will actually show that $L_+$ and $L_-$ on the orthogonal complement of their bound states  are both unitarily equivalent to the ``free'' operator $|\nabla|$. As a consequence, both $L_+$ and $L_-$ have purely absolutely continuous spectrum $\sigma_{\mathrm{ac}}(L_+) = \sigma_{\mathrm{ac}}(L_-) = [0, \infty)$ and, in particular, there are {\em no embedded} eigenvalues $E > 0$ of $L_+$ and $L_-$.

\subsection{Point Spectrum of $L_+$}
We first study the point spectrum for the self-adjoint operators $L_+$ acting on $L^2(\R)$, which is the set defined as
\be
\sigma_{\mathrm{p}}(L_+) = \{ E \in \R : \mbox{$L_+ \vphi = E \vphi$ for some $\vphi \in H^1(\R)$ with $\vphi \not \equiv 0$} \} .
\ee
We obtain the following complete description.

\begin{prop} \label{prop:Lp_L2}
Let $m \in \N$ with $m \geq 1$. Then the point spectrum of the operator $L_+= |\nabla| - \frac{2m}{1+x^2}$ acting on $L^2(\R)$ consists of exactly $2m$ eigenvalues, i.\,e., 
$$
\sigma_{\mathrm{p}}(L_+) = \{ E_0, \ldots, E_{2m-1} \} .
$$
Moreover, every eigenvalue $E_j$ is simple and nonnegative and we have that
$$
E_0 < E_1 < E_2 < \ldots < E_{2m-2} < E_{2m-1} = 0 .
$$
In particular, $0$ is always $L^2$-eigenvalue of $L_+$ and there are no positive eigenvalues $E >0$ embedded in $\sigma_{\mathrm{ess}}(L_+)=[0, +\infty)$. 
\end{prop}

\begin{remarks*}
{\em 
1) In fact, all eigenvalues $E_k $ of $L_+$ can (in principle) be calculated together with the corresponding eigenfunctions $\vphi_k \in L^2(\R)$ by means of suitable orthogonal polynomials, which are related to the tridiagonal matrices $M^{(m)} \in \C^{(2m-1) \times (2m-1)}$ introduced below. In particular, the non-zero eigenvalues $E_k < 0$ of $L_+$ are exactly given by the eigenvalues of the matrix $M^{(m)}$.

For example, the cases $m=1$ and $m=2$ yield
$$
M^{(1)} = -1  \quad \mbox{and} \quad  M^{(2)} = \left [ \begin{array}{ccc} -1 & \im & 0 \\ -\frac{\im}{2} & -2 & \frac{\im}{2} \\ 0 & -\im & -1 \end{array} \right ].
$$
Its eigenvalues are easily found to be $\{ -1 \}$ for $M^{(1)}$ and $\{\frac 1 2( -3 - \sqrt{5}), -1, \frac 1 2 (-3 + \sqrt{5}) \}$ for $M^{(2)}$. Thus we get
\begin{itemize}
\item $L_+ = |\nabla| - \frac{2}{1+x^2}$ has two $L^2$-eigenvalues  which are simple and given by
$$
E_0 = -1, \quad E_1 = 0.
$$
\item $L_+ = |\nabla|- \frac{4}{1+x^2}$ has four $L^2$-eigenvalues which are simple and given by
$$
E_0 = \frac{- 3 - \sqrt{5}}{2}, \quad E_1 = -1, \quad E_2 = \frac{-3 + \sqrt{5}}{2}, \quad E_3 = 0.
$$
\end{itemize}

3) From Proposition \ref{prop:null_Lp} we recall that $L_+$ always has a resonance $\vphi \in \dot{H}^{\frac 1 2}(\R)\setminus L^2(\R)$ with $L_+ \vphi = 0$.

4) Note that we rule out that $L_+$ has {\em no positive eigenvalues} $E > 0$ embedded inside the essential spectrum $\sigma_{\mathrm{ess}}(L_+)= [0, \infty)$.
}
\end{remarks*}

\begin{proof} 
Let $\vphi \in \mathrm{dom}(L_+) = H^1(\R)$ with $\vphi \not \equiv 0$ solve $L_+ \vphi = E \vphi$ with some $E \in \R$. If $E=0$, then we conclude that $E=0$ is a simple eigenvalue of $L_+$ in $L^2(\R)$ from Proposition \ref{prop:null_Lp} and the remark following it. 

Hence it suffices to consider the case $E \neq 0$ for the rest of the proof. As usual, we let $\tvphi = \vphi \circ \Pi : \Ss^1 \to \R$ denote the stereographic lifting of $\vphi$ to the unit circle. By applying Lemma \ref{lem:pol}, we deduce the equivalence 
\be \label{eq:goodstuff}
L_+ \vphi = E \vphi \quad \Longleftrightarrow \quad (1-\sin \theta) \widetilde{L}_+ \tvphi = E \tvphi ,
\ee
where we recall that $\widetilde{L}_+ = |\nabla|_{\Ss^1} - m \mathds{1}$. Note that since we deal with $E \neq 0$, we cannot easily get rid of the factor $(1-\sin \theta)$ in contrast to the previous discussion of the nullspace of $L_+$. Moreover, we remark that for $E \neq 0$ any solution $\vphi \in \dot{H}^{\frac 1 2}(\R)$ of $L_+ \vphi = E \vphi$ automatically belongs to $H^1(\R)$. Furthermore, it is easy to see that $\vphi \in L^2(\R)$ if and only if $\int_{\Ss^1} \frac{|\tvphi|^2}{1- \sin \theta} d \theta$ is finite. The latter condition is readily checked to be true for any solution $\tvphi \in \dot{H}^{\frac 1 2}(\Ss^1)$ of $(1-\sin \theta) \tL_+ \tvphi = E \tvphi$ provided that $E \neq 0$.

Therefore, in view of the equivalence \eqref{eq:goodstuff}, we are led to determine the non-zero eigenvalues $E \neq 0$ for the non self-adjoint and unbounded operator
\be
J = (1-\sin \theta) \widetilde{L}_+,
\ee
which is a closed operator with operator domain $H^1(\Ss^1)$. The upshot of this approach turns out to be that $J$ exhibits the structure of a {\bf Jacobi operator} (i.\,e.~an infinite tridiagonal matrix) when expressed as a matrix in the canonical (complex) basis of $L^2(\Ss^1)$ that is given by $e_k(\theta) =  (2 \pi)^{-1/2} e^{\im k \theta}$ with $k \in \Z$. Recall that the complex inner product for complex-valued functions $f,g \in L^2(\Ss^1)$ is taken to be
\be
\langle f | g \rangle =  \int_{\Ss^1} \overline{f(\theta)} g(\theta) \, d \theta.
\ee 
For the matrix elements of $J$, we use the short-hand notation $J_{kl} = \langle e_k | J e_l \rangle$ with $k,l \in \Z$. Note that $J e_l \in L^2(\Ss^1)$ for all $l \in \Z$. An elementary calculation yields
$$
 J_{kl} = \frac{(|l|-m)}{2 \pi} \int_{\Ss^1} (1- \sin \theta)  e^{\im (l-k) \theta} \,d \theta = (|l|-m) \cdot \begin{dcases*} 1 & for $k = l$, \\
 \pm \frac{\im}{2} & for $k = l \pm 1$, \\ 0 & else. \end{dcases*}
$$
Thus we see that $J$ has a tridiagonal matrix of the form
\begin{equation} \label{eq:Jkl}
[J_{kl}] = 
\left [ \begin{array}{ccc|c|ccc|c|ccc} & A^{(-)} & & \vdots & & 0 & & \vdots & & & 0 \\ &  & & 0 &   &  & & 0 &   &   \\[0.5ex] \hline   \ldots & 0  & \frac{\im}{2} & 0 &  \frac{\im}{2}& 0 & \ldots & 0 & 0 & \ldots \\[0.5ex]  \hline & \ldots & 0 & \vdots & & M^{(m)} & & \vdots & 0 & \ldots \\  \hline & \ldots & 0 & 0 & \ldots & 0 & -\frac{\im}{2} & 0 & -\frac{\im}{2} & 0 & \ldots \\ \hline & & & 0  &  &  &  & 0  &  \\ 0 & & & \vdots & & &  & \vdots &  & A^{(+)} & \end{array}   \right ]
\end{equation}
where the two columns that contain only zeros correspond to $l=-m$ and $l=+m$. In \eqref{eq:Jkl}, the matrices  
\be \label{eq:Apm}
A^{(-)} = \left [ \begin{array}{cccc}
 & & &  0 \\
\ddots &  \ddots & \ddots & \\
& \frac{3 \im}{2} & 2 & -\frac{\im}{2} \\
0 &  & \im & 1 \end{array} \right ] , \quad A^{(+)} = \left [ \begin{array}{cccc}
1 & -\im & & 0 \\
\frac{\im}{2} & 2 & -\frac{3 \im}{2} \\ & \ddots & \ddots & \ddots\\ 0  \end{array} \right ]
\ee
are semi-infinite and tridiagonal.  The finite matrix $M^{(m)}$ is the $(2m-1)\times(2m-1)$-tridiagonal matrix given by the submatrix of $[J_{kl}]$ restricted to $k,l \in \{ -m+1, \ldots, m-1\}$, i.\,e., we have
\begin{equation} \label{eq:matrixMn}
M^{(m)} = \left [ \begin{array}{cccc}
\alpha_{1} & \beta_{2m-2} & & 0\\
- \beta_1 & \ddots & \ddots & \\
& \ddots & \ddots &  \beta_1 \\
0 & & - \beta_{2m-2} & \alpha_{2m-1} \end{array} \right ]  .
\end{equation}
Here the non-zero entries  $(\alpha_n)_{n=1}^{2m-1}$ and $(\beta_n)_{n=1}^{2m-2}$ are given by
\be \label{eq:an}
\alpha_n = \begin{dcases*} -n & for $1 \leq n \leq m-1$,  \\ 
n - 2m & for $m \leq n \leq 2m-1$,  \end{dcases*} 
\ee
\be \label{eq:bn}
 \beta_n = \begin{dcases*} \frac{\im}{2} n & for $1 \leq n \leq m-1$, \\ \frac{\im}{2} (2m-n) & for $m \leq n \leq 2m-2$.  \end{dcases*}  
\ee

We now determine the non-zero eigenvalues of the operator $J : H^1(\Ss^1) \subset L^2(\Ss^1) \to L^2(\Ss^1)$ as follows. In view of \eqref{eq:Jkl}, we decompose $L^2(\Ss^1)$ into the three closed subspaces as
$$
L^2(\Ss^1) = \Lambda^{(-)}_m + \Lambda^{(0)}_m + \Lambda^{(+)}_m,
$$
where
$$
\Lambda^{(-)}_m := \overline{\mathrm{span} \, \{ e_k : k \leq - m \}}, \quad \Lambda^{(0)}_m := \mathrm{span} \, \{ e_k : |k| \leq m \}, \quad \Lambda^{(+)}_m = \overline{\mathrm{span} \, \{ e_k : k \geq m \}}.
$$
Note that  $\Lambda^{(\pm)} \cap \Lambda^{(0)}_m = \mathrm{span} \, \{e_{-m}, e_{+m} \} = \mathrm{ker} \, J$ is the kernel of $J$ (which corresponds to the nullspace $\mathcal{N}(L_+)$ after undoing the stereographic projection). Furthermore, we readily see from \eqref{eq:Jkl} that   $J$ acts invariantly on these three subspaces, i.\,e., $J g \in \Lambda^{(\pm)}_m$ for any $g \in \Lambda^{(\pm)}_m \cap H^1(\Ss^1)$ and $J g \in \Lambda^{(0)}_m$ for any $g \in \Lambda^{(0)}_m$. Thus when determining the eigenvalues of $J$, we can analyze the action of $J$ on the spaces $\Lambda^{(\pm)}_m$ and $\Lambda^{(0)}_m$ separately.

We begin with the finite-dimensional subspace $\Lambda^{(0)}_m$.

\begin{lemma} \label{lem:J_eigen}
The operator $J$ restricted to $\Lambda^{(0)}_m$ has exactly $2m-1$ non-zero eigenvalues $E \neq 0$ and they satisfy (counting multiplicity) 
$$
E_0 < E_1 < \cdots < E_{2m-2} <0 .
$$
In particular, every eigenvalue $E$ is simple and strictly negative.
\end{lemma}

\begin{proof}
The operator $J$ restricted to the $2m+1$-dimensional subspace $\Lambda^{(0)}_m$ has the corresponding matrix given by
\be
M = \left [ \begin{array}{ccccc} 0 & \frac{\im}{2} & 0 & \ldots & 0 \\ \vdots & & M^{(m)} & & \vdots \\ 0 & \ldots &0 & -\frac{\im}{2} & 0 \end{array}  \right ] .
\ee
Here $M^{(m)}$ denotes the $(2m - 1)\times (2m-1)$-matrix taken from \eqref{eq:matrixMn}. Note that the first and last column of $M$ are both zero. Thus, by Cramer's rule, the characteristic polynomial $p_M(\lambda)$ of the matrix $M$ factorizes as
\be
p_M(\lambda)  = \det ( M - \lambda \mathds{1}_{2m+1} ) = \lambda^2 \cdot \det ( M^{(m)} - \lambda \mathds{1}_{2m-1} )  .
\ee
Clearly, the vectors $[1,0, \ldots, 0 ]^T$ and  $[0, \ldots, 0, 1]^T$ are eigenvectors for $M$ with eigenvalue $\lambda = 0$. Hence the eigenvalue $\lambda = 0$ is (at least) double degenerate. 

Next, by exploiting the tridiagonal structure of $M^{(m)}$, we deduce that its characteristic polynomial factorizes as 
\be
p_{M^{(m)}}(\lambda)=\det ( M^{(m)} - \lambda \mathds{1}_{2m-1}) = (-1) \prod_{k=1}^{2m-1} (\lambda - \lambda_k)  
\ee
with distinct real roots $\lambda_1 < \lambda_2 < \ldots < \lambda_{2m-1}$; see Lemma \ref{lem:jacobi}. By setting $E_k = \lambda_{k+1}$ for $0 \leq k \leq 2m-2$, we conclude that $J$ restricted to $\Lambda^{(0)}_m$ has exactly $2m-1$ non-zero eigenvalues $E_0 < E_1 < \ldots < E_{2m-2}$. 

It remains to show that the largest eigenvalue $E_{2m-2} < 0$ is strictly negative. Indeed, assume that $J g = E g$ with $E = E_{2m-2} > 0$ and $g \in \Lambda^{(0)}_m$.  Recalling that $J = (1-\sin \theta) \tL_+$ with $\tL_+ = |\nabla|_{\Ss^1} - m \mathds{1}$, we note that
\be
\int_{\Ss^1} \frac{|g(\theta)|^2}{1 - \sin \theta} \, d \theta = \frac{1}{E} \langle g, \tL_+ g \rangle = \frac{1}{E} \sum_{|k| \leq m} (k-m) |\widehat{g}(k)|^2 \leq 0. 
\ee 
But this shows that $g(\theta) = 0$ for a.\,e.~$\theta \in [0,2\pi]$. Hence we must have that $E_{2m-2} \leq 0$. Finally, assume that $E_{2m-2}=0$. Then the matrix $M$ has a three-dimensional kernel. But this implies that $\dim \mathcal{N}(L_+) = \dim \mathrm{ker}(J) = \dim \mathrm{ker}(\tL_+) \geq 3$, which contradicts Proposition \ref{prop:null_Lp}. \end{proof}

\begin{lemma} \label{lem:J_eigen2}
The operator $J$ acting on $\Lambda^{(\pm)}_m$ has no non-zero eigenvalues $E \neq 0$.
\end{lemma}

\begin{proof}
It suffices to consider the case $\Lambda^{(+)}_m$, since the arguments for $\Lambda^{(-)}_m$ will be analogous. Thus we suppose that $g \in \Lambda^{(+)}_m \cap H^1(\Ss^1)$ solves $J g = E g$ for some $E \neq 0$. 

First, for any $E \neq 0$, we remark that 
\be \label{eq:rule_out}
\int_{\Ss^1} \frac{|g(\theta)|^2}{1-\sin \theta} \, d \theta = \frac{1}{E} \langle g, \widetilde{L}_+ g \rangle = \frac{1}{E} \sum_{k \geq m}  (k-m) |\widehat{g}(k)|^2 < +\infty
\ee
 and that $E < 0$ implies that $g(\theta) = 0$ a.\,e.~on $\Ss^1$. In particular, the operator $J$ acting on $\Lambda^{(+)}_m$ cannot have $E < 0$ as an eigenvalue.   

It remains to rule out that $E>0$ can occur. Since $g \in \Lambda^{(+)}_m \cap H^1(\Ss^1)$ only has positive frequencies (i.\,e.~we have $\widehat{g}(k) = 0$ for $k < 0$), we have $|\nabla|_{\Ss^1} g = -\im g'$. Thus we find that $g : \Ss^1 \to \C$ satisfies the ordinary differential equation
\be
(1- \sin \theta) \left ( -\im \frac{d}{d \theta} - m \right ) g = E g .
\ee
If  we now set $\widetilde{g}(\theta) = e^{-\im m \theta} g(\theta)$, we deduce that the equation
\be
 \frac{d \widetilde{g}}{d \theta} = \frac{\im E \widetilde{g}}{1-\sin \theta} .
\ee
holds for all $\theta \in [0, 2 \pi]$ with $\theta \neq \pi/2$. By integration, we obtain its general non-trivial solution
\be
\widetilde{g}(\theta) = g_0 e^{\im E ( \tan \theta + \sec \theta )} .
\ee
defined on $\Ss^1 \setminus \{ \im \}$, where $g_0 \in \C$ is some constant $g_0 \neq 0$. But since $|g(\theta)| = |\widetilde{g}(\theta)| = |g_0| > 0$ on $\Ss^1 \setminus \{ \im \}$, this shows that
\be
\int_{\Ss^1} \frac{|g(\theta)|^2}{1-\sin \theta} \, d \theta = \int_{\Ss^1} \frac{|g_0|^2}{1-\sin \theta} \, d\theta = +\infty .
\ee
But this contradicts  \eqref{eq:rule_out} and completes the proof of Lemma \ref{lem:J_eigen2}.
\end{proof}

\begin{remark*}
{\em Below, we will further expound the idea of using the {\bf ``gauge transformation''}  $g(\theta) \mapsto \widetilde{g}(\theta) = e^{-\im m \theta} g(\theta)$ when we study the continuous spectrum of $L_+$ and $L_-$ and show a unitary equivalence  to the free operator $|\nabla|$ on $\R$. }
\end{remark*}

{\bf Completing the Proof of Proposition \ref{prop:Lp_L2}.} Let $E \neq 0$ and $\tvphi = \sum_k c_k e_k \in H^1(\Ss^1)$ solve $J \tvphi = E \tvphi$. Suppose that $c_k \neq 0$ for some $k > m$. In view of the matrix for $J$ displayed in \eqref{eq:Jkl}, we find that $g = \sum_{k \geq m} c_k e_k \in \Lambda^{(+)}_m$ is a nontrivial solution of $J g = Eg$ with $E \neq 0$, which contradicts Lemma \ref{lem:J_eigen2}. Thus $c_k = 0$ for all $k > m$. Likewise, we prove that $c_k  = 0$ for all $k < m$. In summary, we obtain $\tvphi = \sum_{|k| \leq m} c_k e_k \in \Lambda^{(0)}_m$ and we conclude that $E \neq 0$ is one of the eigenvalues given by Lemma \ref{lem:J_eigen}. Thus the operator $J$ has exactly $2m-1$ non-zero and simple eigenvalues with
$$E_0 < E_1 < \cdots < E_{2m-2} < 0.$$
By the equivalence  \eqref{eq:goodstuff}, the operator $L_+$ has exactly $\{E_k \}_{k=0}^{2m-2}$ as non-zero eigenvalues and these are simple. Finally, from Proposition \ref{prop:null_Lp} and the remark following it, we deduce that $L_+$ also has $E_{2m-1}=0$ as a simple eigenvalue in $L^2$. 

The proof of Proposition \ref{prop:Lp_L2} is now complete.
\end{proof}

\subsection{Point Spectrum of $L_-$} We now turn to the complete description of the point spectrum of the self-adjoint operator $L_-$ acting on $L^2(\R)$, i.\,e., we consider the set
\be
\sigma_{\mathrm{p}}(L_-) = \{ E \in \R : \mbox{$L_- \vphi = E \vphi$ for some $\vphi \in H^1(\R)$ with $\vphi \not \equiv 0$} \} .
\ee
We derive the following result.

\begin{prop} \label{prop:Lm_L2}
Let $m \in \N$ with $m \geq 1$. Then the point spectrum of $L_- = |\nabla| - \frac{2m}{1+x^2}+ R_m$ acting on $L^2(\R)$ consists only of zero, i.\,e.,
$$
\sigma_{\mathrm{p}}(L_-) = \{ 0 \} .
$$ 
The corresponding eigenspace of $L_-$ in $L^2(\R)$ has dimension $2m$ and is given by
$$
\mathrm{ker} (L_-) = \mathrm{span} \{ \vphi_0, \ldots, \vphi_{2m-1} \},
$$
where $\{ \vphi_k \}_{k=0}^{2m-1} \subset H^1(\R)$ denote the  set of (normalized) $L^2$-eigenfunctions of the operator $L_+$.
\end{prop}

\begin{proof}
The proof parallels the analysis for the operator $L_+$ in the proof of Proposition \ref{prop:Lp_L2} above. Therefore, we only sketch the main steps and focus on the necessary modifications.

We consider the non self-adjoint closed operator
\be
H = (1- \sin \theta) \widetilde{L}_- \quad 
\ee 
with domain $H^1(\Ss^1)$ and with the self-adjoint operator
\be
\widetilde{L}_- = |\nabla|_{\Ss^1} - m \mathds{1} + \widetilde{R}_m 
\ee
where $\widetilde{R}_m$ is the self-adjoint and bounded integral operator on $L^2(\Ss^1)$ introduced in \eqref{eq:Rm_tilde} above. As in the proof of Proposition \ref{prop:Lp_L2}, we deduce the following equivalence for eigenfunctions $\vphi \in H^1(\R)$ of $L_- \vphi = E \vphi$ with $E \neq 0$:
\be \label{eq:goodstuff2}
L_+ \vphi = E \vphi \quad \Longleftrightarrow \quad H \tvphi = E \tvphi,
\ee
where $\tvphi = \vphi \circ \Pi$. Thus in order to determine all non-zero eigenvalues $E \neq 0$ of $L_-$, we have to find all non-zero eigenvalues of $H$.

Recall that $\widetilde{R}_m$ commutes with rotations on $\Ss^1$ and hence it becomes diagonal in the basis $\{ e_k : k \in \Z \}$. The corresponding eigenvalues of $\widetilde{R}_m$ in the invariant subspaces $\mathcal{Y}_\ell (\Ss^1) = \mathrm{span} \{ e_k  : |k| = \ell \}$ are calculated by Lemma \ref{lem:funk-hecke} through the Funk--Hecke formula. Thus, in analogy to \eqref{eq:Jkl}, the matrix elements $H_{kl} = \langle e_k, H e_l \rangle$ with $k,l \in \Z$ are given by the tridiagonal infinite matrix
\be
[H_{kl}] = \left [ \begin{array}{ccc|c|ccc|c|ccc} & A^{(-)} & & \vdots & & 0 & & \vdots & & & 0 \\ &  & & 0 &   &  & & 0 &   &   \\[0.5ex] \hline   \ldots & 0  & \frac{\im}{2} & 0 & 0 & 0 & \ldots & 0 & 0 & \ldots \\[0.5ex]  \hline & \ldots & 0 & \vdots & & N^{(m)} & & \vdots & 0 & \ldots \\  \hline & \ldots & 0 & 0 & \ldots & 0 & 0 & 0 & -\frac{\im}{2}  & 0 & \ldots \\ \hline & & & 0  &  &  &  & 0  &  \\ 0 & & & \vdots & & &  & \vdots &  & A^{(+)} & \end{array}   \right ] 
\ee      
with the same semi-infinite matrices $A^{(-)}$ and $A^{(+)}$ as in \eqref{eq:Apm} above. Furthermore, the central $(2m-1) \times (2m-1)$-matrix denoted by $N^{(m)}$ is now the {\em zero matrix}, i.\,e.,
\be
N^{(m)} = \left [ \begin{array}{ccc} 0 & \cdots & 0 \\ \vdots & \ddots & \vdots \\ 0 & \cdots & 0 \end{array} \right ].
\ee
We claim that the following holds.

\begin{lemma}
The operator $H$ has no non-zero eigenvalue $E \neq 0$.
\end{lemma}

\begin{proof}
As in the analysis for $J$ above, we find that $H$ acts invariantly on the subspaces $\Lambda^{(\pm)}_m$ and $\Lambda^{(0)}_m$. Thus we can split our eigenvalue analysis accordingly.

First, we note that $H$ restricted to the $2m+1$-dimensional subspace $\Lambda^{(0)}_m$ is given by the $(2m+1) \times (2m+1)$-matrix
\be
N =  \left [ \begin{array}{ccccc} 0 & 0 & 0 & \ldots & 0 \\ \vdots & & N^{(m)} & & \vdots \\ 0 & \ldots &0 & 0 & 0 \end{array}  \right ] ,
\ee
where all entries are zero. Thus $H$  restricted to $\Lambda^{(0)}_m$ is identically zero and hence $E \neq 0$ cannot be an eigenvalue of $H$ acting on $\Lambda^{(0)}_m$.

Finally, by adapting the arguments detailed in the proof of Lemma \ref{lem:J_eigen2}, we conclude that  $H$ acting $\Lambda^{(\pm)}_m$ has non-zero eigenvalues $E \neq 0$ either. 
\end{proof}

In view of \eqref{eq:goodstuff2}, we conclude that $L_+$ has no non-zero eigenvalue $E \neq 0$. We claim that $E=0$ is an $L^2$-eigenvalue for $L_-$ with and that its eigenspace in $L^2(\R)$ has dimension $2m$. From Proposition \ref{prop:null_Lm} we recall that nullspace of $L_-$ in $\dot{H}^{\frac 1 2}(\R)$ is 
\be \label{eq:null_Lm_again}
\mathcal{N}(L_-) = \mathrm{span} \, \{ 1, f_1, \ldots, f_m, g_1, \ldots, g_m \} ,
\ee
where the $\dot{H}^{\frac 1 2}$-functions $f_k$ and $g_k$ are given through \eqref{eq:TmUm}. Now consider the subspace
\be
N = \mathrm{span} \, \{ f_1 - A_1, \ldots, f_m - A_m, g_1 - B_1, \ldots, g_m -B_m \} \subset \mathcal{N}(L_-),
\ee
where the constants $A_k$ and $B_k$ are chosen such that
\be
A_k = \begin{dcases*} 0 & if $k$ is odd \\ -1 & if $k=2,6,10, \ldots$ \\ +1 & if $k=4,8,12, \ldots$ \end{dcases*}, \quad B_k = \begin{dcases*} 1 & if $k=1,5,9, \ldots$ \\ -1 & if $k=3,7,11,\ldots$ \\ 0 & if $k$ is even \end{dcases*} .
\ee
It is clear that $\dim N = 2m$ and it follows that $N \subset H^1(\R)$ by elementary properties of Chebyshev polynomials. Therefore the eigenvalue $E=0$ for $L_-$ in $L^2(\R)$ has an eigenspace of at least $2m$ dimensions. Finally, we note that the kernel of $L_- : H^1(\R) \to L^2(\R)$ cannot have more than $2m$ dimensions in $L^2(\R)$, as this would contradict \eqref{eq:null_Lm_again}.

The proof of Proposition \ref{prop:Lm_L2} is now complete.
\end{proof}

\subsection{Continuous Spectrum and Unitary Gauge Transform}

We start with a remarkable factorization formula, which is somewhat reminiscent of the first-order factorzaition of classical one-dimensional Schr\"odinger operators $H= -\pt_{xx} + V$ using the ground state eigenfunction, which is typically referred to as Darboux transformation.

\begin{lemma}[Darboux-type formula] \label{lem:darboux} For any $f \in \dot{H}^{\frac 1 2}(\R)$, it holds
$$
L_+ f = Q_m \left ( - \im \frac{d}{dx} \right ) \overline{Q}_m f_+ + \overline{Q}_m \left (+\im \frac{d}{dx} \right )  Q_m f_- ,
$$
where 
$$
Q_m(x) = f_m(x) + \im g_m(x) = \left ( \im \cdot \frac{x- \im}{x+ \im} \right )^m.
$$  
\end{lemma}

\begin{proof}
By a standard density argument, it suffices to consider the case $f \in H^1(\R)$. Let $f_+ = \Pi_+ f$ and $f_- = \Pi_- f$ denote the projections of $f$ onto positive and negative frequencies, respectively. Next, let $g \in C^\infty_c(\R)$ be an arbitrary test function. From the behavior under the stereographic projection we deduce
\be
(g, L_+ f) = (g, L_+ f_+) + (g, L_+ f_-) = \langle \widetilde{g}, \tL_+ \widetilde{f}_+ \rangle + \langle \widetilde{g}, \tL_+ \widetilde{f}_- \rangle,
\ee  
 with $\tL_+ = |\nabla|_{\Ss^1} - m \mathds{1}$ and where $\widetilde{f}_+ = f_+ \circ \Pi \in \dot{H}^{\frac 1 2}(\Ss^1)$ an $\widetilde{g}= g \circ \Pi \in \dot{H}^{\frac 1 2}(\Ss^1)$. Since also $f_+ \in L_+^2(\R)$, it is a classical fact (see, e.\,g.,~\cite[Chapter 13]{Ma-09}) from the theory of Hardy spaces that its stereographic lifting $\widetilde{f}_+ \in L_+^2(\Ss^1) = \mathcal{H}^2(\Ss)$ is supported on the positive frequencies on the circle. Hence $\widetilde{f}_+ = \sum_{k \geq 0} f_k e^{\im k \theta}$ satisfies 
 \be
 \tL_+ \widetilde{f}_+ = \left ( |\nabla|_{\Ss^1} - m  \right ) \widetilde{f}_+ = \left ( - \im \frac{d}{d \theta} - m \right ) \widetilde{f}_+ = e^{\im m \theta} \left ( - \im \frac{d}{d \theta} \right ) e^{-\im m \theta}  \widetilde{f}_+ .
 \ee
Recalling that $\widetilde{Q}_m = Q_m \circ \Pi = e^{\im m \theta}$, $\frac{d}{d x} = (1- \sin \theta) \frac{d}{d \theta}$ for $e^{\im \theta} = \Pi(x)$, we conclude that
\be
\langle \widetilde{g}, \tL_+ \widetilde{f}_+ \rangle = \left \langle \widetilde{g}, Q_m \left ( -\im \frac{d}{d \theta} \right ) \overline{Q}_m \widetilde{f}_+ \right \rangle = \left ( g, Q_m \left (- \im \frac{d}{d x} \right ) \overline{Q}_m f_+ \right ) .
\ee
In the same fashion as above, we derive the identity 
\be
\langle \widetilde{g}, \tL_+ \widetilde{f}_- \rangle =  \left (g, \overline{Q}_m \left ( + \im \frac{d}{dx} \right ) Q_m f_- \right ),
\ee 
which completes the proof of Lemma \ref{lem:darboux}.
\end{proof}

Next, we define the linear bounded map $U$ on $L^2(\R)$ by setting
\be
U f = Q_m f_+ + \overline{Q}_m f_-  \quad \mbox{for $f \in L^2(\R)$},
\ee
where $f_+= \Pi_+ f$ and $f_- = \Pi_- f$. In fact, we can show the following result.

\begin{lemma} \label{lem:unitary}
We have $U f \in P^\perp L^2(\R)$ for any $f \in L^2(\R)$. Moreover, the map $U : L^2(\R) \to P^\perp L^2(\R)$ is {\bf unitary}, where its inverse is given by its adjoint $U^* : P^\perp L^2(\R) \to L^2(\R)$ with
$$
U^* g = \overline{Q}_m g_+ + Q_m g_-,
$$
where $g_+ = \Pi_+ g$ and $g_- = \Pi_- g$.
\end{lemma}

\begin{proof}
We divide the proof into the following steps.

\subsubsection*{Step 1} Let $f \in L^2(\R)$ be given and we write $f_+ = \Pi_+ f$ and $f_- = \Pi_- f$. We claim that 
\be \label{eq:U1}
U f = Q_m f_+ + \overline{Q}_m f_- \in P^\perp L^2(\R).
\ee 
To see this, we first note that, since $f_+ \in L^2_+(\R)$, its stereographic lift satisfies $\widetilde{f}_+ = f_+ \circ \Pi \in L_+^2(\Ss^1)$ by a classical result (see, e.\,g.,~\cite[Chapter 13]{Ma-09}). Likewise, we find that $\widetilde{f}_- = f_- \circ \Pi \in L_-^2(\Ss^1)$. Recalling that $\widetilde{Q}_m = Q_m \circ \Pi = e^{\im m \theta}$,  we deduce
\be
\widetilde{Q}_m \widetilde{f}_+ = \sum_{k \geq m} a_k e^{\im k \theta} \in L^2_+(\Ss^1), \quad \widetilde{\overline{Q}}_m \widetilde{f}_- = \sum_{k \leq -m} b_k e^{\im k \theta} \in L^2_-(\Ss^1).
\ee
Now let $\vphi_k \in H^1(\R)$ be an $L^2$-eigenfunction of $L_+$ with eigenvalue $E_k < 0$. Then
\begin{align*}
(\vphi_k, Q_m f_+ ) & = E_k^{-1} ( \vphi_k, L_+( Q_m f_+) )  = E_k^{-1} \langle \tvphi_k, \tL_+ (\widetilde{Q}_m \widetilde{f}_+) \rangle \\
& = E_k^{-1} \left \langle \tvphi_k, \tL_+ \left ( \sum_{k \geq m} a_k e^{\im k \theta} \right ) \right \rangle =0 ,
\end{align*}
which follows from $\tL_+ ( \sum_{k \geq m} a_k e^{\im k \theta} ) = \sum_{k \geq m+1}  (k-m) a_k e^{\im k \theta} \perp \tvphi_k \in \Lambda^{(0)}_m$. Thus we have found $ Q_m f_+ \perp \vphi_k$ and, similarly, we show that $\overline{Q}_m f_- \perp \vphi_k$. In summary, we deduce the $L^2$-orthogonality so that
$$
U f \perp \mathrm{span} \left \{ \vphi_0, \ldots, \vphi_{2m-2} \right \}.
$$
To finally conclude that $P U f = 0$, i.\,e., $Uf \in P^\perp L^2(\R)$, it remains to show that
$$
Uf \perp \vphi_{2m-1},
$$ 
where $\vphi_{2m-1} \in H^1(\R)$ is the unique $L^2$-zero eigenfunction of $L_+$. From the discussion in Section \ref{sec:nullspace} we see that $\vphi_{2m-1} = \mathrm{Re} \, Q_m$ if $m \in \N$ is odd and $\vphi_{2m-1} = \mathrm{Im} \, Q_m$ if $m \in \N$ is even. Thus,  up to an inessential normalization constant, we have $\vphi_{2m-1} = Q_m + \alpha_m \overline{Q}_m$ with $\alpha_m = (-1)^{m+1}$. Hence we have to show that
\be \label{eq:final_touch}
( Q_m + \alpha_m \overline{Q}_m, Uf ) = (Q_m + \alpha_m \overline{Q}_m, Q_m f_+) + (Q_m + \alpha_m \overline{Q}_m, \overline{Q}_m f_-) = 0  .
\ee 
To prove this, we first claim that
\be \label{eq:final_touch2}
(Q_m + \alpha_m \overline{Q}_m , \overline{Q}_m f_-)  = 0.
\ee 
Since $f_- \in L^2_-(\R)$ and $(Q_m + \alpha_m \overline{Q}_m, \overline{Q}_m f_-) = (Q_m^2 + \alpha_m, f_-)$, it suffices to show that $Q_m^2 + \alpha_m$ belongs to $L^2_+(\R)$. Indeed, we notice that $(Q_m^2 + \alpha_m) \circ \Pi = e^{2 \im m \theta} + \alpha_m \in L^2_+(\Ss^1)$. Moreover, we recall that the map
$$
T : L^2_+(\Ss) \to L^2_+(\R), \quad g \mapsto (Tg)(x) = \frac{\sqrt{2}}{x - \im} (g \circ \Pi^{-1})(x)
$$ 
is unitary; see, e.\,g., \cite[Chapter 13]{Ma-09} adapted to our choice of the stereographic projection. In particular, we find that $w := T( \widetilde{h} ) = \frac{1}{x- \im} (Q_m^2 + \alpha_m) \in L^2_+(\R)$ and hence the (distributional) Fourier transform $\mathcal{F}(Q_m^2 + \alpha_m ) = ( -\im \frac{d}{d \xi} - \im) \widehat{w}$ is supported in $[0, \infty)$. Since $Q_m^2 + \alpha_m \in L^2(\R)$, we must have that $Q_m^2 + \alpha_m \in L^2_+(\R)$, whence \eqref{eq:final_touch2} follows. In the same fashion, we conclude that $(Q_m + \alpha_m \overline{Q}_m, f_+) = ( 1 + \alpha_m \overline{Q}^2_m, f_+)= 0$ thanks to the fact that $1 + \alpha_m \overline{Q}^2_m \in L^2_-(\R)$, which follows by taking the complex conjugate of $\alpha_m^{-1} ( \alpha_m + Q_m^2) \in L^2_+(\R)$. 

Thus we have shown that \eqref{eq:final_touch} holds and thereby completing the proof that $Uf \in P^\perp L^2(\R)$ holds.

\subsubsection*{Step 2} We show that $U : L^2(\R) \to P^\perp L^2(\R)$ is a unitary map. First, we demonstrate that $U$ is an isometry. Let $f \in L^2(\R)$ be given and write $f_+ = \Pi_+ f$ and $f_- = \Pi_- f$ as before. By adapting the Hardy space arguments used in Step 1 above, we deduce that $Q_m f_+ \in L^2_+(\R)$ and $\overline{Q}_m f_- \in L^2_-(\R)$. Thus, using also that $Q_m \overline{Q}_m=1$, we see
\begin{align*}
(U f, U f) & = (Q_m f_+ + \overline{Q}_m f_-, Q_m f_+ + \overline{Q}_m f_-) = (Q_m f_+, Q_m f_+) + (\overline{Q}_m f_-, \overline{Q}_m f_-) \\
& = (f_+, f_+) + (f_-, f_-) = (f,f),
\end{align*}
whence it follows that $U$ is an isometry.

It remains to prove that $U : L^2(\R) \to P^\perp L^2(\R)$ is surjective. Let $g \in P^\perp L^2(\R)$ be given. We claim that
\be \label{eq:unitary_good}
\overline{Q}_m g_+ \in L^2_+(\R) \quad \mbox{and} \quad Q_m g_- \in L^2_-(\R).
\ee
Since $L^2_\pm(\R)$ are closed subspaces in $L^2(\R)$ and by the density result provided in Lemma \ref{lem:dense_ran}, it suffices to show the above claim under the additional assumptions that $g \in \mathrm{ran} (L_+)$ and $|g(x)| \leq C/(1+x^2)$. 

To prove the claim \eqref{eq:unitary_good}, we argue as follows. Let $\vphi = \sum_{k=0}^{2m} c_k \vphi_k \in \mathrm{ran}(P) + \mathcal{N}(L_+)$ be a linear combination of the $L^2$-eigenfunctions $\{ \vphi_k  \}_{k=0}^{2m-1}$ of $L_+$ and the resonant zero mode $\vphi_{2m} \in \dot{H}^{\frac 1 2}(\R) \setminus L^2(\R)$. Since $P^\perp g = g$, we have $g \perp \vphi_k$ for all $0 \leq k \leq 2m-1$. Furthermore, since $g = L_+ h$ for some $h \in H^1(\R)$ and by the decay properties $|g(x)| \leq C /(1+x^2)$ and $\vphi_{2m} \in \dot{H}^{\frac 1 2}(\R) \subset L^1(\R, (1+x^2)^{-1} dx)$, we conclude that $(f,\vphi_{2m} ) = (L_+ g, \vphi_{2m} ) = (g, L_+ \vphi_{2m}) = 0$. In summary, by using the stereographic projection, we obtain 
\be \label{eq:foo_doo}
0 = (g, \vphi) = (L_+ h, \vphi ) = \langle \tL_+ \widetilde{h}, \tvphi \rangle = \langle \widetilde{h}, \tL_+ \tvphi \rangle
\ee
with $\widetilde{h} = h \circ \Pi$ and $\tvphi = \vphi \circ \Pi$ as usual. From the proof of Proposition \ref{prop:Lm_L2} we recall that $\mathrm{span} \, \{ \tvphi_0, \ldots, \tvphi_{2m} \} = \Lambda^{(0)}_m = \mathrm{span} \, \{ e_k : |k| \leq m \}$. Thus, for every $k \in \Z$ with $|k| \leq m$, we can find coefficients $c_k \in \C$ such that $\tvphi = e_k$. Since $\tL_+ e_k = (|k|-m) e_k$ and by the self-adjointness of $\tL_+$, we deduce  from \eqref{eq:foo_doo} that $\langle \widetilde{h}, e_k \rangle = 0$ for $|h| \leq m-1$. Therefore we have $\widetilde{h}= \sum_{|k| \geq m} h_k e^{\im k \theta} \in \Lambda^{(-)}_m \oplus \Lambda^{(+)}_m$. Since $L_+ h = g$ is equivalent to $J \widetilde{h} = \widetilde{g}$ and the operator $J$ acts on $\Lambda^{(\pm)}_m$ invariantly, we deduce that 
\be \label{eq:f_lift}
\widetilde{g} = \sum_{|k| \geq m} g_k e^{\im k \theta} \in \Lambda^{(-)}_m \oplus \Lambda^{(+)}_m .
\ee
Recalling that $\widetilde{Q}_m = Q_m \circ \Pi = e^{\im m \theta}$, we thus see
\be
\widetilde{\overline{Q}}_m \widetilde{g}_+ = e^{-\im m \vartheta} \sum_{k \geq m} g_k e^{\im k \theta} \in L^2_+(\Ss^1), \quad \widetilde{Q}_m \widetilde{g}_- = e^{\im m \theta} \sum_{k \leq m} g_k e^{\im k \theta} \in L^2_-(\Ss^1).
\ee
From this we conclude now (similar to the arguments in Step 1 above) that \eqref{eq:unitary_good} holds true.

Thus if we define the map $U^* : P^\perp L^2(\R) \to L^2(\R)$ by setting
\be
U^* g = \overline{Q}_m g_+ + Q_m g_- \in L^2_+(\R) \oplus L^2_-(\R),
\ee
we see that $U(U^* g)= g$. Hence $U$ is surjective and thus it is a unitary map with its inverse given by the adjoint $U^*$.

The proof of Lemma \ref{lem:unitary} is now complete.
\end{proof}

Next, we show that the unitary map $U : L^2(\R) \to P^\perp L^2(\R)$ provides a common unitary equivalence of both the operators $L_+$ and $L_-$ acting on $P^\perp L^2(\R)$ to the free half-Laplacian $|\nabla|$. The precise result is as follows, which is basically a corollary to Lemmas \ref{lem:darboux} and \ref{lem:unitary}.

\begin{cor}[Unitary Equivalence on $P^\perp L^2(\R)$] \label{cor:U_free} We have the unitary equivalence
$$
     L_+ f = L_- f =   U |\nabla| U^* f \quad \mbox{for $f \in P^\perp L^2(\R) \cap H^1(\R)$},
$$
with the unitary map $U : L^2(\R) \to P^\perp L^2(\R)$ from above. As a consequence, the operators $L_+$ and $L_-$ have purely absolutely continuous spectra given by $\sigma_{\mathrm{ac}}(L_+) = \sigma_{\mathrm{ac}}(L_-) = [0, \infty)$.
\end{cor}

\begin{remark*}
{\em Regarding domain questions, we remark that it is straightforward to check that $U^* f \in H^1(\R)$ for any $f \in H^1(\R)$. }
\end{remark*}

\begin{proof} Since $L_+ f =  L_- f$ for $f \in H^1(\R) \cap P^\perp L^2(\R)$, it suffices to show the claimed identity for the operator $L_+$. Let $f \in H^1(\R) \cap P^\perp L^2(\R)$ be given and let $f_+ = \Pi_+ f$ and $f_- = \Pi_- f$. Recalling that $\overline{Q}_m f_+ \in L^2_+(\R)$ and $Q_m f_- \in L^2_-(\R)$  for $f \in P^\perp L^2(\R)$ together with $\mp \im \frac{d}{dx} = |\nabla|$ on $L^2_{\pm}(\R)$, we apply Lemma \ref{lem:darboux} to conclude that
$$
L_+ f  = Q_m \left ( -\im \frac{d}{dx} \right ) \overline{Q}_m f_+ + \overline{Q}_m \left ( +\im \frac{d}{dx} \right ) Q_m f_- = U |\nabla| U^* f,
$$
which is the claimed identity.
\end{proof}

\subsection{Coercivity Estimate}
We recall the definition of the set of functions $\{ \psi_k \}_{k=0}^{2m}$ in \eqref{def:psi_k}. We prove the following estimate.

\begin{lemma}
Let $f, g \in \dot{H}^{\frac 1 2}(\R)$ satisfy $(\psi_k, f) = (\psi_k, g)=0$ for $0 \leq k \leq 2m$. Then we have
$$
(f, L_+ f) + (g, L_- g) \geq \frac{1}{m+1} \left ( \| f \|_{\dot{H}^{\frac 1 2}}^2 + \| g \|_{\dot{H}^{\frac 1 2}}^2 \right ).
$$
Moreover, the constant $\frac{1}{m+1}$ on the left-hand side is optimal.
\end{lemma}

\begin{proof}
We first consider the quadratic form $(f, L_+ f)$ and let $\widetilde{f} = f \circ \Pi$ be conformal lift of $f$ to $\Ss^1$. Using the stereographic projection $\Pi(e^{\im \theta}) = x$ so that $\frac{2 dx}{1+x^2} = d \theta$, the two conditions $(\psi_{2m-1}, f)= (\psi_{2m}, f)= 0$ become
$$
0= (\psi_{2m-1}, f) = \langle \widetilde{\vphi}_{2m-1}, \widetilde{f} \rangle , \quad 0 = (\psi_{2m}, f) = \langle \widetilde{\vphi}, \widetilde{f} \rangle .
$$
Since $\mathrm{span} \{ \widetilde{\vphi}_{2m-1}, \widetilde{\vphi} \} = \mathrm{span} \{ e_{-m}, e_m \}$, we see that $\widetilde{f} \perp \mathrm{span} \{ e_{-m}, e_m \}$. 

Next,  we use that $\psi_k = \vphi_k$ for $0 \leq k \leq 2m-2$, where $\vphi_k \in H^1(\R)$ are the $L^2$-eigenfunctions of $L_+$ with corresponding eigenvalue $E_k < 0$. Thus, for any $0 \leq k \leq 2m-2$,
$$
0=(\vphi_k, f ) = E_k^{-1} (L_+  \vphi_k, f ) = E_k^{-1} \langle \tL_+ \tvphi_k, \widetilde{f} \rangle .
$$
Since $\mathrm{span} \, \{ \tvphi_0, \ldots, \tvphi_{2m-1}, \tvphi \} = \Lambda^{(0)}_m = \mathrm{span} \, \{ e_ k : |k| \leq m \}$ and $\tL_+ e_k = (|k|-m) e_k$, we conclude that $\widetilde{f} \perp e_k$ for $|k| \leq m-1$. 

In summary, we have found that $\widetilde{f} = \sum_{|k| \geq m+1} f_k e^{\im k \theta}$.  Hence we have
\begin{align*}
(f, L_+ f) & = \langle \widetilde{f}, \tL_+ \widetilde{f} \rangle  = \sum_{|k| \geq m+1} ( |k| - m) |f_k|^2 \\
& \geq \frac{1}{m+1} \sum_{|k| \geq m+1} |k| |f_k|^2 = \frac{1}{m+1} \langle \widetilde{f}, |\nabla|_{\Ss^1} \widetilde{f} \rangle = \frac{1}{m+1} (f, |\nabla| f),
\end{align*}
where we also used the elementary inequality $|k|-m \geq \frac{1}{m+1} |k|$ for $|k| \geq m+1$. Note that equality holds above when $f_k = 0$ for $|k| \geq m+2$, showing that the constant $\frac{1}{m+1}$ is optimal.

The proof of $(g,L_-g) \geq \frac{1}{m+1} ( g, |\nabla| g )$ follows in the same fashion, noticing that $\tL_- \widetilde{g} = (\tL_+ + \widetilde{R}_m) \widetilde{g} = \tL_+ \widetilde{g}$, since it holds that $\widetilde{R}_m \widetilde{g} = 0$ for $\widetilde{g}  = \sum_{|k| \geq m+1} g_k e^{\im k \theta}$ as we recall from Lemma \ref{lem:funk-hecke}.
\end{proof}

\subsection{Spectral Properties of $JL$} Here we prove Theorem \ref{thm:JL}. Recall that we consider the matrix operator
$$
\mathcal{L} = JL = \left [ \begin{array}{cc} 0 & - L_- \\ L_+ & 0  \end{array} \right ]
$$
acting on $H=L^2(\R) \times L^2(\R)$ with operator domain $H^1(\R) \times H^1(\R)$. 

We start with some general observations reflecting the special structure of the problem at hand. Given the projection $P: L^2(\R) \to L^2(\R)$ onto the space spanned by the $L^2$-eigenfunctions of $L_+$ from above, we can define the orthogonal projection 
$$
\mathcal{P} : L^2(\R) \times L^2(\R) \to L^2(\R) \times L^2(\R) \quad \mbox{with} \quad \mathcal{P} = \left [ \begin{array}{cc} P & 0 \\ 0 & P \end{array} \right ] ,
$$
where we set $\mathcal{P}^\perp = \mathds{1}  - \mathcal{P}$. Using that $H= \mathcal{P} H \mathcal{P} \oplus \mathcal{P}^\perp H \mathcal{P}^\perp$, the operator $\mathcal{L}$ acts invariantly on these two subspaces and we have the corresponding decomposition
$$
\mathcal{L} = \mathcal{P} \mathcal{L} \mathcal{P}+ \mathcal{P}^\perp \mathcal{L} \mathcal{P}^\perp .
$$
Recalling that $P L_- P = 0$, we find
$$
\mathcal{P} \mathcal{L} \mathcal{P} = \mathcal{P} \mathcal{L}_1 \mathcal{P} \quad \mbox{with} \quad \mathcal{L}_1 =  \left [ \begin{array}{cc} 0 & 0 \\ A & 0 \end{array} \right ] ,
$$
with the self-adjoint finite-rank operator $A= P L_+ P$ given by $A= \sum_{k=0}^{2m-1} E_k \vphi_k (\vphi_k ,\cdot)$, where $\{ \vphi_0, \ldots, \vphi_{2m-1} \}$ denote the orthonormal set of $L^2$-eigenfunctions of $L_+$ with eigenvalues $E_k$. Let us now suppose that $\mu \in \C$ is an eigenvalue of $\mathcal{L}_1$, i.\,e,~
\be
\left [ \begin{array}{cc} 0 & 0 \\ A & 0 \end{array} \right ] \left [ \begin{array}{c} u \\ v \end{array} \right ] = \mu \left [ \begin{array}{c} u \\ v \end{array} \right ] 
\ee 
for some $u, v \in \mathrm{span} \, \{ \vphi_0, \ldots, \vphi_{2m-1} \}$. Thus $A u = \mu v$ and $0 = \mu u$. If $\mu \neq 0$, then $u=0$ and consequently $v = 0$. Hence $\mu=0$ is the only possible eigenvalue. In this case $Au = 0$ implies that $u = \alpha \vphi_{2m-1}$ with some $\alpha \in \R$ and $u \in \mathrm{span} \, \{ \vphi_0, \ldots, \vphi_{2m-1} \}$ is arbitrary. Thus we have shown
\be
\sigma(\mathcal{L}_1) = \{ 0 \} \quad \mbox{and} \quad \mathrm{ker} ( \mathcal{L}_1) = \mathrm{span} \, \left \{  \left [ \begin{array}{c} \vphi_{2m-1} \\ 0\end{array} \right ], \left [ \begin{array}{c} 0 \\ \vphi_0 \end{array} \right ] , \ldots,  \left [ \begin{array}{c}0  \\ \vphi_{2m-1} \end{array} \right ]   \right \}.
\ee
Furthermore, we readily find that 
$$
\bigcup_{n \geq 1} \mathrm{ker} ( \mathcal{L}_1^n) = \mathrm{ker} ( \mathcal{L}_1^2) = \mathrm{span} \, \left \{  \left [ \begin{array}{c} \vphi_{0} \\ 0\end{array} \right ], \cdots, \left [ \begin{array}{c} \vphi_{2m-1} \\ 0\end{array} \right ], \left [ \begin{array}{c} 0 \\ \vphi_0 \end{array} \right ] , \ldots,  \left [ \begin{array}{c}0  \\ \vphi_{2m-1} \end{array} \right ]   \right \}.
$$
using that the $\mathcal{L}_1^2 = \mathcal{L}_1 \mathcal{L}_1 = 0$ is nilpotent.

Next, we study $\mathcal{P}^\perp \mathcal{L} \mathcal{P}^\perp$. By using the unitary equivalence stated in Theorem \ref{thm:unitaryL2}, we deduce that
\be
\mathcal{P}^\perp \mathcal{L} \mathcal{P}^\perp =  \mathcal{P}^\perp  \mathcal{L}_2  \mathcal{P}^\perp \quad \mbox{with} \quad \mathcal{L}_2 =  \mathcal{U} \left [ \begin{array}{cc} 0 & -|\nabla| \\ |\nabla| & 0 \end{array} \right ] \mathcal{U}^*
\ee
with the unitary map defined as
\be
\mathcal{U} : L^2(\R) \times L^2(\R) \to \mathcal{P}^\perp (L^2(\R) \times L^2(\R)) \quad \mbox{with} \quad \mathcal{U} = \left [ \begin{array}{cc} U & 0 \\ 0 & U \end{array} \right ] ,
\ee 
where $U : L^2(\R) \to P^\perp L^2(\R)$ is the unitary map from Theorem \ref{thm:unitaryL2} above.  By the unitary equivalence and from standard methods (e.\,g.,~using the Fourier transform), it is easy to see that
\be
\sigma(\mathcal{L}_2) = \sigma_{\mathrm{ess}}(\mathcal{L}_2) = \sigma \left ( \left [ \begin{array}{cc} 0 & -|\nabla| \\ |\nabla| & 0 \end{array} \right ] \right ) =  \im \R.
\ee

Collecting the above results, we can now deduce the items (i)--(iv) in Theorem \ref{thm:JL} about $\mathcal{L}=JL$, which completes the proof of Theorem \ref{thm:JL}. \hfill $\Box$

\begin{appendix}

\section{Short Primer on the Half-Wave Maps Equation}

\label{sec:primer}

Here we collect some  background facts about the half-wave maps equation defined in $N \geq 1$ spatial dimensions and target $\Ss^2$, which is the evolution equation given by
\be \label{eq:hwm_n} 
\pt_t \uu + \uu \wedge |\nabla| \uu = 0
\ee
for a function $\uu : [0, T ) \times \R^N \to \Ss^2$.  In what follows, we will give a brief summary about the Hamiltonian structure, symmetries, and scaling. We do not address the question of well-posedness for the Cauchy problem; but see, e.\,g., the recent work by Krieger and Sire \cite{KrSi-16} on small data global well-posedness in the energy-supercritical case of $N \geq 5$ space dimensions. We mention that decent well-posedness results for lower dimensions (in particular, for the energy-critical case $N=1$) are currently not available.

In particular, we discuss below the feature of an explicit {\bf Lorentz boost symmetry} for traveling solitary waves in the energy-critical dimension $N=1$. Moreover, we mention the connection to  {\bf completely integrable systems} in the energy-critical case.

\subsection{Hamiltonian Structure and Linear Momentum}
In any dimension $N \geq 1$, the half-wave maps equation \eqref{eq:hwm_n} can be regarded as an infinite-dimensional Hamiltonian system whose energy functional is found to be
\be
E[\uu] = \frac{1}{2} \int_{\R^N} \uu \cdot |\nabla| \uu \, dx  = c_N \int \! \! \int_{\R^N \times \R^N} \frac{|\uu(x)- \uu(y)|^2}{|x-y|^{N+1}} \, dx \, dy,
\ee
which is well-defined for $\uu \in \dot{H}^{\frac 1 2}(\R^N; \Ss^2)$ and $c_N > 0$ is some constant. To exhibit the Hamiltonian nature of the half-wave maps equation, we endow the phase space $\dot{H}^{\frac 1 2}(\R^N; \Ss^2)$ with the canonical Poisson bracket for $\Ss^2$-valued functions, i.\,e., we set
\be
\{ u_i(x), u_j(y) \} =  \eps_{ijk}  u_k(x) \delta(x-y) ,
\ee
where $\eps_{ijk}$ is the usual antisymmetric Levi-Civit\`a symbol. A moment's reflection shows that \eqref{eq:hwm_n} is equivalent to 
\be \label{eq:uE}
\pt_t \uu = \{ \uu, E \} .
\ee
As a direct consequence we obtain (formally, at least) that the energy $E[\uu(t)] = E[\uu(0)]$ is conserved along the flow.  

Next, let us focus on the energy-critical case with $N=1$ space dimension, where we can make the following interesting formal observations. Indeed, recalling the notion of linear momentum in the classical continuum Heisenberg spin chain (see, e.\,g., \cite{Ha-86}), we introduce the functional
\be
P[\uu] = \int_{\R} \mathbf{A}(\mathbf{u}(x)) \cdot \pt_x \uu(x) \, dx, 
\ee
which plays the role of a {\bf linear momentum}, provided we make the following choice
\be
\mathbf{A}(\uu) = \frac{\eB \wedge \uu}{(\eB \cdot \uu)-1}  \quad \mbox{for any $\uu \in \Ss^2$ with $\uu \neq \eB$},
\ee
where $\eB \in \Ss^2$ is a fixed unit vector (e.\,g., we can take $\eB = \eB_z$). If we assume for the moment that $\uu : \R \to \Ss^2$ avoids the point $\eB$ in a suitable way (e.\,g., we suppose $| \uu(x) - \eB | > \delta$ for all $x \in \R$ and some $\delta > 0$), it is straightforward to check that
\be \label{eq:uP}
\pt_x \uu = \{  \uu, P \} .
\ee 
Therefore, at least for such configurations $\uu$ as above, the functional $P[\uu]$ can be seen as the generator of spatial translations, in analogy to the time translations generated by  $E[\uu]$. Furthermore, the functional $P[\uu]$ has a concrete geometric meaning which is as follows. For definiteness, let us take $\eB = (0,0,1)$ be the north pole on $\Ss^2$ and assume that $\uu : \R \to \Ss^2$  that avoids $\eB$ in the sense above and traces out a closed curve $\gamma$ on $\Ss^2$, i.\,e., we have ~$\uu(+ \infty) = \uu (-\infty)$. By applying Stoke's theorem, we deduce that
\be
P[\uu] = {\oint}_\gamma \mathbf{A} \cdot d \mathbf{x} = \int \! \! \int_{S} ( \nabla \wedge \mathbf{A}) \cdot \mathbf{n} \, d\Sigma = \int \! \! \int_{S} d \Sigma,
\ee
where $S$ is the (oriented) surface on $\Ss^2$ whose boundary is given by the closed curve $\gamma$ and $S$ lies opposite to $\eB$. Hence $P[\uu]$ corresponds to the {\em solid angle} traced out by $S$. Of course, due to different choices for $\eB$, the notion of $P[\uu]$ contains an ambiguity of integer multiples of $4 \pi$. Hence, as the physically reasonable candidate for a conversed quantity corresponding to linear momentum one should rather consider the functional $T[\uu] = e^{\im \frac{1}{2} P[\uu]}$; cf.~\cite{Ha-86} in the setting of classical Heisenberg spin chains.

Finally, we remark that the solitary wave profiles $\QB_v \in \dot{H}^{\frac 1 2}(\R; \Ss^2)$ solving \eqref{eq:uv} can be (formally, at least) seen as critical points of the functional 
\be
S[\uu] = E[\uu] + v P[\uu],
\ee
where the velocity $v \in \R$ is a given parameter. Indeed, a calculation shows that the profile equation \eqref{eq:uv} for $\QB_v$ is formally equivalent to the criticality condition 
\be
\frac{d}{d \eps}  \Big |_{\eps=0} S [ \QB_v + \eps \hh]  = 0
\ee 
for all perturbations $\hh \in C^\infty_c(\R; \R^3)$ with $\hh(x) \in T_{\QB_v(x)} \Ss^2$. In fact, this can be formally seen by writing $\hh = \QB_v \wedge \mathbf{g} \in T_{\QB_v} \Ss^2$ with $\mathbf{g} \in C^\infty_c(\R; \R^3)$. Then we verify that
$$
\frac{d}{d \eps} \Big |_{\eps=0} S [\QB_v + \eps \hh] = \int_{\R} \{ \QB_v , S \} \cdot \mathbf{g} \, dx
$$
Consider now  a traveling solitary wave $\uu(t,x) = \QB_v(x-vt)$. Then $0 = (\pt_t + v \pt_x) \uu = \{ \uu, E + v P \}$ holds, which means that $\{ \QB_v, S\} = 0$. This shows -- on a formal level -- that critical points of $S_v[\uu]$ are given by boosted half-harmonic maps $\QB_v$ and vice versa. 

Of course, it is an interesting open problem to setup a rigorous mathematical framework that justifies the formal observations made above.

\subsection{Lorentz Boosts as Conformal Transformations on $\Ss^2$}
In the energy-critical case of $N=1$ dimensions in the domain, we recall that the half-wave maps equation \eqref{eq:hwm_n} possesses non-trivial traveling solitary waves with finite energy, i.\,e., solutions of the form $\uu(t,x) = \QB_v(x-vt)$ with $|v| < 1$ and where the profile function $\QB_v \in \dot{H}^{\frac 1 2}(\R; \Ss^2)$ solves the equation
\be \label{eq:uv_app}
\QB_v \wedge |\nabla| \QB_v - v \pt_x \QB_v = 0.
\ee 
As shown by Theorem \ref{thm:vsmall} above, solutions $\QB_v \in \dot{H}^{\frac 1 2}(\R; \Ss^2)$ can be obtained  by making the transform
\be \label{eq:QB_lorentz}
\QB(x) = (f(x), g(x), 0) \mapsto \QB_v(x) = \left ( \sqrt{1-v^2} f(x), \sqrt{1-v^2} g(x), v \right ) ,
\ee
where $\QB = (f,g,0) \in \dot{H}^{\frac 1 2}(\R; \Ss^2)$ is a half-harmonic map with values in the equatorial plane. It is easy to check that $\QB_v$ defined as above does indeed solve \eqref{eq:uv_app}. Indeed, it is clear that $|\QB_v|^2 = (1-v^2) (f^2 + g^2) + v^2 = 1$. Furthermore, if we let $\alpha = \sqrt{1-v^2}$, we note that
\be
\pt_x \QB_v = ( \alpha \pt_x f, \alpha \pt_x g, 0), \quad \QB_v \wedge |\nabla| \QB_v = ( - \alpha v |\nabla| g, \alpha v |\nabla| f, 0 ),
\ee
using that $f |\nabla| g - g |\nabla| f = 0$. Now we recall that $|\nabla| f = -\pt_y \mathrm{Re} \, \mathcal{B}(x+\im y) |_{y=0}$ and $|\nabla| g = -\pt_y \mathrm{Im} \, \mathcal{B}(x+ \im y) |_{y=0}$, where $\mathcal{B} : \overline{\C}_+ \to \C$ is some finite Blaschke product. By the Cauchy--Riemann equation satisfied by $\mathcal{B}(z)$ and the boundary regularity at hand, we conclude that $|\nabla| f = \pt_x g$ and $|\nabla| g = -\pt_x f$. Therefore $\QB_v \wedge |\nabla| \QB_v = ( \alpha v \pt_x f, \alpha v \pt_x g,0) = v \pt_x \QB_v$, which is the same as \eqref{eq:uv_app}.

Now, we provide a more geometrical interpretation of the calculation done above. In fact, we can regard the transformation $\QB \to \QB_v$ above as a Lorentz boost realized by a M\"obius transform on $\Ss^2$, identified with the so-called {\bf celestial sphere} in special relativity; see the monograph by Penrose and Rindler \cite{PeRi-84}. Indeed, let $\mathcal{S}$ denote the stereographic projection from $\Ss^2$ to the projective complex plane $\hat{\C} = \C \cup \{ \infty \}$ given by
$$
 \mathcal{S}(x) = \frac{x_1 + \im x_2}{1-x_3} .
$$
Note that the north pole $N=(0,0,1) \in \Ss^2$ gets mapped to $\infty$. Its inverse map is found to be 
$$
\mathcal{S}^{-1}(z) = \left ( \frac{2 \, \mathrm{Re} \, z}{1 + |z|^2}, \frac{2 \, \mathrm{Im} \, z}{1 + |z|^2}, \frac{|z|^2-1}{1+|z|^2} \right ).
$$
Now from \cite{PeRi-84} we recall that the Lorentz boost in $z$-direction with velocity $v \in (-1,1)$ on the celestial sphere $\Ss^2$, is given by the transform $\Lambda_v : \hat{\C} \to \hat{\C}$ with 
$$
z \mapsto \Lambda_v (z)= e^{-\chi} z
$$
acting on the projective complex plane $\hat{\C}$, where the parameter $\chi = \mathrm{artanh} \, v$ denotes the so-called {\bf rapidity}. Thus we find that applying the corresponding Lorentz boost to the half-harmonic map $\QB=(f,g,0)$ yields 
\begin{align*}
(\mathcal{S}^{-1} \circ \Lambda_v \circ \mathcal{S})(\QB) & = \left ( \frac{2 e^{-\chi} f}{1+ e^{-2 \chi}}, \frac{2 e^{-\chi} g}{1+e^{-2 \chi}}, \frac{e^{-2 \chi} -1}{1+ e^{-2 \chi}} \right ) = \left ( \frac{f}{\cosh \chi} , \frac{g}{\cosh \chi}  , \tanh \chi \right ).
\end{align*}
Since $\tanh \chi = v$ and $\cosh \chi = \frac{1}{\sqrt{1-v^2}}$, we see that this is exactly the same as $\QB_v$ given in \eqref{eq:QB_lorentz} above. 

\subsection{Connection to Completely Integrable Spin Systems}
In the energy-critical case $N=1$, the half-wave maps equation \eqref{eq:hwm_n} arises as in the continuum limit of discrete {\bf Calogero--Moser (CM) spin systems}. These CM systems  that have been intensively studied in the theory of completely integrable systems and they arise as a natural extensions including spin as internal degrees of freedom. We refer to \cite{GiHe-84,BlLa-99} and references therein for further background on CM systems with spin. We also remark that the classical CM spin systems can be (formally, at least) obtained by taking a suitable semi-classical limit of the quantum spin chains related to the celebrated {\bf Haldane--Shastry spin chains} (\cite{Ha-88, Sha-88}), which are exactly solvable quantum models. 

Let us briefly sketch the formal relation of \eqref{eq:hwm_n} to the dynamics of Calogero-Moser spin systems; a rigorous analysis will be given in \cite{BuLe-17}. The staring point is the lattice $h \Z$ with equidistant spacing $h > 0$,  where we attach a classical spin $\mathbf{S}(x_k) \in \Ss^2$ to each lattice site $x_k = h k$ with $k \in \Z$. Intriguing examples for such a classical spin system are given by so-called Calogero--Moser type models, which are spin Hamiltonian with long-range interactions of inverse-square type, e.\,g.,  
\be
H_{CS} = \frac 1 2 \sum_{k  \neq l} \frac{1 - \mathbf{S}(x_k) \cdot \mathbf{S}(x_l)}{(x_k - x_l)^2} .
\ee 
We remark that, in greater generality, the Hamiltonian $H_{CS}$ can also contain a kinetic energy term $T=\sum_{k} \frac{p_k^2}{2m_k}$ with momenta $p_k$ (in closer relation to the original system introduced by J.~Moser in his seminal work \cite{Mo-75}). But here we consider the ``infinite mass limit'' with $m_k \to +\infty$ and hence only the spin part in $H_{CS}$ is relevant.

The evolution equation for the time-dependent spin configuration $\mathbf{S} = \mathbf{S}(t,x_k) \in \Ss^2$ generated by $H_{CS}$ is found to be $\pt_t \mathbf{S} = \mathbf{S} \wedge H'(\mathbf{S})$ with $H'(\mathbf{S}) = \sum_{k \neq l} \frac{\mathbf{S}(x_k) - \mathbf{S}(x_l)}{(x_k - x_l)^2}$. When taking the continuum limit $h \to 0^+$, we obtain (formally, at least) the half-wave maps equation \eqref{eq:hwm_n} as an effective evolution for the field $\uu = \uu(t,x) \in \Ss^2$ defined on the real line $\R$ obtained from $h \Z$ as $h \to 0^+$. For the case of nearest-neighbors interaction given by the Heisenberg Hamiltonian $H= \sum_{|k-l|=1} (1- \mathbf{S}(x_k) \cdot \mathbf{S}(x_l)$, such a continuum limit leading to the Landau-Lifschitz equation $\pt_t \uu = \uu \wedge \DD \uu$ has been rigorously proven in \cite{SuSuBa-86}.  Of course, the case of (critical) long-range interactions present in $H_{CS}$ is analytically  more subtle and it will be rigorously addressed in \cite{BuLe-17}.

We mention that choice of the spin-spin interaction $1/(x_k-x_l)^2$ in the definition of $H_{CS}$ is often referred to as the rational case; other natural completely integrable choices are $1/\sin^2(x_k -x_l)$ (trigonometric case) and $1/\sinh^2(x_k-x_l)$ (hyperbolic case), which however formally lead to the the same continuum dynamics as $h \to 0^+$. 

\section{Regularity Theory}

\label{sec:regularity}

Let $v \in \R$ and $\QB \in \dot{H}^{\frac{1}{2}}(\R;\S^2)$ be a distributional solution to
\begin{equation}\label{eq:ourel:old}
 \QB \wedge \Dx \QB - v \pt_x \QB = 0 \quad \mbox{in $\R$.}
\end{equation}
The main result of this section is higher regularity for solutions of \eqref{eq:ourel:old}.
\begin{thm2}\label{th:regularity}
Let $\QB \in \dot{H}^{\frac 1 2}(\R; \Ss^2)$ be as above with  $|v| < 1$. Then $\QB \in C^\infty \cap \dot{H}^2$. 
\end{thm2}
For $v = 0$, the solution $\QB$ from \eqref{eq:ourel:old} is a half-harmonic map from $\R$ into the sphere $\S^2$. H\"older-continuity of $\uu$ was shown in the pioneering work by F. Da Lio and T. Rivi\`{e}re~\cite{DR1dSphere}. Another proof was given by V. Millot and Y. Sire~\cite{Millot-Sire-2015}, who identified half-harmonic maps into spheres with the partially free boundary harmonic maps  studied via reflection arguments by Ch.~Scheven~\cite{Scheven-2006}. For the general situation with $v \neq 0$, it is far from clear how to extend Scheven's reflection argument. Instead, we show how to treat the case $v \in (-1,1)$ via an extension of the techniques by Da Lio and Rivi\`{e}re and variations thereof; see \cite{DR1dMan,SNHarmS10,DndMan,DSpalphSphere,Schikorra-epsilon,BPSknot12,SchikorraCPDE14}. 

The reader should be warned that some parts of this section are rather technical and thus they maybe skipped at first -- or even second -- reading.

First let us fix some notation as follows. We use $\Rz$ to denote the Hilbert transform, normalized so that
\begin{equation}\label{eq:HTconstant}
 \|\Rz f \|_{L^2(\R)} = \|f \|_{L^2(\R)}, \quad \Rz (\Dx f) = \pt_x f.
\end{equation}
In particular, if we write $\uu = \QB$ from now and using the Hilbert transform, we can write \eqref{eq:ourel:old} as
 \begin{equation}\label{eq:ourel}
 \uu \wedge \Dx \uu -v \Rz\Dx \uu = 0 \quad \mbox{in $\R$}.
\end{equation}
Furthermore, we introduce the Riesz potential $\lapms{\sigma}$ in one dimension with $\sigma \in (0,1)$, which is the inverse of the fractional Laplacian $\laps{\sigma}$ such that $\lapms{\sigma} \laps{\sigma} f = f$ for all $f$ in the Schwartz space. We recall the potential representation
\[
 \lapms{\sigma} f(z)= c_\sigma \int_{\R} |x-z|^{\sigma-1}\, f(z)\ dz.
\]
It is useful to observe that $\Rz = \partial_x \lapms{1}$ on a suitable set of functions. For two operators $T$ and $S$, we denote the commutator
\begin{equation}\label{eq:commutatornot}
 [T,S] = T\circ S - S\circ T.
\end{equation}
In what follows, we use $X \aleq Y$ to mean $X \leq C Y$, where $C >0$ is some constant that only depends $|v| < 1$ and possibly some fixed exponents.

\subsection{Preliminaries and Morrey Space Estimates}
By Lagrange's identity
\begin{equation}\label{eq:lagrange}
 |\ve{a}|^2\, |\ve{b}|^2 = |\langle \ve{a},\ve{b} \rangle|^2 + |\ve{a} \wedge \ve{b}|^2 \quad \forall \ve{a},\ve{b} \in \R^3,
\end{equation}
we see that for $\uu(x) \in \dot{H}^{\frac 1 2}(\R;\S^2)$, 
\[
 |\lapv \uu(x)| \leq |\uu(x) \cdot \lapv \uu(x)| + |\uu(x) \wedge \lapv \uu(x)|.
\]
In order to exploit this pointwise inequality for a regularity bounds for \eqref{eq:ourel}, we first show that $|\uu \wedge \lapv \uu|$ can be estimated in terms of $|\uu \wedge \lapv \uu - \Rz\lapv \uu|$. This is accomplished by the following Lemma, which crucially uses $|v| < 1$.

\begin{lemma}\label{la:TfBfest}
Let $\uu \in L^\infty(\R;\S^2)$, $p \in (1,\infty)$, and $v \in \R$. For $\ff \in L^p(\R;\R^3)$ we set
\[
 T_{\uu}\ff := \uu\wedge \ff - v \Rz\ff.
\]
Then if $|v| < 1$, with constants depending only on $v$ and $p$,
\[
 \|\ff\|_{L^p(\R)} \aleq \|\uu \cdot \ff\|_{L^p(\R)} + \|T_{\uu}\ff\|_{L^p(\R)} +     \|[\Rz,\uu\cdot] \ff\|_{L^p(\R)} + \|[\Rz,\uu\wedge] \ff\|_{L^p(\R)}.
\]
More generally, on any interval $I \subset \R$,
\[
\begin{split}
  \|\ff\|_{L^p(I)} \aleq &\|\uu \cdot \ff\|_{L^p(I)} + \|\Rz(\uu \cdot \ff)\|_{L^p(I)} + \|T_{\uu}\ff\|_{L^p(I)}+ \|\Rz(T_{\uu}\ff)\|_{L^p(I)}\\
  &+     \|[\Rz,\uu\cdot] \ff\|_{L^p(I)} +   \|[\Rz,\uu\wedge] \ff\|_{L^p(I)}.
\end{split}
  \]
\end{lemma}
\begin{proof}
We set
\[
 B_{\uu} \ff := \uu \cdot \ff, \quad A_{\uu}\ff := \uu \wedge \ff.
\]
Since $\uu(x) \in \S^2$ pointwise a.e. and with Lagrange's identity \eqref{eq:lagrange},
\[
\begin{split}
 \| \ff \|_{L^p(I)} \leq& \|B_{\uu} \ff\|_{L^p(I)} + \|A_{\uu}\ff\|_{L^p(I)}\\
 \leq & \|B_{\uu} \ff\|_{L^p(I)} + \|T_{\uu} \ff\|_{L^p(I)} + |v| \|\Rz\ff\|_{L^p(I)}.
\end{split}
 \]
We iterate this estimate, applying it to the last term $|v| \|\Rz\ff\|_{L^p(I)}$,
\[
 \leq  \|B_{\uu} \ff\|_{L^p(I)} + \|T_{\uu} \ff\|_{L^p(I)} + |v|  \brac{ \|B_{\uu} \Rz \ff\|_{L^p(I)} + \|T_{\uu}\Rz \ff\|_{L^p(I)} + |v|\,  \|\Rz\Rz\ff\|_{L^p(I)}}.
 \]
Since $\Rz \circ \Rz = -1$, this is equal to
 \[
 =\|B_{\uu} \ff\|_{L^p(I)} + \|T_{\uu} \ff\|_{L^p(I)} +  |v|  \|B_{\uu} \Rz \ff\|_{L^p(I)} + |v|\,  \|T_{\uu}\Rz \ff\|_{L^p(I)} + |v|^2 \| \ff \|_{L^p(I)}.
 \]
Absorbing the term $|v|^2 \| \ff \|_{L^p(I)}$ and, using that $|v| <1$ and dividing by $(1-v^2)$, we have shown
\[
 \| \ff \|_{L^p(I)} \leq \frac{1}{(1-v^2)}\brac{\|B_{\uu} \ff\|_{L^p(I)} + \|T_{\uu} \ff\|_{L^p(I)} +  |v|\,  \|B_{\uu} \Rz \ff\|_{L^p(I)} + |v|\,  \|T_{\uu}\Rz \ff\|_{L^p(I)}}.
\]
Using the commutator notation \eqref{eq:commutatornot}, the right-hand side can be written as the upper bound
\[
\begin{split}
 \aleq & \|B_{\uu} \ff\|_{L^p(I)} + \|T_{\uu} \ff\|_{L^p(I)} +\|\Rz(B_{\uu} \ff)\|_{L^p(I)} +  \|[\Rz,\uu\cdot] \ff\|_{L^p(I)} + \|\Rz(T_{\uu} \ff)\|_{L^p(I)} \\
 &    +   \|[\Rz,\uu\wedge] \ff\|_{L^p(I)}.
\end{split}
 \]
If $I = \R$, then we can use the boundedness of the Hilbert transform $\Rz$ on $L^p(\R)$ to conclude.
\end{proof}

In view of \eqref{eq:ourel}, when estimating $\uu \wedge \lapv \uu - \Rz\lapv \uu$ we notice that \[
 \lapv (\uu \wedge \lapv \uu - \Rz\lapv \uu) = [\lapv , \uu \wedge] (\lapv \uu).
\]
This induces an estimate in the Hardy space $\mathcal{H}^1(\R)$.
\begin{lemma}\label{la:hardy}
Let $\uu$ satisfy the equation \eqref{eq:ourel}, then
\begin{equation}\label{la:hardy:first}
\|\lapv \brac{\uu \wedge \lapv \uu - v \Rz[\lapv \uu]}\|_{\mathcal{H}^1(\R)} \aleq  \|\lapv \uu \|_{L^2(\R)}^2.
\end{equation}
More generally, for any $\sigma \in [0,\frac{1}{2}]$,
\begin{equation}\label{la:hardy:second}
 \|\laps{\sigma} \brac{\uu\wedge \lapv \uu - v \Rz[\lapv \uu]}\|_{L^{\frac{2}{1+2\sigma}}(\R)} \aleq \|\lapv \uu \|_{L^2(\R)}^2.
\end{equation}
\end{lemma}
\begin{proof}
Estimate \eqref{la:hardy:second} follows from \eqref{la:hardy:first} and the Sobolev embedding
\[
 \|\laps{\sigma} f\|_{L^{\frac{2}{1+2\sigma}}(\R)} \aleq \|\lapv f\|_{\mathcal{H}^1(\R)} \quad \forall \sigma \in [0,1].
\]
For \eqref{la:hardy:first}, by the duality between the Hardy space $\mathcal{H}^1$ and $BMO$, it suffices to prove for any scalar function $\varphi \in C_c^\infty(\R)$,
\[
 \mathcal{I} :=\int_{\R} \brac{\uu \wedge \lapv \uu + \nu \Rz[\lapv \uu]}\ \lapv \varphi \aleq  \|\lapv \uu \|_{L^2(\R)}^2\ [\varphi]_{BMO}.
\]
To see the latter, we write in view of \eqref{eq:ourel} and integrating by parts
\[
 \mathcal{I} =  \int_{\R} \uu\wedge \lapv \uu\ \lapv \varphi - \uu\wedge \Dx \uu\ \varphi = \int_{\R} \brac{\uu\, \lapv \varphi - \lapv (\uu\, \varphi)}\wedge \lapv \uu.
 \]
Now we add the term $(\lapv \uu \wedge \lapv \uu)  \varphi  \equiv 0$ to get
 \[
\mathcal{I} = \int_{\R} \brac{\lapv \uu\ \varphi + \uu\ \lapv \varphi - \lapv (\uu\wedge \varphi)}\wedge  \lapv \uu.
\]
Setting $H_\sigma(a,b) := \laps{\sigma} (ab) - \laps{\sigma}a\ b - a \laps{\sigma} b$, we have shown that 
\[
\mathcal{I} = -\int_{\R} H_{\frac{1}{2}} (\uu, \varphi) \wedge \lapv \uu\\
 \]
Now the claim follows from \cite[(5.29)]{Schikorra-epsilon}, see also \cite[Theorem~8.2]{Lenzmann-Schikorra-commutators}, where is is shown that
\[
 \|H_{\frac{1}{2}} (\uu, \varphi)\|_{L^2(\R)} \aleq \|\lapv \uu \|_{L^2(\R)}\ [\varphi]_{BMO}.
\]
\end{proof}

From Lemma~\ref{la:TfBfest} and Lemma~\ref{la:hardy} we obtain the following estimate which, localized to small balls, is the main growth estimate of the equation and will lead us in Proposition~\ref{pr:below1-est} below to obtain $W^{\sigma,p}_{loc}$-regularity for any $\sigma \in (0,1)$ and $p \in (1,\infty)$.

\begin{prop}\label{pr:scalingintegrability}
Assume $\uu$ is as in Theorem~\ref{th:regularity} and $|v| < 1$. Then for any $s \in [\frac{1}{2},1)$,
\[
 \|\laps{s} \uu \|_{L^{\frac{1}{s}}(\R)} \aleq \|\lapv \uu\|_{L^2(\R)}^2 \aleq \|\laps{s} \uu \|_{L^{\frac{1}{s}}(\R)}^2 .
\]
\end{prop}
\begin{proof}
The second estimate is just Sobolev inequality. For the first estimate, in view of Lemma~\ref{la:TfBfest} we note
\[
 \|\laps{s} \uu \|_{L^{\frac{1}{s}}} \aleq \|\uu \cdot \laps{s} \uu\|_{L^{\frac{1}{s}}(\R)} + \|T_{\uu}\laps{s}\uu\|_{L^{\frac{1}{s}}(\R)} +  \|[\Rz,\uu\wedge] \laps{s} \uu\|_{L^{\frac{1}{s}}(\R)}
\]
with $T_{\uu}$ as in Lemma \ref{la:TfBfest} above. Since $\uu \cdot \uu = 1$ a.\,e., the first term on right side is
\[
 \uu \cdot \laps{s} \uu = \frac{1}{2} H_{s} (\uu\cdot,\uu),
\]
where $H_s = \laps{s}(ab) - a\laps{s}b-\laps{s}a\ b$ denotes again the three-term commutator by Da Lio and Rivi\`{e}re introduced above. We thus have for any $s < 1$, see \cite{Schikorra-epsilon} or \cite[Theorem~8.2]{Lenzmann-Schikorra-commutators}, combined with the Sobolev inequality,
\[
 \|\uu \cdot \laps{s} \uu\|_{L^{\frac{1}{s}}(\R)} \aleq \|\laps{s/2} \uu\|_{L^{\frac{2}{s}}(\R)}^2 \aleq \|\lapv \uu\|_{L^2(\R)}^2.
\]
As for the second term, we set $\sigma := s-\frac{1}{2}$ and observe
\[
 \|T_{\uu}\laps{s}\uu\|_{L^{\frac{1}{s}}(\R)} = \| \uu \wedge \laps{\frac{1}{2}+\sigma} \uu - v \Rz[\laps{\frac{1}{2}+\sigma} \uu]\|_{L^{\frac{1}{s}}}.
\]
With the help of Lemma~\ref{la:hardy} we can bound this by
\[
\aleq \|\lapv \uu\|_{L^2(\R)}^2 + \| [\laps{\sigma} ,\uu \wedge]( \laps{\frac{1}{2}} \uu) \|_{L^{\frac{1}{s}}(\R)}.
\]
Note that the second term vanishes if $s = \frac{1}{2}$ and $\sigma = 0$. For $\sigma \neq 0$, with the Coifman-Meyer or Kenig-Ponce-Vega type estimates in \cite[Theorem 8.1.]{Lenzmann-Schikorra-commutators}, for some $\nu \in (\sigma,\frac{1}{2})$,
\[
\aleq \|\lapv \uu\|_{L^2(\R)}^2 + \|\laps{\nu} \uu \|_{L^{\frac{1}{\nu}}(\R)} \| \laps{s-\nu} \uu \|_{L^{\frac{1}{s-\nu}}(\R)} \aleq \|\lapv \uu\|_{L^2(\R)}^2,
\]
where the last step follows from Sobolev's embedding. Finally, we treat the third term. If $s = \frac{1}{2}$, we use the Coifman--Rochberg--Weiss Theorem \cite{Coifman-Rochberg-Weiss-1976}, see also \cite[Theorem 4.1]{Lenzmann-Schikorra-commutators}, to find that
\[
 \|[\Rz,\uu\wedge] \lapv \uu\|_{L^{2}(\R)} \aleq [\uu]_{BMO}\ \|\lapv \uu \|_{L^2(\R)} \aleq \|\lapv \uu \|_{L^2(\R)}^2.
\]
If $s \in (\frac{1}{2},1)$, with the Coifman--Meyer or Kenig--Ponce--Vega type estimates in \cite[Theorem 6.1.]{Lenzmann-Schikorra-commutators}, 
if we choose $\theta \in (s-\frac{1}{2},s)$ together with Sobolev's embedding, we get
\[
  \|[\Rz,\uu\wedge] \lapv \uu\|_{L^{2}(\R)}  \aleq \|\laps{\theta} \uu\|_{L^{\frac{1}{\theta}}(\R)}\ \|\laps{s-\theta}\uu\|_{L^{\frac{1}{s-\theta}}(\R)} \aleq \|\lapv \uu\|_{L^2(\R)}^2.
\]
\end{proof}

Finally, we will need the following estimate on Riesz potentials acting on Morrey spaces due to Adams (see, e.\,g., \cite[Theorem 3.1]{Adams75}).

\begin{thm2}[Adams' Sobolev inequality on Morrey spaces]\label{th:adams}
The Morrey space $L^{p,\lambda}(D)$ for $D \subset \R$ is defined via its norm,
\[
 \|f\|_{L^{p,\lambda}(D)} := \sup_{r>0,x\in\R}  r^{\frac{\lambda-1}{p}}\, \|f\|_{L^{p}(B(r,x))}.
\]
Then for any $\lambda \in (0,1]$, $s_2 \in (0,s_1)$,
\[
 \|\laps{s_2} f\|_{L^{p_2,\lambda}(\R)} \aleq C_{s_1,s_2,p_1,p_2,\lambda}   \|\laps{s_1} f\|_{L^{p_1,\lambda}(\R)},
\]
where $p_1,p_2 \in (1,\infty)$ so that
\[
 \frac{1}{p_2} = \frac{1}{p_1}-\frac{s_1-s_2}{\lambda}.
\]
\end{thm2}

The following is essentially the localized version of Proposition~\ref{pr:scalingintegrability} for $s = \frac{1}{2}$.

\begin{prop}\label{pr:morreyiteration}
Let $\uu$ be as in Theorem~\ref{th:regularity} with $|v| < 1$. Assume moreover for some $\lambda \in(0,1]$,
\begin{equation}\label{eq:morreycondition}
 \|\lapv \uu\|_{L^{2,\lambda}(\R)} < \infty.
\end{equation}
Then for any $B(x_0,\rho) \subset \R^n$, and any $k_0 \geq 10$,
 \[
 \begin{split}
\| \lapv \uu\|_{L^{2,\lambda}(B(x_0,\rho))} 
 \aleq&\ 2^{k_0\frac{1-\lambda}{2}} \|\lapv \uu\|_{L^2(B(x_0,2^{k_0} \rho))} \ \|\lapv \uu\|_{L^{2,\lambda}(B(x_0,2^{k_0}\rho))} \\
 &+\|\lapv \uu\|_{L^2(\R)}\ \sum_{k=k_0}^\infty 2^{-k\frac{\lambda}{2}}\ \|\lapv \uu\|_{L^{2,\lambda}(B(x_0,2^{k}\rho))}.
 \end{split}
\]
\end{prop}
\begin{proof}
Fix $B(x,r) \subset B(x_0,\rho)$. 
Recall that
\[
 T_{\uu}\lapv \uu := \uu\wedge \lapv \uu - \nu \Rz[\lapv \uu].
\]
By Proposition~\ref{la:TfBfest},
\[
\begin{split}
 \|\lapv \uu\|_{L^2(B(x,r))}   \aleq & \|\uu \cdot \lapv \uu\|_{L^2(B(x,r))} + \|\Rz(\uu \cdot \lapv \uu)\|_{L^2(B(x,r))}\\
 & +    \|[\Rz,\uu \cdot] \lapv \uu\|_{L^2(B(x,r))} + \|[\Rz,\uu\wedge] \lapv \uu\|_{L^2(B(x,r))} \\
  &+\|T_{\uu} \lapv \uu\|_{L^2(B(x,r))}+\|\Rz(T_{\uu}\lapv \uu)\|_{L^2(B(x,r))}.
\end{split}
\]
We now gather the following facts.
\begin{enumerate}
\item Since $|\uu| \equiv 1$,
\[
 \uu \cdot \lapv \uu = -\frac{1}{2} H_{\frac{1}{2}} (\uu,\uu) := -\frac{1}{2}\brac{\lapv(\uu \cdot \uu) - \lapv \uu \cdot \uu - \uu \lapv \uu}.
\]
\item For any $\alpha \in (0,\frac{1}{2})$, with the estimates for three-commutators, see \cite[Theorem~8.2]{Lenzmann-Schikorra-commutators}, and \cite[Theorem 6.1]{Lenzmann-Schikorra-commutators}, respectively,
\[
\|H_{\frac{1}{2}} (\uu,\uu)\|_{L^2(\R)} +
 \|[\Rz,\uu \cdot] \lapv \uu\|_{L^2(\R)} + \|[\Rz,\uu\wedge] \lapv \uu\|_{L^2(\R)}
 \]
 \[
 \aleq \|\laps{\alpha} \uu \|_{L^{\frac{1}{\alpha}}(\R)}\ \|\laps{\frac{1}{2}-\alpha} \uu \|_{L^{\frac{2}{1-2\alpha}}(\R)} \aleq \|\lapv \uu \|_{L^{2}(\R)}\ \| \laps{\frac{1}{2}-\alpha} \uu \|_{L^{\frac{2}{1-2\alpha}}(\R)}.
\]
The last inequality comes from Sobolev's inequality.
\item Using Adams' estimate, Theorem~\ref{th:adams}, 
\[
  \|\laps{\frac{1}{2}-\alpha} \uu \|_{L^{\frac{1-2\alpha}{2}}(B(x,R))} \aleq R^{-\frac{\lambda-1}{2}} \|\lapv \uu \|_{L^{2,\lambda}(\R)}
\]
\end{enumerate}
Now, for any $\ell_0 \geq 2$, we pick a generic cutoff function $\eta_{B(x,2^{\ell_0} r)} \in C_c^\infty(B(x,2^{\ell_0} r))$ with $\eta_{B(x,2^{\ell_0} r)} \equiv 1$ in $B(x,2^{\ell_0-1} r)$. We conclude
\[
\begin{split}
 &\|\uu \cdot \lapv \uu\|_{L^2(B(x,r))} + \|\Rz(\uu \cdot \lapv \uu)\|_{L^2(B(x,r))}\\
 \aleq &\|\uu \cdot \lapv \uu\|_{L^2(B(x,2^{\ell_0}r))} + \|[\Rz,(1-\eta_{B(x,2^{\ell_0} r)})]\uu \cdot \lapv \uu\|_{L^2(B(x,r))}
\end{split}
 \]
and by the disjoint support (for details on this kind of arguments, see \cite{Schikorra-epsilon}) we find
\[
 \aleq \|\uu \cdot \lapv \uu\|_{L^2(B(x,2^{\ell_0}r))} + \sum_{k=\ell_0}^\infty 2^{-\frac{k}{2}} \|\uu \cdot \lapv \uu\|_{L^2(B(x,2^k r))}
\]
Localizing the above facts (again, for details we refer to \cite{Schikorra-epsilon}) we obtain for any $k_0 \geq \ell_0 +1$,
\[
\begin{split}
 &r^{\frac{\lambda-1}{2}}\brac{\|\uu \cdot \lapv \uu\|_{L^2(B(x,r))} + \|\Rz(\uu \cdot \lapv \uu)\|_{L^2(B(x,r))}}\\
 \aleq\ &2^{k_0\frac{1-\lambda}{2}} \|\lapv \uu\|_{L^2(B(x,2^{k_0}r))}\ \|\lapv \uu \|_{L^{2,\lambda}(B(x,2^k \rho))}\\
 &+  \|\lapv \uu \|_{L^2(\R)}\sum_{k=k_0}^\infty 2^{-k\frac{\lambda}{2}} \|\lapv \uu\|_{L^2(B(x,2^k r))}.
\end{split}
 \]
In the same way we obtain estimates for 
\[
\begin{split}
 &r^{\frac{\lambda-1}{2}} \brac{ \|[\Rz,\uu \cdot] \lapv \uu\|_{L^2(B(x,r))} + \|[\Rz,\uu\wedge] \lapv \uu\|_{L^2(B(x,r))}}\\
  \aleq\ &2^{k_0\frac{1-\lambda}{2}} \|\lapv \uu\|_{L^2(B(x,2^{k_0}r))}\ \|\lapv \uu \|_{L^{2,\lambda}(B(x,2^k \rho))}\\
 &+  \|\lapv \uu \|_{L^2(\R)}\sum_{k=k_0}^\infty 2^{-k\frac{\lambda}{2}} \|\lapv \uu\|_{L^2(B(x,2^k r))}.
\end{split}
 \]
Finally, we need to estimate
\[
\|T_{\uu} \lapv \uu\|_{L^2(B(x,r))}+\|\Rz(T_{\uu}\lapv \uu)\|_{L^2(B(x,r))} \\
\]
For this we employ Lemma~\ref{la:hardy} and get (by a localization argument) the same estimate.
\end{proof}

From Proposition~\ref{pr:morreyiteration} we obtain the following result.
\begin{prop}\label{pr:below1-est}
Let $\uu$ be as above. For any $s < 1$ and any $p \in (1,\infty)$, we have
\[
 \laps{s} \uu \in L^p_{loc}(\R) .
\]
In addition, for any $\Lambda > 0$ there is a $C(\Lambda,\uu,p,s) > 0$ such that
\begin{equation}\label{eq:Lpest}
 \sup_{I \subset \R, |I| \leq \Lambda} \|\laps{s} \uu\|_{L^p(I)} \leq C(\Lambda,\uu),
\end{equation}
where the supremum is over intervals of size less or equal to $\Lambda$.

In particular, we have the H\"older continuity $\uu \in C^{\mu}(\R)$ for any $\mu \in (0,1)$.
\end{prop}
\begin{proof}
Let us first establish \eqref{eq:Lpest}. We begin with the estimate from Proposition~\ref{pr:morreyiteration} for $\lambda_0  := 1$. Then \eqref{eq:morreycondition} is satisfied since $\lapv \uu \in L^2(\R)$. For any $\sigma > 0$ and any $\eps > 0$, we can find $k_0 \geq 10$ and $\rho_0 > 0$ so that
\begin{equation}\label{eq:radiusepsilon}
 2^{-k_0 \frac{\sigma}{2}}\|\lapv \uu\|_{L^2(\R)}  + \sup_{x_0 \in \R} 2^{k_0\frac{1-\lambda_0}{2}} \|\lapv \uu\|_{L^2(B(x_0,2^{k_0} \rho_0))} \leq \eps.
\end{equation}
The estimate from Proposition~\ref{pr:morreyiteration} becomes
 \[
 \begin{split}
\| \lapv \uu\|_{L^{2,\lambda_0}(B(x_0,\rho))} 
 \aleq\ & \eps\ \|\lapv \uu\|_{L^{2,\lambda}(B(x_0,2^{k_0}\rho))} \\
 &+\eps\ \sum_{k=k_0}^\infty 2^{-k\frac{\lambda-\sigma}{2}}\ \|\lapv \uu\|_{L^{2,\lambda}(B(x_0,2^{k}\rho))},
 \end{split}
\]
which holds for any $x_0 \in \R$ and any $\rho \leq \rho_0$. For small enough $\eps > 0$, using an iteration scheme for $\rho_k = 2^{-k}\rho_0$ (we refer e.g. to \cite[Appendix 4]{Schikorra-epsilon}), we find some $\lambda_1 < 1$ so that \eqref{eq:morreycondition} holds for this $\lambda_1$ (we need to ensure \eqref{eq:morreycondition} only for small radii, since $\lapv \uu \in L^2(\R)$). 

We repeat this argument: we employ Proposition~\ref{pr:morreyiteration} with this $\lambda_1 < \lambda_0$, choosing a radius $\rho_1 < \rho_0$ even smaller so that \eqref{eq:radiusepsilon} is satisfied for $\lambda_1$ instead of $\lambda_0$, etc. \

We obtain a decreasing sequence of $(\lambda_i)_{i \in \N}$ in $(0,1]$ for which \eqref{eq:morreycondition} is satisfied. Actually, retracing the arguments in \cite[Appendix 4]{Schikorra-epsilon} (see also \cite{BPSknot12}) we observe that $\lambda_i \to 0$ as $i \to \infty$. Thus we obtain
\[
 \lapv \uu \in L^{2,\lambda}(\R) \quad \forall \lambda \in (0,1).
\]
In view of Proposition~\ref{pr:scalingintegrability}, using a suitable localization argument, we obtain that
\[
 \laps{s} \uu \in L^{1/s,\lambda}(\R) \quad \mbox{for any $\lambda \in (0,1)$, $s \in (0,1)$.}
\]
By Theorem~\ref{th:adams}, this implies
\[
 \laps{s} \uu \in L^{p,\lambda}(\R)\quad \mbox{for any $\lambda \in (0,1)$, $s \in (0,1)$, $p \in (1,\infty)$.}
\]
Finally, the $C^\mu$-regularity follows from Sobolev embedding from \eqref{eq:Lpest}. Indeed, for any ball $B(x,1)$ one $C^\mu(B(x,1))$-estimate, independent of $x \in \R$. Observe that this immediately implies global $C^\mu$-regularity, since $\uu$ is bounded. 
\end{proof}

\subsection{$\dot{H}^1$-Bound and Local Bootstrap}
So far, in Proposition~\ref{pr:below1-est} we only have estimates of differential order $< 1$.
To jump over the threshold $1$, we observe that $\uu \cdot \uu' = 0$ implies
\[
 \|\uu \cdot \Dx \uu\|_{L^2(\R)} = \|[\Rz, \uu\cdot ](\uu')\|_{L^2(\R)} \aleq [u]_{C^{\frac{1}{2}}} \|\lapv \uu\|_{L^2(\R)}. 
\]
In the last step we used the Coifman-Meyer estimate, see \cite[Theorem 6.1.]{Lenzmann-Schikorra-commutators}. The same is true for
\[
 \|[\Rz, \uu\wedge](\uu')\|_{L^2(\R)} \aleq [u]_{C^{\frac{1}{2}}} \|\lapv \uu\|_{L^2(\R)}
\]
Now we apply Lemma~\ref{la:TfBfest}, observing that in view of \eqref{eq:ourel} we have $T_\uu \Dx \uu \equiv 0$. Then
\begin{equation}\label{eq:H1}
 \|\uu'\|_{L^2(\R)} = \|\Dx \uu\|_{L^2(\R)}\aleq [u]_{C^{\frac{1}{2}}} \|\lapv \uu\|_{L^2(\R)}.
\end{equation}
In particular we have shown $\uu \in \dot{H}^1(\R; \Ss^2)$.

Repeating the above argument instead of $L^2$ with $L^p$, suitably localized, we also have
\[
 \uu',\ \Dx \uu \in L^p_{loc}(\R) \quad \mbox{for any $p \in (1,\infty)$}.
\]

The \emph{local} bootstrap regularity is then a (simpler) iteration of the previous arguments. We repeatedly use Lemma~\ref{la:TfBfest} with $\ff := \laps{s} \uu$.

The resulting terms are all controlled by terms of lower order, one simply localizes the following estimates: 
For $s > 1$, see \cite{SNHarmS10,DndMan},
\[
 \|\uu \cdot \laps{s} \uu\|_{L^2(\R)} \aleq \begin{cases}
                                             \|\laps{\frac{s}{2}} u\|_{L^4(\R)}^2 \quad &\mbox{if $s \leq 1$,}\\
                                             \|\laps{s-1} u\|_{L^4(\R)}^2\ \|\laps{1} u\|_{L^4(\R)} &\mbox{if $s > 1$.}
                                            \end{cases}
\]
With the equation \eqref{eq:ourel}, we have by the Coifman-Meyer estimate, for $s > 1$,
\[
 \|T_{\uu} \laps{s} \uu\|_{L^2(\R)} = \|[\laps{s-1},\uu](\Dx \uu)\|_{L^2(\R)} \aleq 
                                             \|\uu\|_{C^{s-1}}\ \|\laps{s-1} \uu \|_{L^2(\R)}.
\]
Finally, see \cite[Theorem 8.1]{Lenzmann-Schikorra-commutators}, for any $\alpha \in (0,1)$,
\[
 \|[\Rz,\uu\cdot] \laps{s}\uu\|_{L^p(I)} +   \|[\Rz,\uu\wedge] \laps{s} \uu\|_{L^p(I)} \aleq [u]_{C^{0,\alpha}}\ \|\laps{s-\alpha} \uu \|_{L^2(\R)}.
\]
Localizing all these arguments, we obtain for any $\Lambda > 0$ any order $k$ a constant so that for the supremum on intervals of size $\Lambda$,
\begin{equation}\label{eq:nablakuuL^2loc}
 \sup_{I \subset \R, |I| \leq \Lambda} \| \nabla^k \uu \|_{L^2(I)} \aleq C(k,\uu,\Lambda) < \infty.
\end{equation}

\subsection{Smoothness and $\dot{H}^{2}$-Bound}
For any interval $I \subset \R$, it holds that $\|f\|_{L^\infty(I)} \leq |I|^{-1}\, \|f\|_{L^1(I)} + \|f'\|_{L^1(I)}$. From \eqref{eq:nablakuuL^2loc} we thus conclude
\[
\sup_{x \in \R} \|\uu\|_{C^k(B(\rho,x))} \aleq C(k,\uu,\rho) \quad \forall k \in \N,
\]
which implies the global estimates:
\begin{equation}\label{eq:Ck}
 \|\uu\|_{C^k(\R)} \aleq C(k,\uu) \quad \forall k \in \N.
\end{equation}
In particular, we obtain $\uu \in C^\infty(\R)$ is smooth. Next, we derive global $\dot{H}^2$-bounds. By differentiating \eqref{eq:ourel}, we get
\begin{equation}\label{eq:diffourel}
\uu_x \wedge \Dx \uu+ \uu \wedge \Dx \uu_x -v \uu_{xx} = 0 \quad \mbox{in $\R$}.
\end{equation}
We have
\[
 \|\Dx \uu_x\|_{L^2(\R)} \leq \|\uu \cdot \Dx \uu_x\|_{L^2(\R)}+\|\uu \wedge \Dx \uu_x\|_{L^2(\R)}
\]
and in view of \eqref{eq:diffourel},
\[
 \leq \|\uu \cdot \Dx \uu_x\|_{L^2(\R)}+ |v| \|\uu_{xx} \|_{L^2(\R)}+\|\uu_x \wedge \Dx \uu\|_{L^2(\R)}.
\]
Using that $\|\uu_{xx} \|_{L^2(\R)} =  \|\Dx \uu_x\|_{L^2(\R)}$ and absorbing the term involving $|v| < 1$, 
\begin{equation}\label{eq:DxuuxL2estimate}
 \|\Dx \uu_x\|_{L^2(\R)} \aleq \|\uu \cdot \Dx \uu_x\|_{L^2(\R)}+ \|\uu_x \wedge \Dx \uu\|_{L^2(\R)}.
\end{equation}
Next, from $\uu_x \wedge \uu_x = 0$ and $\Rz \uu_x = - \Dx \uu$ we get
\[
 \uu_x \wedge \Dx \uu = -\uu_x \wedge \Rz[\uu_x] = [\Rz, \uu_x \wedge](\uu_x).
\]
Consequently, with the Coifman--Meyer or Kenig--Ponce--Vega type estimates in \cite[Theorem 6.1]{Lenzmann-Schikorra-commutators}
\begin{equation}\label{eq:uuxwedgeDxuu}
 \|\uu_x \wedge \Dx \uu\|_{L^2(\R)} \aleq [\uu_x]_{C^{0,1}}\ \|\Dx \uu\|_{L^2} \aleq \|\uu \|_{C^2(\R)}\ \|\uu \|_{H^1(\R)}.
\end{equation}
Moreover, we use $\uu \cdot \uu_x = 0$ since $\uu \in \S^2$ pointwise a.\,e., which implies
\[
 \uu \cdot \Dx \uu_x = [\Dx,\uu \cdot ] (\uu_x)
\]
Again with the Coifman--Meyer or Kenig--Ponce--Vega type estimates in \cite[Theorem 6.1]{Lenzmann-Schikorra-commutators}
\begin{equation}\label{eq:uuxcdotDxuux}
 \|\uu \cdot \Dx\uu_x \|_{L^2(\R)} \aleq [\uu]_{C^{0,1}}\ \|\uu_x \|_{L^2(\R)} \aleq \|\uu\|_{C^1(\R)}\ \|\uu\|_{\dot{H}^1(\R)}.
\end{equation}
Plugging \eqref{eq:uuxcdotDxuux} and \eqref{eq:uuxwedgeDxuu} into \eqref{eq:DxuuxL2estimate},
\[
 \|\uu\|_{\dot{H}^2(\R)} = \|\Dx \uu_x\|_{L^2(\R)} \aleq  \|\uu \|_{C^2(\R)}\ \|\uu \|_{\dot{H}^1(\R)}.
\]
By the $C^2(\R)$-estimates from \eqref{eq:Ck} and the $\dot{H}^1(\R)$-estimates from \eqref{eq:H1}, we deduce that $\uu \in \dot{H}^2(\R)$. We thus have provided a proof of Theorem \ref{th:regularity}.

\section{Miscellanea}

Let $m \geq 1$ and recall that $M^{(m)}$ denotes the $(2m-1)\times (2m-1)$--matrix defined in \eqref{eq:matrixMn} containing the finite sequences $(\alpha_n)_{n=1}^{2m-1}$ and $(\beta_n)_{n=1}^{2m-2}$ from \eqref{eq:an} and \eqref{eq:bn}.

\begin{lemma} \label{lem:jacobi}
The matrix $M^{(m)}$ has exactly $2m-1$ real and simple eigenvalues with
$$
\lambda_1 < \lambda_2 < \lambda_3 < \ldots < \lambda_{2m-1} 
$$ 
\end{lemma}

\begin{proof}
The result basically follows from the tridiagonal structure of $M^{(m)}$. Let $N = 2m -1$ in what follows. We define $D=[d_{kl}]$ to be the diagonal $N \times N$-matrix whose diagonal entries are given by
\be
d_{kk} = \begin{dcases*} 1 & for $k=1$, \\ \sqrt{ \frac{\beta_{N-1} \beta_{N-2} \cdots \beta_{N-k+1}}{(-\beta_1)(- \beta_2) \cdots (-\beta_{k-1})}} & for $2 \leq k \leq N$. \end{dcases*}
\ee
Note that $\frac{\beta_{N-1} \beta_{N-2} \cdots \beta_{N-k+1}}{(-\beta_1) (-\beta_2) \cdots (-\beta_{k-1})} \neq 0$ is a real number. In particular, the matrix $D$ is invertible and a calculation yields that
\be
T := D M^{(n)} D^{-1} = \left [ \begin{array}{cccc}
\alpha_{1} & \gamma_1 & & 0\\
\gamma_1 & \ddots & \ddots & \\
& \ddots & \ddots & \gamma_{2m-2} \\
0 & & \gamma_{2m-2} & \alpha_{2m-1} \end{array} \right ] 
\ee
where $\gamma_{n} = (-1)^{n+1} \sqrt{ |\beta_n| |\beta_{N-n}|}  \neq 0$ are real non-zero entries.  Since the $\alpha_n$ are also real numbers, the matrix $T$ is real and symmetric. Hence $T$ and $M^{(m)}$ have the same $N=2m-1$ real eigenvalues $\{ \lambda_{k} \}_{k=1}^{N}$ counting multiplicities. 

It is now standard fact that the tridiagonal structure of $T$ yields that each eigenvalue $\lambda_k$ must be simple. Indeed, suppose that $T \mathbf{v} = \lambda \mathbf{v}$ for some $\mathbf{v} \in \R^N$ and $\lambda \in \R$. From the tridiagonal form of $T$ and $\gamma_n \neq 0$, we readily see that $v_1 = 0$ implies that $v_k = 0$ for all $2 \leq k \leq N$. Therefore if the dimension of some eigenspace of $T$ was greater than one, we could construct an eigenvector $\mathbf{v} \in \R^N$ of $T$ with $v_1=0$ and hence $\mathbf{v}=0$, which is a contradiction.
\end{proof}

\begin{lemma} \label{lem:dense_ran}
Suppose $f \in H^1(\R)$ with $P^\perp f = f$ and let $\eps > 0$. Then there is a function $f_\eps \in H^1(\R)$ with $\| f - f _\eps \|_{L^2} < \eps$ such that $P^\perp f_\eps = f_\eps$, $f_\eps \in \mathrm{ran} \, (L_+)$, and $|f_\eps(x)| \leq C \langle x \rangle^{-2}$ for all $x \in \R$ with some constant $C=C(f,\eps)$.
\end{lemma}

\begin{proof}
Let $f \in H^1(\R)$ with $P^\perp f = f$ and $\eps > 0$ be given. Since $\ker (L_+) \subset \mathrm{ran}(P)$ and by the self-adjointness of $L_+$, we find $f \in \ker (L_+)^\perp = \overline{\mathrm{ran} (L_+)}$. Thus there is some $g_\eps \in H^1(\R)$ such that $\| f- L_+ g_\eps  \|_{L^2} < \eps/2$. Furthermore, by a standard density argument, we can find some $h_\eps \in C^\infty_c(\R)$ such that $\| L_+ g - L_+ h_\eps \|_{L^2} < \eps/2$ and hence $\| f- L_+ h_\eps \| < \eps$. Let us now define the function  $f_\eps := L_+ P^\perp h_\eps \in \mathrm{ran}\, (L_+)$. Clearly we obtain $f_\eps \in H^1(\R)$, since $L_+ : H^2(\R) \to H^1(\R)$ and $P^\perp = \mathds{1} - P$, where $P$ projects onto a finite linear combination of smooth and bounded functions. Because $L_+$ and $P^\perp$ commute and $\| P^\perp \|_{L^2 \to L^2} = 1$, we also find that $\| f - f_\eps \|_{L^2} = \| P^\perp (f - L_+ h_\eps) \|_{L^2} < \eps$.

Finally, we prove the pointwise estimate $|f_\eps(x)| \leq C \langle x \rangle^{-2}$ for $x \in \R$ with some constant $C=C(f,\eps)$. Indeed, we note that 
$$
f_\eps = L_+ P^\perp h_\eps = L_+ h_\eps - \sum_{k=0}^{2m-2} c_k \vphi_k,
$$
with the coefficients $c_k = ( \vphi_k, h_\eps)$, where $\{ \vphi_0, \ldots, \vphi_{2m-2} \}$ denote the (normalized) $L^2$-eigenfunctions of $L_+$ corresponding to the eigenvalues $E_0< E_1 < \ldots E_{2m-2}$. Note here that since $L_+ \vphi_{2m-1}= 0$, the zero $L^2$-eigenfunction of $L_+$ with eigenvalue $E_{2m-1}=0$ does not appear in the sum over $k$ above. By well-known methods \cite{CaMaSi-90}, the eigenfunctions $\vphi_{k}$ of $L_+$ for $E_k < 0$ satisfy the pointwise bound $|\vphi_k(x)| \lesssim \langle x \rangle^{-2}$. Moreover, it is elementary to checkt that $h_\eps \in C^\infty_c(\R)$ implies that $|(L_+ h_\eps)(x)| \leq C \langle x \rangle^{-2}$ for some constant $C=C(h_\eps)$. Hence we conclude the asserted decay estimate for the function $f_\eps$.
\end{proof}

\end{appendix}


\def\cprime{$'$}

\end{document}